\documentclass[12pt]{amsart}
\usepackage{amssymb,amsmath,tabularx,mathrsfs,mathbbol,yfonts,upgreek}
\usepackage{amsthm,verbatim,comment}
\usepackage[bookmarks=true]{hyperref}
\usepackage{pstricks,pst-node,pst-plot}
\usepackage{geometry}
\usepackage{stmaryrd}
\usepackage{paralist}
\usepackage[all]{xy}
\usepackage{array}
\usepackage{url}
\usepackage{bbm}
\usepackage{dsfont}
\usepackage{longtable}
\numberwithin{equation}{section}

\DeclareSymbolFontAlphabet{\mathbb}{AMSb}
\DeclareSymbolFontAlphabet{\mathbbl}{bbold}

\newtheorem{thm}{Theorem}[section]

\newtheorem{lem}[thm]{Lemma}

\newtheorem{cor}[thm]{Corollary}

\theoremstyle{definition}

\newtheorem{eg}[thm]{Example}

\newtheorem*{rem*}{Remarks}

\newtheoremstyle{case}{}{}{}{}{}{:}{ }{}
\theoremstyle{case}

\newcommand{\F}{\mathbb{F}}

\title[On Terwilliger $\F$-algebras of factorial association schemes II]{On Terwilliger $\F$-algebras of factorial association schemes II}
\begin{document}
\author{Yu Jiang}
\address[Y. Jiang]{School of Mathematical Sciences, Anhui University (Qingyuan Campus), No. 111, Jiulong Road, Hefei, 230601, China}
\email[Y. Jiang]{jiangyu@ahu.edu.cn}
\begin{abstract}
The Terwilliger algebras of association schemes over an arbitrary field $\F$ were called the Terwilliger $\F$-algebras of association schemes in \cite{J2}. In \cite{HJ}, He and Jiang studied the Terwilliger $\F$-algebras of factorial association schemes. In this paper, we continue studying the Terwilliger $\F$-algebras of factorial association schemes. We get all block idempotents of the Terwilliger $\F$-algebras of factorial association schemes. We get the $\F$-dimensions, the centers, the Jacobson radicals of the block algebras of the Terwilliger $\F$-algebras of factorial association schemes.
\end{abstract}
\maketitle
\noindent
\textbf{Keywords.} {Terwilliger $\F$-algebra; Block idempotent; $\F$-Dimension; Center; Radical}\\
\textbf{Mathematics Subject Classification 2020.} 05E30 (primary), 05E16 (secondary)
\vspace{-1.5em}
\section{Introduction}
Association schemes on nonempty finite sets, briefly called schemes, are important objects in algebraic combinatorics. Many different tools are used to study schemes.

The subconstituent algebras of commutative schemes, introduced by Terwilliger in \cite{T1}, are new tools for studying schemes. They are finite-dimensional semisimple associative $\mathbb{C}$-algebras and are now known as the Terwilliger algebras of commutative schemes. In \cite{Han}, Hanaki showed that the Terwilliger algebras can be defined for an arbitrary scheme and an arbitrary commutative unital ring. In \cite{J2}, the Terwilliger algebras of schemes over a field $\F$ were called the Terwilliger $\F$-algebras of schemes. So the Terwilliger algebras of commutative schemes are their Terwilliger $\mathbb{C}$-algebras.

The Terwilliger $\mathbb{C}$-algebras of many commutative schemes have been investigated (for example, see \cite{BST, CD, LMP, LM, LMW, M, T1, T2, T3, TY}). In general, the study of the Terwilliger $\F$-algebras of schemes is not sufficient (see \cite{Her}). In \cite{BST}, the authors suggested that studying the Terwilliger $\mathbb{C}$-algebras of factorial schemes is interesting. In \cite{HJ}, He and Jiang studied the Terwilliger $\F$-algebras of factorial schemes. In this paper, we continue studying the Terwilliger $\F$-algebras of factorial schemes. Our main results are as follows: We get all block idempotents of the Terwilliger $\F$-algebras of factorial schemes (see Theorem \ref{T;Idempotent}). We get the $\F$-dimensions, the centers, the Jacobson radicals of the block algebras of the Terwilliger $\F$-algebras of factorial schemes (see Theorems \ref{T;Dimension}, \ref{T;Center}, \ref{T;Jacobson}, respectively). These results may contribute to studying the algebraic structures of the Terwilliger $\F$-algebras of symmetric schemes.

The organization of this paper is as follows: In Section 2, we list the basic notation and the preliminaries of this paper. In Section 3, we complete the proof of Theorem \ref{T;Idempotent}. We present the proofs of Theorems \ref{T;Dimension}, \ref{T;Center}, \ref{T;Jacobson} in Sections 4, 5, 6, respectively.
\section{Basic notation and preliminaries}
In this section, we list the basic notation and the preliminaries of this paper. For a background on the theory of association schemes, the reader may refer to \cite{B, Z}.
\subsection{Conventions}
Let $\mathbb{N}_0$ be the set of nonnegative integers. For any $g, h\in\mathbb{N}_0$, let $[g, h]=\{a: a\in\mathbb{N}_0, g\leq a\leq h\}$ and $[g,\infty)=\{a: a\in\mathbb{N}_0, g\leq a\}$. For any $g\!\in\![1,\infty)$ and sets $\mathbb{U}_1, \mathbb{U}_2,\ldots, \mathbb{U}_g$, let $\prod_{h=1}^g\mathbb{U}_h$ be the cartesian product $\mathbb{U}_1\times \mathbb{U}_2\times\cdots\times\mathbb{U}_g$ and $\mathbf{i}_j$ be the $j$th-entry of an element $\mathbf{i}$ in $\prod_{h=1}^g\mathbb{U}_h$ for any $j\in[1,g]$. An association scheme on a nonempty finite set is briefly called a scheme. Fix a field $\F$ of characteristic $p$. Let $\delta_{g, h}$ be the Kronecker delta of arbitrary symbols $g$ and $h$ whose values are in $\F$. The addition, the multiplication, and the scalar multiplication displayed in this paper are the usual $\F$-matrix operations. Let $\F_p$ be the prime subfield of $\F$. Let $\mathbb{Z}$ be the ring of integers. For any $g\in\mathbb{Z}$, let $\overline{g}$ be the image of $g$ under the unique unital ring homomorphism from $\mathbb{Z}$ to $\F_p$. For any subset $\mathbb{U}$ of an $\F$-linear space $\mathbb{V}$, let $\langle \mathbb{U}\rangle_\mathbb{V}$ be the $\F$-linear subspace of $\mathbb{V}$ spanned by $\mathbb{U}$. All $\F$-algebras displayed in this paper shall be finite-dimensional associative unital $\F$-algebras. All modules displayed in this paper shall be finite-dimensional left modules of a given $\F$-algebra.
\subsection{Schemes}
Let $\mathbb{X}$, $\mathbb{E}$ be nonempty finite sets. Let $\{R_a: a\!\in\!\mathbb{E}\}$ be a partition of $\mathbb{X}\times\mathbb{X}$. Call $\mathfrak{S}\!=\!(\mathbb{X}, \{R_a: a\!\in\!\mathbb{E}\})$ a {\em scheme} if the following conditions hold together:
\begin{enumerate}[(S1)]
\item There is a unique $0_{\mathfrak{S}}\in\mathbb{E}$ such that $R_{0_{\mathfrak{S}}}$ equals the diagonal set $\{(a,a): a\!\in\!\mathbb{X}\}$.
\item For any $g\!\in\!\mathbb{E}$, there is a unique $g^*\!\in\!\mathbb{E}$ such that $R_{g^*}$ equals $\{(a, b): (b,a)\in R_g\}$.
\item For any $g, h, i\in\mathbb{E}$, $(x,y), (\widetilde{x},\widetilde{y})\!\in\! R_h$, there is a constant $p_{g,i}^h\in[0,\infty)$ such that
$$p_{g,i}^h=|\{a: (x,a)\in R_g, (a,y)\in R_i\}|=|\{a: (\widetilde{x},a)\in R_g, (a,\widetilde{y})\in R_i\}|.$$
\end{enumerate}

From now on, let $\mathfrak{S}=(\mathbb{X}, \{R_a: a\in\mathbb{E}\})$ be a fixed scheme. Call $\mathfrak{S}$ a {\em symmetric scheme} if $g^*=g$ for any $g\in\mathbb{E}$. Call $\mathfrak{S}$ a {\em commutative scheme}
if $p_{g,i}^h=p_{i,g}^h$ for any $g, h, i\in\mathbb{E}$. According to (S2) and (S3), notice that the symmetric schemes are commutative schemes. Let $x R_g=\{a: (x,a)\in R_g\}$ and $k_g=p_{g,g^*}^{0_{\mathfrak{S}}}$ for any $x\in\mathbb{X}$ and $g\in\mathbb{E}$. For any $x, y\in\mathbb{X}$ and $g\in\mathbb{E}$, notice that $k_g=|xR_g|=|yR_g|\in[1,\infty)$ by (S2) and (S3). For any $g\!\in\![0,\infty)$, call $\mathfrak{S}$ a {\em $g'$-valenced scheme} if $g\!\nmid\! k_h$ for any $h\in\mathbb{E}$.

Let $g\in[1,\infty)$. Let $\mathbb{X}_h$ and $\mathbb{E}_h$ be nonempty finite sets for any $h\in[1, g]$. For any $h\in[1, g]$,
let $\{R_{a, h}: a\in\mathbb{E}_h\}$ be a partition of $\mathbb{X}_h\times\mathbb{X}_h$. For any $h\in[1, g]$, let $\mathfrak{S}_h=(\mathbb{X}_h, \{R_{a, h}: a\in\mathbb{E}_h\})$ be a scheme. For any $\mathbf{x}, \mathbf{y}\in\prod_{h=1}^g\mathbb{X}_h$ and $\mathbf{i}\in\prod_{h=1}^g\mathbb{E}_h$, let $\mathbf{x}=_\mathbf{i}\mathbf{y}$ if $(\mathbf{x}_h, \mathbf{y}_h)\!\in\!R_{\mathbf{i}_h, h}$ for any $h\!\in\![1, g]$. Let $R_\mathbf{i}\!=\!\{(\mathbf{a}, \mathbf{b})\!:\! \mathbf{a}, \mathbf{b}\!\in\!\prod_{h=1}^g\mathbb{X}_h, \mathbf{a}\!=_\mathbf{i}\!\mathbf{b}\}$ for any $\mathbf{i}\in\prod_{h=1}^g\mathbb{E}_h$. Recall that $(\prod_{h=1}^g\mathbb{X}_h, \{R_\mathbf{a}: \mathbf{a}\in\prod_{h=1}^g\mathbb{E}_h\})$ is also a defined scheme (see \cite{B, Z}). Call $\mathfrak{S}$ the {\em direct product} of $\mathfrak{S}_1, \mathfrak{S}_2, \ldots, \mathfrak{S}_g$ if the equations $\mathbb{X}=\prod_{h=1}^g\mathbb{X}_h$, $\mathbb{E}=\prod_{h=1}^g\mathbb{E}_h$, and $\{R_a: a\in\mathbb{E}\}=\{R_\mathbf{a}: \mathbf{a}\in\prod_{h=1}^g\mathbb{E}_h\}$ hold together. If $\mathfrak{S}$ is the direct product of $\mathfrak{S}_1, \mathfrak{S}_2, \ldots, \mathfrak{S}_g$, notice that   $0_\mathfrak{S}=(0_{\mathfrak{S}_1}, 0_{\mathfrak{S}_2},\ldots, 0_{\mathfrak{S}_g})$. If $\mathfrak{S}$ is the direct product of $\mathfrak{S}_1, \mathfrak{S}_2, \ldots, \mathfrak{S}_g$, notice that $\mathfrak{S}$ is a symmetric scheme if and only if $\mathfrak{S}_1, \mathfrak{S}_2,\ldots, \mathfrak{S}_g$ are symmetric schemes. The following lemma is required.
\begin{lem}\label{L;Lemma2.1}\cite[Theorem 2.6.1 (iv)]{Z}
Assume that $g\in[1,\infty)$. Assume that $\mathfrak{S}$ is the direct product of the schemes $\mathfrak{S}_1, \mathfrak{S}_2, \ldots, \mathfrak{S}_g$ and $\mathbf{h}, \mathbf{i}, \mathbf{j}\in\mathbb{E}$. Then $p_{\mathbf{h},\mathbf{j}}^\mathbf{i}=\prod_{k=1}^g p_{\mathbf{h}_k,\mathbf{j}_k}^{\mathbf{i}_k}$.
\end{lem}
Let $\mathbb{U}$ be a set and $|\mathbb{U}|\in[2,\infty)$. Let $R(\mathbb{U})_0$ and $R(\mathbb{U})_1$ be $\{(a, a): a\in\mathbb{U}\}$ and $\{(a, b): a, b\in\mathbb{U}, a\neq b\}$, respectively. This implies that $\{R(\mathbb{U})_0, R(\mathbb{U})_1\}$ is a partition of $\mathbb{U}\times\mathbb{U}$. Moreover, recall that $(\mathbb{U}, \{R(\mathbb{U})_0, R(\mathbb{U})_1\})$ is a defined scheme (see \cite[Example 1.2]{B}). Call $\mathfrak{S}$ the {\em trivial scheme} induced from $\mathbb{U}$ if the equations $\mathbb{X}=\mathbb{U}$, $\mathbb{E}=[0,1]$, and $\{R_a: a\in\mathbb{E}\}=\{R(\mathbb{U})_0, R(\mathbb{U})_1\}$ hold together. If $\mathfrak{S}$ is the trivial scheme induced from $\mathbb{U}$, notice that $0_\mathfrak{S}\!=\!0$. The following lemma is required.
\begin{lem}\label{L;Lemma2.2}\cite[Example 1.2]{B}
Assume that $\mathbb{U}$ is a set, $|\mathbb{U}|\!\in\![2,\infty)$, $\mathfrak{S}$ is the trivial scheme induced from $\mathbb{U}$,  $\mathbb{V}\!=\!\{(0, 0, 0), (0,1, 1), (1,0, 1), (1,1, 0), (1,1,1)\}$. Then, for any $g, h, i\in\mathbb{E}$, $p_{g, i}^h\neq 0$ only if $(g, h, i)\in\mathbb{V}$. Moreover, $p_{0,0}^0=p_{0,1}^1=p_{1,0}^1=1$, $p_{1,1}^0\!=\!|\mathbb{U}|\!-\!1$, $p_{1,1}^1\!=\!|\mathbb{U}|-2$. In particular, $\{(a, b, c)\!: a, b, c\!\in\![0,1], \ p_{a, c}^b\!\neq\! 0\}$ is equal to
\[\begin{cases}
\mathbb{V}\setminus\{(1,1,1)\}, &\text{if $|\mathbb{U}|=2$},\\
\mathbb{V}, &\text{if $|\mathbb{U}|\neq2$}.
\end{cases}\]
\end{lem}
Fix $n\in[1,\infty)$ and a set sequence $\mathbb{U}_1, \mathbb{U}_2,\ldots, \mathbb{U}_n$ whose cardinalities are at least two.
For any $\mathbb{V}\subseteq[1,n]$, let $\mathbb{V}^\circ=\{a: a\in\mathbb{V}, |\mathbb{U}_a|\neq 2\}$. For any $\mathbb{V}\subseteq[1,n]$, let $\mathbb{V}^+=\{a: a\in\mathbb{V}, p\mid |\mathbb{U}_a|-1\}$ and $\mathbb{V}^-=\{a: a\in\mathbb{V}, |\mathbb{U}_a|\neq 2, p\nmid |\mathbb{U}_a|-1\}$. In particular, let $n_+=|[1,n]^+|$ and $n_-=|[1,n]^-|$. So $n_++n_-=|[1,n]^\circ|$.
For any sets $\mathbb{V}$ and $\mathbb{W}$, let
$\mathbb{V}\triangle\mathbb{W}\!=\!(\mathbb{V}\setminus\mathbb{W})\!\cup\!(\mathbb{W}\setminus\mathbb{V})$.
Let $[0,1]^n$ be the cartesian product of $n$ copies of $[0,1]$. So each element in $[0,1]^n$ is a $n$-tuple whose entries are in $[0,1]$. For any $\mathbf{g}\in[0,1]^n$, let $\mathbbm{g}=\{a: a\in[1, n],\mathbf{g}_a=1\}$. For any $\mathbf{g}, \mathbf{h}\in[0,1]^n$, let $\mathbf{g}\preceq\mathbf{h}$ if $\mathbbm{g}\subseteq\mathbbm{h}$. Let $\mathbf{0}$ be the $n$-tuple with all-zero entries. Let $\mathbf{1}$ be the $n$-tuple with all-one entries. So $\mathbf{0}\preceq\!\mathbf{g}\!\preceq\mathbf{1}$ for any $\mathbf{g}\!\in\![0,1]^n$. The proof of the following lemma is trivial.
\begin{lem}\label{L;Lemma2.3}
Assume that $\mathbf{g}, \mathbf{h}\in[0,1]^n$. Then $\mathbf{g}=\mathbf{h}$ if and only if $\mathbbm{g}\!=\!\mathbbm{h}$. In particular, $\preceq$ is a partial order on $[0,1]^n$ with the minimal element $\mathbf{0}$ and the maximal element $\mathbf{1}$. Moreover, for any $\mathbb{U}\subseteq [1,n]$, there is a unique $\mathbf{i}\in[0,1]^n$ such that $\mathbbm{i}=\mathbb{U}$.
\end{lem}
Let $\mathfrak{Tri}(\mathbb{U}_1), \mathfrak{Tri}(\mathbb{U}_2), \ldots, \mathfrak{Tri}(\mathbb{U}_n)$ be the corresponding trivial schemes induced from $\mathbb{U}_1, \mathbb{U}_2, \ldots, \mathbb{U}_n$. Call $\mathfrak{S}$ the {\em factorial scheme} of $\mathfrak{Tri}(\mathbb{U}_1), \mathfrak{Tri}(\mathbb{U}_2), \ldots, \mathfrak{Tri}(\mathbb{U}_n)$ if $\mathfrak{S}$ is the direct product of $\mathfrak{Tri}(\mathbb{U}_1), \mathfrak{Tri}(\mathbb{U}_2), \ldots, \mathfrak{Tri}(\mathbb{U}_n)$ (see \cite[Pages 344, 345]{B}). If $\mathfrak{S}$ is the factorial scheme of $\mathfrak{Tri}(\mathbb{U}_1), \mathfrak{Tri}(\mathbb{U}_2), \ldots, \mathfrak{Tri}(\mathbb{U}_n)$, notice that $\mathbb{X}=\prod_{h=1}^n\mathbb{U}_h$, $\mathbb{E}=[0,1]^n$, and $0_\mathfrak{S}=\mathbf{0}$. If $\mathfrak{S}$ is the factorial scheme of $\mathfrak{Tri}(\mathbb{U}_1), \mathfrak{Tri}(\mathbb{U}_2), \ldots, \mathfrak{Tri}(\mathbb{U}_n)$, notice that $\mathfrak{S}$ is a symmetric scheme as $\mathfrak{Tri}(\mathbb{U}_1),\mathfrak{Tri}(\mathbb{U}_2),\ldots, \mathfrak{Tri}(\mathbb{U}_n)$ are symmetric schemes. In particular, if $\mathfrak{S}$ is the factorial scheme of $\mathfrak{Tri}(\mathbb{U}_1), \mathfrak{Tri}(\mathbb{U}_2), \ldots, \mathfrak{Tri}(\mathbb{U}_n)$, notice that $\mathfrak{S}$ is also a commutative scheme. The following two lemmas are required.
\begin{lem}\label{L;Lemma2.4}
Assume that $\mathfrak{S}$ is the factorial scheme of $\mathfrak{Tri}(\mathbb{U}_1), \mathfrak{Tri}(\mathbb{U}_2), \ldots,\! \mathfrak{Tri}(\mathbb{U}_n)$. Assume that $\mathbf{g}, \mathbf{h}, \mathbf{i}\in\mathbb{E}$. Then $p_{\mathbf{g}, \mathbf{i}}^\mathbf{h}\neq 0$ if and only if $\mathbbm{g}\triangle\mathbbm{i}\subseteq\mathbbm{h}\subseteq(\mathbbm{g}\triangle\mathbbm{i})\cup(\mathbbm{g}\cap\mathbbm{i})^\circ$.
\end{lem}
\begin{proof}
The desired lemma follows from the combination of Lemmas \ref{L;Lemma2.1}, \ref{L;Lemma2.2}, \ref{L;Lemma2.3}.
\end{proof}
\begin{lem}\label{L;Lemma2.5}
Assume that $\mathfrak{S}$ is the factorial scheme of $\mathfrak{Tri}(\mathbb{U}_1), \mathfrak{Tri}(\mathbb{U}_2), \ldots, \mathfrak{Tri}(\mathbb{U}_n)$ and $\mathbf{g}\in\mathbb{E}$. Then $k_\mathbf{g}=\prod_{h\in\mathbbm{g}}(|\mathbb{U}_h|-1)$, where the product over $\varnothing$ is equal to one.
\end{lem}
\begin{proof}
The desired lemma follows from an application of Lemmas \ref{L;Lemma2.1} and \ref{L;Lemma2.2}.
\end{proof}
\subsection{Algebras}
Let $\mathbb{A}$ be an $\F$-algebra with the zero element $0_\mathbb{A}$ and the identity element $1_\mathbb{A}$. Let $e, f\in\mathbb{A}$. Call $e$ an {\em idempotent} of $\mathbb{A}$ if $e^2=e$. Call $e$ and $f$ a pair of {\em orthogonal idempotents} of $\mathbb{A}$ if $e, f$ are idempotents of $\mathbb{A}$ and $ef=fe=0_\mathbb{A}$. Call $e$ a {\em primitive idempotent} of $\mathbb{A}$ if $e$ is a nonzero idempotent of $\mathbb{A}$ and $e$ is not a sum of pairwise orthogonal nonzero idempotents of $\mathbb{A}$. Call the unital $\F$-subalgebra of $\mathbb{A}$ generated by $\{a\!:\! a\!\in\! \mathbb{A}, ab\!=\!ba\ \forall\ b\!\in\!\mathbb{A}\}$ the {\em center} of $\mathbb{A}$. Let $\mathrm{Z}(\mathbb{A})$ be the center of $\mathbb{A}$. Call $e$ a {\em block idempotent} of $\mathbb{A}$ if $e$ is a primitive idempotent of $\mathrm{Z}(\mathbb{A})$. Let $\mathrm{Bl}(\mathbb{A})$ be the set of the block idempotents of $\mathbb{A}$. Notice that the distinct elements in $\mathrm{Bl}(\mathbb{A})$ are pairwise orthogonal idempotents of $\mathbb{A}$ and $1_\mathbb{A}$ is the sum of all elements in $\mathrm{Bl}(\mathbb{A})$.

Let $\mathbb{U}\mathbb{V}=\langle\{ab: a\in\mathbb{U}, b\in\mathbb{V}\}\rangle_\mathbb{A}$ and $\mathbb{U}e=\mathbb{U}\{e\}$ for any subsets $\mathbb{U}, \mathbb{V}$ of $\mathbb{A}$. If $e$ is an idempotent of $\mathrm{Z}(\mathbb{A})$, notice that $\mathbb{A}e$ is an $\F$-subalgebra of $\mathbb{A}$ with the identity element $e$ and the $\F$-subalgebra $\mathbb{A}e$ of $\mathbb{A}$ satisfies the equation $\mathrm{Z}(\mathbb{A}e)=\mathrm{Z}(\mathbb{A})e$. In particular, for the case $e\!\in\!\mathrm{Bl}(\mathbb{A})$, call the $\F$-subalgebra $\mathbb{A}e$ of $\mathbb{A}$ a {\em block algebra} of $\mathbb{A}$.

For any $\F$-linear subspace $\mathbb{U}$ of a given $\F$-linear space $\mathbb{V}$, let $\mathbb{V}/\mathbb{U}$ be the quotient $\F$-linear space of $\mathbb{V}$ with respect to $\mathbb{U}$. Let $\mathbb{I}$ be a two-sided ideal of $\mathbb{A}$. Let $\mathbb{A}/\mathbb{I}$ be the quotient $\F$-algebra of $\mathbb{A}$ with respect to $\mathbb{I}$. Call $\mathbb{I}$ an {\em indecomposable two-sided ideal} of $\mathbb{A}$ if $\mathbb{I}\neq\{0_\mathbb{A}\}$ and $\mathbb{I}$ is not a direct sum of two nonzero two-sided ideals of $\mathbb{A}$. So $\mathbb{A}$ is a direct sum of indecomposable two-sided ideals of $\mathbb{A}$. If $e\in\mathrm{Bl}(\mathbb{A})$, notice that the block algebra $\mathbb{A}e$ of $\mathbb{A}$ is an indecomposable two-sided ideal of $\mathbb{A}$. So $\mathbb{A}$ is the direct sum of all pairwise distinct block algebras of $\mathbb{A}$. Call $\mathbb{I}$ a {\em minimal two-sided ideal} of $\mathbb{A}$ if $\mathbb{I}\neq\{0_\mathbb{A}\}$ and no nonzero two-sided ideal of $\mathbb{A}$ is properly contained in $\mathbb{I}$. Then the minimal two-sided ideals of $\mathbb{A}$ are indecomposable two-sided ideals of $\mathbb{A}$. In general, the indecomposable two-sided ideals of $\mathbb{A}$ are not necessary to be minimal two-sided ideals of $\mathbb{A}$. Call $\mathbb{A}$ a {\em semisimple $\F$-algebra} if all indecomposable two-sided ideals of $\mathbb{A}$ are minimal two-sided ideals of $\mathbb{A}$. Call $\mathbb{I}$ a {\em nilpotent two-sided ideal} of $\mathbb{A}$ if there is $g\in[1,\infty)$ such that $h_1h_2\cdots h_g=0_\mathbb{A}$ for any $h_1, h_2, \ldots, h_g\in\mathbb{I}$. If $\mathbb{I}$ is a nilpotent two-sided ideal of $\mathbb{A}$ and $g\in[1,\infty)$, call $g$ the {\em nilpotent index} of $\mathbb{I}$ if $g$ is the smallest choice in $[1,\infty)$ that satisfies the equation $h_1h_2\cdots h_g=0_\mathbb{A}$ for any $h_1, h_2, \ldots, h_g\in\mathbb{I}$. The sum of all nilpotent two-sided ideals of $\mathbb{A}$ is called the {\em Jacobson radical} of $\mathbb{A}$. Let $\mathrm{Rad}(\mathbb{A})$ be the Jacobson radical of $\mathbb{A}$. Notice that $\mathrm{Rad}(\mathbb{A})$ is a nilpotent two-sided ideal of $\mathbb{A}$. The following two lemmas are required.
\begin{lem}\label{L;Lemma2.6}\cite[Theorem 3.1.6, Proposition 3.2.4]{DK}
Assume that $\mathbb{A}$ is an $\F$-algebra. Then $\mathbb{A}$ is a semisimple $\F$-algebra if and only if $\mathrm{Rad}(\mathbb{A})=\{0_\mathbb{A}\}$. Moreover, assume that $e$ is an idempotent of $\mathrm{Z}\!({\mathbb{A}})$. Then the  $\F$-subalgebra $\mathbb{A}e$ of $\mathbb{A}$ satisfies the equations
$$\mathrm{Rad}(\mathbb{A}e)=\mathrm{Rad}(\mathbb{A})e\ \text{and}\ \mathrm{Rad}(\mathrm{Z}(\mathbb{A}e))=\mathrm{Rad}(\mathrm{Z}(\mathbb{A}))e.$$
\end{lem}
\begin{lem}\label{L;Lemma2.7}
Assume that $p>0$ and $\mathbb{A}$ is an $\F$-algebra. Assume that $e$ is an idempotent of $\mathbb{A}$ and $f\in\mathrm{Rad}(\mathbb{A})$. Assume that $g\in[1,\infty)$ and $h_1, h_2,\ldots, h_g$ are pairwise orthogonal idempotents of $\mathbb{A}$. Assume that $e$ satisfies the equation $eh_i=h_ie$ for any $i\!\in\![1,g]$. Assume that $e\!\in\!\langle\{f, h_1, h_2,\ldots, h_g\}\rangle_\mathbb{A}$. Then $e\!\in\!\langle\{h_1, h_2,\ldots, h_g\}\rangle_\mathbb{A}$.
\end{lem}
\begin{proof}
Assume that $j$ is a $p$-power and $j$ is no less than the nilpotent index of $\mathrm{Rad}(\mathbb{A})$. This implies that $e-\sum_{k=1}^gc_k h_k=c_ff$ and $(c_ff)^j=0_\mathbb{A}$ for some $c_1, c_2, \ldots, c_g, c_f\in\F$. So the desired lemma follows from the hypotheses and the Binomial Theorem.
\end{proof}
\subsection{Terwilliger $\F$-algebras of schemes}
For any $g\in[1,\infty)$, let $\mathrm{M}_g(\F)$ be the full matrix $\F$-algebra of $(g\times g)$-matrices whose entries are in $\F$. For any $g, h\in[1, \infty)$, let $g\mathrm{M}_h(\F)$ be the direct sum of $g$ copies of $\mathrm{M}_h(\F)$. For any $g, h, i, j\in[1,\infty)$, notice that $g=i$ and $h=j$ if and only if $g\mathrm{M}_h(\F)\cong i\mathrm{M}_j(\F)$ as $\F$-algebras. Let $\mathbb{U}$ be a nonempty finite set. Let $\mathrm{M}_\mathbb{U}(\F)$ be the full matrix $\F$-algebra of square $\F$-matrices whose rows and columns are labeled by the elements in $\mathbb{U}$. So $\mathrm{M}_\mathbb{U}(\F)\cong\mathrm{M}_{|\mathbb{U}|}(\F)$ as $\F$-algebras. For any $x,y\in\mathbb{U}$, let $M_{x,y}(\mathbb{U})$ be the $\{\overline{0}, \overline{1}\}$-matrix in $\mathrm{M}_\mathbb{U}(\F)$ whose unique nonzero entry is the $(x,y)$-entry. Let $O$ be the zero matrix in $\mathrm{M}_\mathbb{X}(\F)$. Let $I$ be the identity matrix in $\mathrm{M}_\mathbb{X}(\F)$. Let $M^T$ be the transpose of $M$ for any $M\in\mathrm{M}_\mathbb{X}(\F)$.

For any $g\in\mathbb{E}$, let $A_g$ be the {\em adjacency $\F$-matrix} of $\mathfrak{S}$ with respect to $R_g$. For any $x\in\mathbb{X}$ and $g\in\mathbb{E}$, let $E_g^*(x)$ be the {\em dual $\F$-idempotent} of $\mathfrak{S}$ with respect to $x$ and $R_g$. For any $x\in\mathbb{X}$ and $g\!\in\!\mathbb{E}$, recall that $A_g$ and $E_g^*(x)$ are defined by the equations
\begin{align*}
A_g=\sum_{(y, z)\in R_g}M_{y, z}(\mathbb{X})\ \text{and}\ E_g^*(x)=\sum_{y\in xR_g}M_{y, y}(\mathbb{X}).
\end{align*}
In particular, if $\mathfrak{S}$ is a symmetric scheme, notice that $A_g^T=A_g$ for any $g\in\mathbb{E}$. In general, for any $x\in\mathbb{X}$ and $g\in\mathbb{E}$, notice that $A_g$ and $E_g^*(x)$
can satisfy the equations
\begin{align}\label{Eq;1}
A_g^T=A_{g^*},\ E_g^*(x)^T=E_g^*(x),
\end{align}
\begin{align}\label{Eq;2}
E_g^*(x)E_h^*(x)=\delta_{g,h}E_g^*(x),
\end{align}
\begin{align}\label{Eq;3}
\text{and}\ A_{0_\mathfrak{S}}=I=\sum_{i\in\mathbb{E}}E_i^*(x).
\end{align}
Call an $\F$-subalgebra of $\mathrm{M}_\mathbb{X}(\F)$ the {\em Terwilliger $\F$-algebra} of $\mathfrak{S}$ with respect to $x$ if it is generated by $\{A_a: a\in\mathbb{E}\}\cup\{E_a^*(x): a\in\mathbb{E}\}$. Let $\mathbb{T}(x)$ be the Terwilliger $\F$-algebra of $\mathfrak{S}$ with respect to $x$. So $\mathbb{T}(x)$ is a unital $\F$-subalgebra of $\mathrm{M}_\mathbb{X}(\F)$ by Equation \eqref{Eq;3}. Equation \eqref{Eq;1} implies that $M\!\in\!\mathbb{T}(x)$ if and only if $M^T\in\mathbb{T}(x)$. In particular, notice that the map sending $M$ to $M^T$ for any $M\in\mathrm{Z}(\mathbb{T}(x))$ is an $\F$-linear isomorphism from $\mathrm{Z}(\mathbb{T}(x))$ to $\mathrm{Z}(\mathbb{T}(x))$. Moreover, notice that the map sending $M$ to $M^T$ for any $M\in\mathrm{Rad}(\mathbb{T}(x))$ is an $\F$-linear isomorphism from $\mathrm{Rad}(\mathbb{T}(x))$ to $\mathrm{Rad}(\mathbb{T}(x))$. For any $g, h, i\in\mathbb{E}$, recall that $E_g^*(x)A_hE_i^*(x)\!\neq\! O$ if and only if $p_{g^*,i}^h\!\neq\!0$ (see \cite[Lemma 3.2]{Han}). So $\mathbb{T}(x)$ has an $\F$-linearly independent subset $\{E_a^*(x)A_bE_c^*(x)\!: a, b, c\!\in\!\mathbb{E}, p_{a^*, c}^b\!\neq \!0\}$ by Equation \eqref{Eq;2}. In general, the algebraic structures of $\mathbb{T}(x)$ and $\mathrm{Rad}(\mathbb{T}(x))$ may depend on $\F$ and the choice of $x$ (see \cite[5.1]{Han}). For some recent progress on the algebraic structures of $\mathbb{T}(x)$ and $\mathrm{Rad}(\mathbb{T}(x))$, the reader may refer to \cite{CX, Han, HJ, J1, J2, J3}.

In \cite{HJ}, He and Jiang studied the Terwilliger $\F$-algebras of factorial schemes. In this paper, we continue studying the Terwilliger $\F$-algebras of factorial schemes. For this purpose, we need to introduce the necessary notation in \cite{HJ}. From now on, let $\mathfrak{S}$ be the factorial scheme of $\mathfrak{Tri}(\mathbb{U}_1), \mathfrak{Tri}(\mathbb{U}_2), \ldots, \mathfrak{Tri}(\mathbb{U}_n)$. Recall that $\mathbb{E}=[0,1]^n$. We shall quote the fact that $\mathfrak{S}$ is a symmetric scheme without citation. We shall also quote the fact that $\mathfrak{S}$ is a commutative scheme without citation. From now on, fix $\mathbf{x}\in\mathbb{X}$. For convenience, we abbreviate $\mathbb{T}=\mathbb{T}(\mathbf{x})$ and $E_\mathbf{g}^*=E_\mathbf{g}^*(\mathbf{x})$ for any $\mathbf{g}\in \mathbb{E}$.

Let $\mathbb{P}=\{(\mathbf{a}, \mathbf{b}, \mathbf{c}): \mathbf{a}, \mathbf{b}, \mathbf{c}\in\mathbb{E}, \mathbbm{a}\triangle\mathbbm{c}\subseteq\mathbbm{b}\subseteq(\mathbbm{a}\triangle\mathbbm{c})\cup(\mathbbm{a}\cap\mathbbm{c})^\circ\}$ and $\mathbf{g}, \mathbf{h}, \mathbf{i}\in\mathbb{E}$. By Lemma \ref{L;Lemma2.3}, $(\mathbf{g}, \mathbf{h}, \mathbf{i})\in\mathbb{P}$ if and only if $(\mathbf{i}, \mathbf{h}, \mathbf{g})\in\mathbb{P}$. Lemmas \ref{L;Lemma2.3} and \ref{L;Lemma2.4} imply that $(\mathbf{g}, \mathbf{h}, \mathbf{i})\in\mathbb{P}$ if and only if $p_{\mathbf{g},\mathbf{i}}^\mathbf{h}\neq 0$. Let $\mathbb{P}_{\mathbf{g}, \mathbf{h}, \mathbf{i}}\!=\!\{\mathbf{a}:
\mathbf{a}\in\mathbb{E},\mathbbm{g}\triangle\mathbbm{i}\subseteq\mathbbm{a}\subseteq\mathbbm{h}\}$. So $\mathbb{P}_{\mathbf{g}, \mathbf{h}, \mathbf{i}}=\mathbb{P}_{\mathbf{i},\mathbf{h},\mathbf{g}}$. If $(\mathbf{g}, \mathbf{h}, \mathbf{i})\in\mathbb{P}$, Lemma \ref{L;Lemma2.4} implies that $\mathbf{h}\in\mathbb{P}_{\mathbf{g}, \mathbf{h}, \mathbf{i}}$. If $(\mathbf{g}, \mathbf{h}, \mathbf{i})\in\mathbb{P}$, let
\begin{align}\label{Eq;4}
B_{\mathbf{g}, \mathbf{h}, \mathbf{i}}=\sum_{\mathbf{j}\in\mathbb{P}_{\mathbf{g}, \mathbf{h},\mathbf{i}}}E_\mathbf{g}^*A_\mathbf{j}E_\mathbf{i}^*
\end{align}
and notice that $B_{\mathbf{g}, \mathbf{h}, \mathbf{i}}, B_{\mathbf{i}, \mathbf{h}, \mathbf{g}}$ are defined matrices in $\mathbb{T}$. Notice that $\mathbb{P}_{\mathbf{g}, \mathbf{0}, \mathbf{g}}\!\!=\!\!\{\mathbf{0}\}$ and
$B_{\mathbf{g},\mathbf{0},\mathbf{g}}\!\!=\!\!E_\mathbf{g}^*$ by combining Lemma \ref{L;Lemma2.3}, Equations \eqref{Eq;4}, \eqref{Eq;3}, \eqref{Eq;2}. By Lemma \ref{L;Lemma2.3}, let $\mathbf{g}^+$ be the element in $\mathbb{E}$ that satisfies the equation $\{a: a\in[1,n], \mathbf{g}^+_a\!=\!1\}\!=\!\mathbbm{g}^+$. If $\mathbbm{i}=\mathbbm{g}\cup\mathbbm{h}$, let $\mathbf{g}\cup\mathbf{h}=\mathbf{i}$ by Lemma \ref{L;Lemma2.3}.
If $\mathbbm{i}=\mathbbm{g}\cap\mathbbm{h}$, let $\mathbf{g}\cap\mathbf{h}=\mathbf{i}$ by Lemma \ref{L;Lemma2.3}. So $\mathbf{g}\cup\mathbf{h}=\mathbf{h}\cup\mathbf{g}$ and
$\mathbf{g}\cap\mathbf{h}=\mathbf{h}\cap\mathbf{g}$ by Lemma \ref{L;Lemma2.3}. As $(\mathbf{g}\cap\mathbf{h})\cap\mathbf{i}=\mathbf{g}\cap(\mathbf{h}\cap\mathbf{i})$ by Lemma \ref{L;Lemma2.3}, let $\mathbf{g}\cap\mathbf{h}\cap\mathbf{i}=(\mathbf{g}\cap\mathbf{h})\cap\mathbf{i}=\mathbf{g}\cap(\mathbf{h}\cap\mathbf{i})$. If $\mathbbm{i}=\mathbbm{g}\setminus\mathbbm{h}$, let $\mathbf{g}\setminus\mathbf{h}=\mathbf{i}$ by Lemma \ref{L;Lemma2.3}. If $(\mathbf{g}, \mathbf{h}, \mathbf{g})\in\mathbb{P}$, then $(\mathbf{i}, \mathbf{h}\cap\mathbf{i},\mathbf{i})\in\mathbb{P}$ by Lemma \ref{L;Lemma2.3}. If $(\mathbf{1}, \mathbf{g}, \mathbf{1})\in\mathbb{P}$, let
\begin{align}\label{Eq;5}
C_\mathbf{g}=\sum_{\mathbf{j}\in\mathbb{E}}\overline{k_{\mathbf{g}\setminus\mathbf{j}}}B_{\mathbf{j},\mathbf{g}\cap\mathbf{j},\mathbf{j}}
\end{align}
and notice that $C_\mathbf{g}$ is a defined matrix in $\mathbb{T}$. As $\mathbf{0}\setminus\mathbf{g}=\mathbf{0}\cap\mathbf{g}=\mathbf{0}$ by Lemma \ref{L;Lemma2.3}, notice that $C_\mathbf{0}=I$ by
combining Lemma \ref{L;Lemma2.5}, Equations \eqref{Eq;5}, \eqref{Eq;3}. If $\mathbbm{i}=\mathbbm{g}\triangle\mathbbm{h}$, let $\mathbf{g}\triangle\mathbf{h}=\mathbf{i}$ by Lemma \ref{L;Lemma2.3}. If $\mathbbm{i}=(\mathbbm{g}\triangle\mathbbm{h})\cup(\mathbbm{g}\cap\mathbbm{h})^\circ$, let $\mathbf{g}\mathbf{h}=\mathbf{i}$ by Lemma \ref{L;Lemma2.3}. So $\mathbf{g}\triangle\mathbf{h}=\mathbf{h}\triangle\mathbf{g}$ and $\mathbf{g}\mathbf{h}=\mathbf{h}\mathbf{g}$ by Lemma \ref{L;Lemma2.3}. By Lemmas \ref{L;Lemma2.3} and \ref{L;Lemma2.4}, notice that $(\mathbf{g}, \mathbf{h}, \mathbf{i})\in\mathbb{P}$, $p_{\mathbf{g}, \mathbf{i}}^\mathbf{h}\neq 0$, and $\mathbf{g}\triangle\mathbf{i}\preceq\mathbf{h}\preceq\mathbf{g}\mathbf{i}$ are pairwise equivalent. Let $\mathbf{j}, \mathbf{k}, \mathbf{l}\in\mathbb{E}$. If $\mathbbm{l}\!=\!(\mathbbm{g}\triangle\mathbbm{k})\cup((\mathbbm{g}\cap\mathbbm{k})^\circ\setminus\mathbbm{i})\cup
((\mathbbm{h}\cup\mathbbm{j})\cap(\mathbbm{g}\cap\mathbbm{i}\cap\mathbbm{k})^\circ)$, let $[\mathbf{g}, \mathbf{h}, \mathbf{i}, \mathbf{j},\mathbf{k}]\!=\!\mathbf{l}$ by Lemma \ref{L;Lemma2.3}. Notice that $(\mathbf{g}, [\mathbf{g}, \mathbf{h}, \mathbf{i}, \mathbf{j},\mathbf{k}], \mathbf{k})\in\mathbb{P}$. In particular, if $(\mathbf{g}, \mathbf{h}, \mathbf{g}), (\mathbf{g}, \mathbf{i}, \mathbf{g})\in\mathbb{P}$, then $\mathbf{h}\cup\mathbf{i}=[\mathbf{g}, \mathbf{h}, \mathbf{g}, \mathbf{i},\mathbf{g}]$ by Lemma \ref{L;Lemma2.3}. The following lemmas and examples are required.
\begin{lem}\label{L;Lemma2.8}\cite[Theorems 3.13, 3.23]{HJ}
$\mathbb{T}$ has an $\F$-basis $\{B_{\mathbf{a}, \mathbf{b}, \mathbf{c}}: (\mathbf{a}, \mathbf{b}, \mathbf{c})\in\mathbb{P}\}$. Moreover, assume that $(\mathbf{g},\mathbf{h},\mathbf{i}), (\mathbf{j}, \mathbf{k}, \mathbf{l})\in\mathbb{P}$. Then $B_{\mathbf{g}, \mathbf{h}, \mathbf{i}}B_{\mathbf{j},
\mathbf{k},\mathbf{l}}\!=\!\delta_{\mathbf{i},\mathbf{j}}\overline{k_{\mathbf{h}\cap\mathbf{i}\cap\mathbf{k}}}B_{\mathbf{g},[\mathbf{g},\mathbf{h},\mathbf{i}, \mathbf{k}, \mathbf{l}], \mathbf{l}}.$
\end{lem}
\begin{lem}\label{L;Lemma2.9}\cite[Theorem 6.9]{HJ}
$\mathrm{Rad}(\mathbb{T})$ has an $\F$-basis $\{B_{\mathbf{a}, \mathbf{b}, \mathbf{c}}\!:\!(\mathbf{a}, \mathbf{b}, \mathbf{c})\!\in\!\mathbb{P}, p\!\mid\! k_\mathbf{b}\}$. Moreover, the nilpotent index of $\mathrm{Rad}(\mathbb{T})$ is equal to $2n_++1$. Furthermore, the following are equivalent: $\mathrm{Rad}(\mathbb{T})=\{O\}$; $n_+=0$; and $\mathfrak{S}$ is a $p'$-valenced scheme.
\end{lem}
\begin{lem}\label{L;Lemma2.10}\cite[Theorems 4.10, 4.13]{HJ}
$\mathrm{Z}(\mathbb{T})$ has an $\F$-basis $\{C_\mathbf{a}: (\mathbf{1},\mathbf{a}, \mathbf{1})\in\mathbb{P}\}$. Moreover, assume that $(\mathbf{1}, \mathbf{g}, \mathbf{1}), (\mathbf{1}, \mathbf{h}, \mathbf{1})\in\mathbb{P}$. Then
$C_\mathbf{g}C_\mathbf{h}=\overline{k_{\mathbf{g}\cap\mathbf{h}}}C_{\mathbf{g}\cup\mathbf{h}}$.
\end{lem}
\begin{lem}\label{L;Lemma2.11}\cite[Corollary 5.22]{HJ}
$\mathrm{Rad}(\mathrm{Z}(\mathbb{T}))$ has an $\F$-basis $\{C_\mathbf{a}\!:\! (\mathbf{1}, \mathbf{a}, \mathbf{1})\!\in\!\mathbb{P}, p\!\mid\! k_\mathbf{a}\}$. Moreover, the nilpotent index of $\mathrm{Rad}(\mathrm{Z}(\mathbb{T}))$ is equal to $n_+\!+\!1$. Furthermore, the following are equivalent: $\mathrm{Rad}(\mathrm{Z}(\mathbb{T}))=\{O\}$; $n_+=0$; and $\mathfrak{S}$ is a $p'$-valenced scheme.
\end{lem}
\begin{eg}\label{E;Example2.12}
Assume that $n=|\mathbb{U}_1|=2$, $|\mathbb{U}_2|=3$, $\mathbf{g}=(0,1)$, $\mathbf{h}=(1,0)$. Then $\mathbb{E}\!=\!\{\mathbf{0}, \mathbf{g}, \mathbf{h}, \mathbf{1}\}$. Then $\mathbb{T}$ has an $\F$-basis containing exactly $B_{\mathbf{0},\mathbf{0},\mathbf{0}}$, $B_{\mathbf{0},\mathbf{g},\mathbf{g}}$, $B_{\mathbf{0},\mathbf{h},\mathbf{h}}$, $B_{\mathbf{0},\mathbf{1},\mathbf{1}}$, $B_{\mathbf{g},\mathbf{0},\mathbf{g}}$, $B_{\mathbf{g},\mathbf{g},\mathbf{0}}$, $B_{\mathbf{g},\mathbf{g},\mathbf{g}}$, $B_{\mathbf{g},\mathbf{h},\mathbf{1}}$, $B_{\mathbf{g},\mathbf{1},\mathbf{h}}$, $B_{\mathbf{g},\mathbf{1},\mathbf{1}}$, $B_{\mathbf{h},\mathbf{0},\mathbf{h}}$, $B_{\mathbf{h},\mathbf{g},\mathbf{1}}$, $B_{\mathbf{h},\mathbf{h},\mathbf{0}}$, $B_{\mathbf{h},\mathbf{1},\mathbf{g}}$, $B_{\mathbf{1},\mathbf{0},\mathbf{1}}$, $B_{\mathbf{1},\mathbf{g},\mathbf{h}}$, $B_{\mathbf{1},\mathbf{g},\mathbf{1}}$, $B_{\mathbf{1},\mathbf{h},\mathbf{g}}$, $B_{\mathbf{1},\mathbf{1},\mathbf{0}}$, $B_{\mathbf{1},\mathbf{1},\mathbf{g}}$ and $B_{\mathbf{g},\mathbf{1}, \mathbf{1}}B_{\mathbf{1},\mathbf{g}, \mathbf{1}}\!=\!\overline{2}B_{\mathbf{g},\mathbf{1}, \mathbf{1}}$ by Lemmas \ref{L;Lemma2.8} and \ref{L;Lemma2.5}.
\end{eg}
\begin{eg}\label{E;Example2.13}
Assume that $n=|\mathbb{U}_1|=2$, $|\mathbb{U}_2|=3$, $\mathbf{g}=(0,1)$, $\mathbf{h}=(1,0)$. Then $\mathbb{E}=\{\mathbf{0}, \mathbf{g}, \mathbf{h}, \mathbf{1}\}$. Then $\mathrm{Z}(\mathbb{T})$ has an $\F$-basis $\{C_\mathbf{0}, C_\mathbf{g}\}$ by Lemma \ref{L;Lemma2.10}. Moreover, notice that $C_\mathbf{0}C_\mathbf{0}=C_\mathbf{0}$, $C_\mathbf{0}C_\mathbf{g}=C_\mathbf{g}C_\mathbf{0}=C_\mathbf{g}$, $C_\mathbf{g}C_\mathbf{g}=\overline{2}C_\mathbf{g}$ by Lemmas \ref{L;Lemma2.10} and \ref{L;Lemma2.5}.
\end{eg}
Let $n_{\mathbf{g}, \mathbf{h}, \mathbf{i}}=|(\mathbbm{g}\cap\mathbbm{i})^-\setminus\mathbbm{h}|$ and $n_{\mathbf{g}, \mathbf{h}}=n_{\mathbf{g}, \mathbf{h}, \mathbf{g}}+|\mathbbm{g}\setminus(\mathbbm{g}^\circ\cup\mathbbm{h})|$. So $n_{\mathbf{g}, \mathbf{h}, \mathbf{i}}=n_{\mathbf{i},\mathbf{h}, \mathbf{g}}$. If $m\in[0, n_{\mathbf{g},\mathbf{h}}]$, let $\mathbb{U}_{\mathbf{g}, \mathbf{h}, m}=\{\mathbf{a}: \mathbf{a}\in\mathbb{E}, \mathbf{h}\preceq\mathbf{a}\preceq\mathbf{g}, p\nmid k_\mathbf{a}, |\mathbbm{a}\setminus\mathbbm{h}|=m\}$. If $\mathbf{h}\preceq\mathbf{g}$, $p\nmid k_\mathbf{h}$, $q, r\in[0, n_{\mathbf{g}, \mathbf{h}}]$, and $q\neq r$, then $\mathbb{U}_{\mathbf{g}, \mathbf{h}, q}\neq\varnothing=\mathbb{U}_{\mathbf{g}, \mathbf{h}, q}\cap\mathbb{U}_{\mathbf{g}, \mathbf{h}, r}$ by Lemmas \ref{L;Lemma2.3} and \ref{L;Lemma2.5}. If $(\mathbf{g}, \mathbf{h}, \mathbf{i})\in\mathbb{P}$ and $p\nmid k_\mathbf{h}$, then $(\mathbf{g}, \mathbf{q}, \mathbf{i})\in\mathbb{P}$ for any $m\in[0,n_{\mathbf{g}\mathbf{i}, \mathbf{h}}]$ and $\mathbf{q}\!\in\!\mathbb{U}_{\mathbf{g}\mathbf{i}, \mathbf{h}, m}$.
If $p\!\nmid\! k_\mathbf{g}$ or $p\!\nmid\! k_\mathbf{h}$, then $p\!\nmid\! k_{\mathbf{g}\cap\mathbf{h}}$ by Lemma \ref{L;Lemma2.5}. If $(\mathbf{g}, \mathbf{h}, \mathbf{i})\!\in\!\mathbb{P}$ and $p\!\nmid\! k_\mathbf{h}$, let
\begin{align}\label{Eq;6}
D_{\mathbf{g}, \mathbf{h}, \mathbf{i}}=
\sum_{m=0}^{n_{\mathbf{g}\mathbf{i}, \mathbf{h}}}\sum_{\mathbf{q}\in\mathbb{U}_{\mathbf{g}\mathbf{i}, \mathbf{h}, m}}(\overline{-1})^m\overline{k_{\mathbf{i}\cap\mathbf{q}}}^{-1}B_{\mathbf{g}, \mathbf{q}, \mathbf{i}}
\end{align}
and notice that $D_{\mathbf{g}, \mathbf{h}, \mathbf{i}}, D_{\mathbf{i}, \mathbf{h}, \mathbf{g}}$ are defined matrices in $\mathbb{T}$.
For the completeness, let $D_{\mathbf{g}, \mathbf{h}, \mathbf{i}}\!=\!O$ if $(\mathbf{g}, \mathbf{h}, \mathbf{i})\!\in\!\mathbb{P}$ and $p\mid k_\mathbf{h}$. As $(\mathbf{0},\mathbf{0}, \mathbf{0})\in\{(\mathbf{a}, \mathbf{b}, \mathbf{a}):(\mathbf{a}, \mathbf{b}, \mathbf{a})\in\mathbb{P}, p\nmid k_\mathbf{b}\}$ by Lemma \ref{L;Lemma2.5}, notice that $(\mathbf{0},\mathbf{0}, \mathbf{0})\in\{(\mathbf{a}, \mathbf{b}, \mathbf{c}):(\mathbf{a}, \mathbf{b}, \mathbf{c})\in\mathbb{P}, p\nmid k_\mathbf{b}\}$. For any $(\mathbf{g}, \mathbf{h}, \mathbf{g}), (\mathbf{i}, \mathbf{j}, \mathbf{i})\in\mathbb{P}$ and $p\nmid k_\mathbf{h}k_\mathbf{j}$, let
$(\mathbf{g}, \mathbf{h}, \mathbf{g})\approx(\mathbf{i}, \mathbf{j}, \mathbf{i})$ if $\mathbbm{g}^-\setminus\mathbbm{h}=\mathbbm{i}^-\setminus\mathbbm{j}$. By Lemma \ref{L;Lemma2.3}, notice that $\approx$ is an equivalence relation on $\{(\mathbf{a}, \mathbf{b}, \mathbf{a}):(\mathbf{a}, \mathbf{b}, \mathbf{a})\in\mathbb{P}, p\nmid k_\mathbf{b}\}$ with exactly $2^{n_{-}}$ $\approx$-equivalence classes. Let $\mathbb{C}_1, \mathbb{C}_2,\ldots, \mathbb{C}_{2^{n_{-}}}$ be exactly all pairwise distinct $\approx$-equivalence classes of $\{(\mathbf{a}, \mathbf{b}, \mathbf{a}):(\mathbf{a}, \mathbf{b}, \mathbf{a})\in\mathbb{P}, p\nmid k_\mathbf{b}\}$. If $m\in[1, 2^{n_{-}}]$, let
\begin{align}\label{Eq;7}
D_m=\sum_{(\mathbf{q}, \mathbf{r}, \mathbf{q})\in\mathbb{C}_m}D_{\mathbf{q}, \mathbf{r}, \mathbf{q}}
\end{align}
and notice that $D_m$ is a defined matrix in $\mathbb{T}$. For any $(\mathbf{g}, \mathbf{h}, \mathbf{i}), (\mathbf{j}, \mathbf{k}, \mathbf{l})\!\in\!\mathbb{P}$ and $p\!\nmid\! k_\mathbf{h}k_\mathbf{k}$, let $(\mathbf{g}, \mathbf{h}, \mathbf{i})\sim(\mathbf{j}, \mathbf{k}, \mathbf{l})$ if $(\mathbbm{g}\cap\mathbbm{i})^\circ\setminus\mathbbm{h}=(\mathbbm{j}\cap\mathbbm{l})^\circ\setminus\mathbbm{k}$. Notice that $\sim$ is an equivalence relation on $\{(\mathbf{a}, \mathbf{b}, \mathbf{c}): (\mathbf{a}, \mathbf{b}, \mathbf{c})\in\mathbb{P}, p\nmid k_\mathbf{b}\}$. If $m\in[1, 2^{n_{-}}]$, Lemma \ref{L;Lemma2.3} implies that there are exactly $2^{n_+}$ $\sim$-equivalence classes of $\{(\mathbf{a}, \mathbf{b}, \mathbf{c}): (\mathbf{a}, \mathbf{b}, \mathbf{c})\in\mathbb{P}, p\nmid k_\mathbf{b}\}$ having nontrivial intersection with $\mathbb{C}_m$. If $m\in[1,2^{n_{-}}]$, let $\mathbb{C}_{m,1}, \mathbb{C}_{m,2},\ldots, \mathbb{C}_{m, 2^{n+}}$ be all these $\sim$-equivalence classes of $\{(\mathbf{a}, \mathbf{b}, \mathbf{c}): (\mathbf{a}, \mathbf{b}, \mathbf{c})\in\mathbb{P}, p\nmid k_\mathbf{b}\}$. If $m\in[1,2^{n_{-}}]$, let $\mathbb{C}_m(q)=\{\mathbf{a}: \exists\ \mathbf{b}\in\mathbb{E}, (\mathbf{a}, \mathbf{b}, \mathbf{a})\in\mathbb{C}_{m, q}\}$ for any $q\in[1, 2^{n_+}]$. If $m\in[1,2^{n_{-}}]$ and $q\in[1, 2^{n_+}]$, Lemma \ref{L;Lemma2.3} implies that $|\mathbb{C}_m(q)|=2^{n-n_+-n_{\mathbf{r}, \mathbf{s}, \mathbf{r}}}$ for any $(\mathbf{r}, \mathbf{s},\mathbf{r})\in\mathbb{C}_m$. We end the whole section by displaying a sequence of lemmas and a related example.
\begin{lem}\label{L;Lemma2.14}\cite[Lemmas 7.3, 7.13, Theorem 7.15]{HJ}
$\mathbb{T}$ has an $\F$-basis containing exactly all matrices in $\{B_{\mathbf{a}, \mathbf{b}, \mathbf{c}}\!:\!(\mathbf{a}, \mathbf{b}, \mathbf{c})\!\in\!\mathbb{P}, p\mid k_\mathbf{b}\}\cup\{D_{\mathbf{a}, \mathbf{b}, \mathbf{c}}\!:\! (\mathbf{a}, \mathbf{b}, \mathbf{c})\!\in\!\mathbb{P}, p\nmid k_\mathbf{b}\}$. Moreover, assume that $(\mathbf{g}, \mathbf{h}, \mathbf{i}), (\mathbf{j}, \mathbf{k}, \mathbf{l})\in\mathbb{P}$ and $p\nmid k_\mathbf{h}k_\mathbf{k}$. Then $p\nmid k_{[\mathbf{g}, \mathbf{h}, \mathbf{i}, \mathbf{k}, \mathbf{l}]}$ if $\mathbf{i}=\mathbf{j}$ and $(\mathbbm{g}\cap\mathbbm{i})^\circ\setminus\mathbbm{h}\!=\!(\mathbbm{i}\cap\mathbbm{l})^\circ\setminus\mathbbm{k}$. Furthermore, $D_{\mathbf{g}, \mathbf{h},\mathbf{i}}D_{\mathbf{j}, \mathbf{k}, \mathbf{l}}=\delta_{\mathbf{i},\mathbf{j}}\delta_{(\mathbbm{g}\cap\mathbbm{i})^\circ\setminus\mathbbm{h}, (\mathbbm{i}\cap\mathbbm{l})^\circ\setminus\mathbbm{k}}D_{\mathbf{g}, [\mathbf{g}, \mathbf{h},\mathbf{i}, \mathbf{k}, \mathbf{l}], \mathbf{l}}$.
\end{lem}
\begin{lem}\label{L;Lemma2.15}\cite[Corollary 7.17]{HJ}
Assume that $(\mathbf{g}, \mathbf{h}, \mathbf{i}),(\mathbf{i}, \mathbf{j}, \mathbf{k}), (\mathbf{g}, \mathbf{l}, \mathbf{k})\in\mathbb{P}$ and $p\nmid k_\mathbf{h}k_\mathbf{j}k_\mathbf{l}$. Then $D_{\mathbf{g}, \mathbf{h}, \mathbf{i}}D_{\mathbf{i}, \mathbf{j}, \mathbf{k}}=D_{\mathbf{g}, \mathbf{l}, \mathbf{k}}$ only if $(\mathbbm{g}\cap\mathbbm{i})^\circ\setminus\mathbbm{h}=(\mathbbm{i}\cap\mathbbm{k})^\circ\setminus\mathbbm{j}=
(\mathbbm{g}\cap\mathbbm{k})^\circ\setminus\mathbbm{l}$.
\end{lem}
\begin{lem}\label{L;Lemma2.16}\cite[Lemma 8.6]{HJ}
Assume that $g\!\in\![1, 2^{n_-}\!]$ and $h\!\in\![1, 2^{n_+}]$. Then there is a bijection from $\mathbb{C}_{g, h}$ to $\mathbb{C}_g(h)\times\mathbb{C}_g(h)$ sending $(\mathbf{i}, \mathbf{j}, \mathbf{k})$ to $(\mathbf{i}, \mathbf{k})$
for any $(\mathbf{i}, \mathbf{j}, \mathbf{k})\!\in\!\mathbb{C}_{g, h}$.
\end{lem}
\begin{eg}\label{E;Example2.17}
Assume that $n=|\mathbb{U}_1|=2$, $|\mathbb{U}_2|=3$, $\mathbf{g}=(0,1)$, $\mathbf{h}=(1,0)$. Then $\mathbb{E}=\{\mathbf{0}, \mathbf{g}, \mathbf{h}, \mathbf{1}\}$. If $p\neq2$, then $\mathbb{T}$ has an $\F$-basis containing exactly $D_{\mathbf{0},\mathbf{0},\mathbf{0}}$, $D_{\mathbf{0},\mathbf{g},\mathbf{g}}$, $D_{\mathbf{0},\mathbf{h},\mathbf{h}}$, $D_{\mathbf{0},\mathbf{1},\mathbf{1}}$, $D_{\mathbf{g},\mathbf{0},\mathbf{g}}$, $D_{\mathbf{g},\mathbf{g},\mathbf{0}}$, $D_{\mathbf{g},\mathbf{g},\mathbf{g}}$, $D_{\mathbf{g},\mathbf{h},\mathbf{1}}$, $D_{\mathbf{g},\mathbf{1},\mathbf{h}}$, $D_{\mathbf{g},\mathbf{1},\mathbf{1}}$, $D_{\mathbf{h},\mathbf{0},\mathbf{h}}$, $D_{\mathbf{h},\mathbf{g},\mathbf{1}}$, $D_{\mathbf{h},\mathbf{h},\mathbf{0}}$, $D_{\mathbf{h},\mathbf{1},\mathbf{g}}$, $D_{\mathbf{1},\mathbf{0},\mathbf{1}}$, $D_{\mathbf{1},\mathbf{g},\mathbf{h}}$, $D_{\mathbf{1},\mathbf{g},\mathbf{1}}$, $D_{\mathbf{1},\mathbf{h},\mathbf{g}}$, $D_{\mathbf{1},\mathbf{1},\mathbf{0}}$, $D_{\mathbf{1},\mathbf{1},\mathbf{g}}$ by Lemma \ref{L;Lemma2.14}. For the remaining case $p=2$,  notice that $\mathbb{T}$ has an $\F$-basis containing exactly $B_{\mathbf{0}, \mathbf{g}, \mathbf{g}}$, $B_{\mathbf{0}, \mathbf{1}, \mathbf{1}}$, $B_{\mathbf{g}, \mathbf{g}, \mathbf{0}}$, $B_{\mathbf{g}, \mathbf{g}, \mathbf{g}}$, $B_{\mathbf{g}, \mathbf{1}, \mathbf{h}}$, $B_{\mathbf{g}, \mathbf{1}, \mathbf{1}}$, $B_{\mathbf{h}, \mathbf{g}, \mathbf{1}}$, $B_{\mathbf{h}, \mathbf{1}, \mathbf{g}}$, $B_{\mathbf{1}, \mathbf{g}, \mathbf{h}}$, $B_{\mathbf{1}, \mathbf{g}, \mathbf{1}}$, $B_{\mathbf{1}, \mathbf{1}, \mathbf{0}}$, $B_{\mathbf{1}, \mathbf{1}, \mathbf{g}}$, $D_{\mathbf{0}, \mathbf{0},\mathbf{0}}$, $D_{\mathbf{0}, \mathbf{h},\mathbf{h}}$, $D_{\mathbf{g}, \mathbf{0},\mathbf{g}}$, $D_{\mathbf{g}, \mathbf{h},\mathbf{1}}$, $D_{\mathbf{h}, \mathbf{0},\mathbf{h}}$, $D_{\mathbf{h}, \mathbf{h},\mathbf{0}}$, $D_{\mathbf{1}, \mathbf{0},\mathbf{1}}$, $D_{\mathbf{1}, \mathbf{h},\mathbf{g}}$ and $D_{\mathbf{g}, \mathbf{h},\mathbf{1}}D_{\mathbf{1}, \mathbf{h},\mathbf{g}}=D_{\mathbf{g}, \mathbf{0},\mathbf{g}}$ by Lemma \ref{L;Lemma2.14}.
\end{eg}
\section{Block idempotent of $\mathbb{T}$}
In this section, we get all block idempotents of $\mathbb{T}$. In particular, we get the decomposition of $\mathbb{T}$ into the direct sum of all pairwise distinct block algebras of $\mathbb{T}$. We first display a sequence of lemmas to check the fact $D_g\!\in\!\mathrm{Z}(\mathbb{T})$ for any $g\in[1,2^{n_-}]$.
\begin{lem}\label{L;Lemma3.1}
Assume that $(\mathbf{g}, \mathbf{h}, \mathbf{i})\in\mathbb{P}$ and $\mathbf{j}, \mathbf{k}, \mathbf{l}\in\mathbb{E}$. Assume that $p\nmid k_\mathbf{j}k_\mathbf{k}$ and $\mathbf{j}\preceq \mathbf{k}\preceq\mathbf{i}\mathbf{i}$. Then the conditions $p\nmid k_\mathbf{l}$, $\mathbf{j}\preceq \mathbf{l}\preceq\mathbf{i}\mathbf{i}$, and $[\mathbf{g}, \mathbf{h},\mathbf{i}, \mathbf{k}, \mathbf{i}]=[\mathbf{g}, \mathbf{h},\mathbf{i},\mathbf{l}, \mathbf{i}]$ hold together if and only if there is some
$\mathbb{U}\subseteq(\mathbbm{h}\cap\mathbbm{i})^-\setminus\mathbbm{j}$ such that $\mathbbm{l}=\mathbbm{j}\cup(\mathbbm{k}\setminus(\mathbbm{h}\cup\mathbbm{j}))\cup\mathbb{U}$.
\end{lem}
\begin{proof}
As $\mathbbm{i}\setminus\mathbbm{g}\subseteq\mathbbm{h}$ and $\mathbbm{k}\cup\mathbbm{l}\subseteq\mathbbm{i}^\circ$, notice that
$((\mathbbm{g}\cap\mathbbm{i})^\circ\cap\mathbbm{k})\setminus(\mathbbm{h}\cup\mathbbm{j})=\mathbbm{k}\setminus(\mathbbm{h}\cup\mathbbm{j})$ and $((\mathbbm{g}\cap\mathbbm{i})^\circ\cap\mathbbm{l})\setminus(\mathbbm{h}\cup\mathbbm{j})=\mathbbm{l}\setminus(\mathbbm{h}\cup\mathbbm{j})$. As $\mathbbm{i}\setminus\mathbbm{g}\subseteq\mathbbm{h}$ and
$\mathbbm{j}\subseteq\mathbbm{k}\cap\mathbbm{l}\subseteq\mathbbm{i}^\circ$, notice that $((\mathbbm{g}\cap\mathbbm{i})^\circ\cap\mathbbm{j}\cap\mathbbm{k})\setminus\mathbbm{h}=\mathbbm{j}\setminus\mathbbm{h}=((\mathbbm{g}\cap\mathbbm{i})^\circ\cap\mathbbm{j}\cap\mathbbm{l})\setminus\mathbbm{h}$.
As $\mathbbm{g}\triangle\mathbbm{i}\subseteq\mathbbm{h}\subseteq(\mathbbm{g}\triangle\mathbbm{i})\cup(\mathbbm{g}\cap\mathbbm{i})^\circ$, Lemma \ref{L;Lemma2.3} implies that $\mathbbm{h}\cup(\mathbbm{j}\setminus\mathbbm{h})\cup(\mathbbm{k}\setminus(\mathbbm{h}\cup\mathbbm{j}))=\mathbbm{h}\cup(\mathbbm{j}\setminus\mathbbm{h})\cup(\mathbbm{l}\setminus(\mathbbm{h}\cup\mathbbm{j}))$ if and only if $[\mathbf{g}, \mathbf{h},\mathbf{i}, \mathbf{k}, \mathbf{i}]=[\mathbf{g}, \mathbf{h},\mathbf{i},\mathbf{l}, \mathbf{i}]$. This implies that $[\mathbf{g}, \mathbf{h},\mathbf{i}, \mathbf{k}, \mathbf{i}]\!=\![\mathbf{g}, \mathbf{h},\mathbf{i},\mathbf{l}, \mathbf{i}]$ if and only if $\mathbbm{k}\setminus(\mathbbm{h}\cup\mathbbm{j})=\mathbbm{l}\setminus(\mathbbm{h}\cup\mathbbm{j})$. By combining the hypotheses $p\nmid k_\mathbf{j}k_\mathbf{k}k_\mathbf{l}$, $\mathbbm{j}\subseteq\mathbbm{l}\subseteq\mathbbm{i}^\circ$, Lemma \ref{L;Lemma2.5}, the equation $[\mathbf{g}, \mathbf{h},\mathbf{i}, \mathbf{k}, \mathbf{i}]=[\mathbf{g}, \mathbf{h},\mathbf{i},\mathbf{l}, \mathbf{i}]$ implies that $\mathbbm{l}=\mathbbm{j}\cup(\mathbbm{k}\setminus(\mathbbm{h}\cup\mathbbm{j}))\cup\mathbb{U}$ and
$\mathbb{U}\subseteq(\mathbbm{h}\cap\mathbbm{i})^-\setminus\mathbbm{j}$. For the remaining direction, if $\mathbbm{l}=\mathbbm{j}\cup(\mathbbm{k}\setminus(\mathbbm{h}\cup\mathbbm{j}))\cup\mathbb{U}$ and $\mathbb{U}\subseteq(\mathbbm{h}\cap\mathbbm{i})^-\setminus\mathbbm{j}$, then the combination of the hypotheses $p\!\nmid\! k_\mathbf{j}k_\mathbf{k}$, $\mathbbm{j}\subseteq\mathbbm{k}\subseteq\mathbbm{i}^\circ$, and Lemma \ref{L;Lemma2.5} implies that $p\nmid k_\mathbf{l}$, $\mathbf{j}\preceq\mathbf{l}\preceq\mathbf{i}\mathbf{i}$, $\mathbbm{k}\setminus(\mathbbm{h}\cup\mathbbm{j})=\mathbbm{l}\setminus(\mathbbm{h}\cup\mathbbm{j})$, and $[\mathbf{g}, \mathbf{h},\mathbf{i}, \mathbf{k}, \mathbf{i}]=[\mathbf{g}, \mathbf{h},\mathbf{i},\mathbf{l}, \mathbf{i}]$. The desired lemma follows from the above discussion.
\end{proof}
\begin{lem}\label{L;Lemma3.2}
Assume that $(\mathbf{g}, \mathbf{h}, \mathbf{i})\in\mathbb{P}$ and $\mathbf{j}, \mathbf{k},\mathbf{l}\in\mathbb{E}$. Assume that $p\nmid k_\mathbf{j}k_\mathbf{k}k_\mathbf{l}$. Assume that
$\mathbf{j}\preceq\mathbf{k}\preceq\mathbf{i}\mathbf{i}$, $\mathbf{j}\preceq \mathbf{l}\preceq\mathbf{i}\mathbf{i}$, $[\mathbf{g}, \mathbf{h},\mathbf{i}, \mathbf{k}, \mathbf{i}]=[\mathbf{g}, \mathbf{h},\mathbf{i},\mathbf{l}, \mathbf{i}]$, $m\in[0, n_{\mathbf{h}, \mathbf{j}, \mathbf{i}}]$. Assume that $|\mathbbm{l}\setminus\mathbbm{j}|\!=\!|\mathbbm{k}\setminus(\mathbbm{h}\!\cup\!\mathbbm{j})|+m$.
Then the number of the choices of $\mathbf{l}$ is equal to $${n_{\mathbf{h}, \mathbf{j}, \mathbf{i}}\choose m}.$$
\end{lem}
\begin{proof}
The desired lemma follows from an application of Lemmas \ref{L;Lemma3.1} and \ref{L;Lemma2.3}.
\end{proof}
\begin{lem}\label{L;Lemma3.3}
Assume that $(\mathbf{g}, \mathbf{h}, \mathbf{i})\in\mathbb{P}$ and $\mathbf{j}, \mathbf{k},\mathbf{l}\in\mathbb{E}$. Assume that $p\nmid k_\mathbf{j}k_\mathbf{k}k_\mathbf{l}$. Assume that $\mathbf{j}\!\preceq\! \mathbf{k}\!\preceq\!\mathbf{i}\mathbf{i}$ and $\mathbf{j}\!\preceq\! \mathbf{l}\!\preceq\!\mathbf{i}\mathbf{i}$. Then $[\mathbf{g}, \mathbf{h},\mathbf{i}, \mathbf{k}, \mathbf{i}]\!=\![\mathbf{g}, \mathbf{h},\mathbf{i},\mathbf{l}, \mathbf{i}]$ only if $k_{\mathbf{k}\setminus\mathbf{h}}\!=\!k_{\mathbf{l}\setminus\mathbf{h}}$.
\end{lem}
\begin{proof}
As $\mathbbm{j}\!\subseteq\!\mathbbm{k}\!\cap\!\mathbbm{l}$, notice that $\mathbbm{k}\setminus\mathbbm{h}\!=\!(\mathbbm{j}\setminus\mathbbm{h})\!\cup\!(\mathbbm{k}\setminus(\mathbbm{h}\cup\mathbbm{j}))\!=\!\mathbbm{l}\setminus\mathbbm{h}$
by Lemma \ref{L;Lemma3.1}. The desired lemma follows from the proved equation $\mathbbm{k}\setminus\mathbbm{h}=\mathbbm{l}\setminus\mathbbm{h}$ and Lemma \ref{L;Lemma2.5}.
\end{proof}
\begin{lem}\label{L;Lemma3.4}
Assume that $(\mathbf{g}, \mathbf{h}, \mathbf{i})\in\mathbb{P}$. Then $B_{\mathbf{g}, \mathbf{h}, \mathbf{i}}^T=B_{\mathbf{i}, \mathbf{h}, \mathbf{g}}$. Moreover, assume that
$p\nmid k_\mathbf{h}$. Then $\overline{k_{\mathbf{i}\setminus\mathbf{g}}}D_{\mathbf{g}, \mathbf{h}, \mathbf{i}}^T=\overline{k_{\mathbf{g}\setminus\mathbf{i}}}D_{\mathbf{i}, \mathbf{h}, \mathbf{g}}$. In particular, if $\mathbf{g}=\mathbf{i}$, then $D_{\mathbf{g},\mathbf{h},\mathbf{g}}^T=D_{\mathbf{g},\mathbf{h},\mathbf{g}}$.
\end{lem}
\begin{proof}
The first statement follows from Equations \eqref{Eq;4} and \eqref{Eq;1}. As $\mathbf{g}\triangle\mathbf{i}\preceq\mathbf{h}$, Lemma \ref{L;Lemma2.5} implies that
$k_{\mathbf{i}\cap\mathbf{j}}=k_{\mathbf{i}\setminus\mathbf{g}}k_{\mathbf{g}\cap\mathbf{i}\cap\mathbf{j}}$ and $k_{\mathbf{g}\cap\mathbf{j}}=k_{\mathbf{g}\setminus\mathbf{i}}k_{\mathbf{g}\cap\mathbf{i}\cap\mathbf{j}}$
for any $k\in[0,n_{\mathbf{g}\mathbf{i},\mathbf{h}}]$ and $\mathbf{j}\in\mathbb{U}_{\mathbf{g}\mathbf{i}, \mathbf{h}, k}$. The second statement is from Equation \eqref{Eq;6} and the first statement. As $k_{\mathbf{g}\setminus\mathbf{g}}=1$ by Lemma \ref{L;Lemma2.5}, the desired lemma follows from the second statement.
\end{proof}
\begin{lem}\label{L;Lemma3.5}
Assume that $g\in[1, 2^{n_-}]$ and $(\mathbf{h},\mathbf{i}, \mathbf{h}), (\mathbf{h},\mathbf{j}, \mathbf{h})\in\mathbb{C}_g$. Then $\mathbf{i}=\mathbf{j}$ and $(\mathbf{h}, \mathbf{i}, \mathbf{h})$ is the unique element in $\mathbb{C}_g$ that satisfies the equation $E_\mathbf{h}^*D_g\!=\!D_gE_\mathbf{h}^*\!=\!D_{\mathbf{h}, \mathbf{i}, \mathbf{h}}$.
\end{lem}
\begin{proof}
Notice that $\mathbbm{i}\cup\mathbbm{j}\subseteq\mathbbm{h}^-$ and $\mathbbm{h}^-\setminus\mathbbm{i}=\mathbbm{h}^-\setminus\mathbbm{j}$ by Lemma \ref{L;Lemma2.5}. The desired lemma follows from the combination of Lemma \ref{L;Lemma2.3}, Equations \eqref{Eq;7}, \eqref{Eq;6}, \eqref{Eq;4}, \eqref{Eq;2}.
\end{proof}
\begin{lem}\label{L;Lemma3.6}
Assume that $g\in[1,2^{n_-}]$ and $(\mathbf{h}, \mathbf{i}, \mathbf{j})\in\mathbb{P}$. Assume that $(\mathbf{j},\mathbf{k},\mathbf{j})\in\mathbb{C}_g$ and $n_{\mathbf{i}, \mathbf{k}, \mathbf{j}}>0$. Then $B_{\mathbf{h}, \mathbf{i}, \mathbf{j}}D_g=B_{\mathbf{h}, \mathbf{i}, \mathbf{j}}D_{\mathbf{j},\mathbf{k},\mathbf{j}}=D_{\mathbf{j},\mathbf{k},\mathbf{j}}B_{\mathbf{h}, \mathbf{i}, \mathbf{j}}=D_gB_{\mathbf{h}, \mathbf{i}, \mathbf{j}}=O$.
\end{lem}
\begin{proof}
Notice that $B_{\mathbf{h}, \mathbf{i}, \mathbf{j}}D_g\!=\!B_{\mathbf{h}, \mathbf{i}, \mathbf{j}}E_\mathbf{j}^*D_g\!=\!B_{\mathbf{h}, \mathbf{i}, \mathbf{j}}D_{\mathbf{j}, \mathbf{k}, \mathbf{j}}$ by combining Equations \eqref{Eq;4}, \eqref{Eq;2}, Lemma \ref{L;Lemma3.5}. If $\mathbf{h}\neq \mathbf{j}$, then $D_{\mathbf{j},\mathbf{k},\mathbf{j}}B_{\mathbf{h}, \mathbf{i},
\mathbf{j}}\!=\!D_{\mathbf{j},\mathbf{k},\mathbf{j}}E_\mathbf{h}^*B_{\mathbf{h}, \mathbf{i}, \mathbf{j}}\!=\!O$ by combining Equations \eqref{Eq;6}, \eqref{Eq;4}, \eqref{Eq;2}. If $\mathbf{h}\!=\!\mathbf{j}$, then $D_{\mathbf{j},\mathbf{k},\mathbf{j}}B_{\mathbf{j}, \mathbf{i}, \mathbf{j}}\!=\!(B_{\mathbf{j}, \mathbf{i}, \mathbf{j}}D_{\mathbf{j},\mathbf{k},\mathbf{j}})^T$ by Lemma \ref{L;Lemma3.4}. If $D_gE_\mathbf{h}^*=O$, then $D_gB_{\mathbf{h}, \mathbf{i}, \mathbf{j}}=D_gE_\mathbf{h}^*B_{\mathbf{h}, \mathbf{i}, \mathbf{j}}=O$ by Equations \eqref{Eq;4} and \eqref{Eq;2}. If $D_gE_\mathbf{h}^*\neq O$, then $(\mathbf{h}, \mathbf{l}, \mathbf{h})\in\mathbb{C}_g$ and $n_{\mathbf{h}, \mathbf{l}, \mathbf{i}}=n_{\mathbf{i}, \mathbf{k}, \mathbf{j}}>0$ for some $\mathbf{l}\in\mathbb{E}$. If $D_gE_\mathbf{h}^*\neq O$, then $D_gB_{\mathbf{h}, \mathbf{i}, \mathbf{j}}=D_gE_\mathbf{h}^*B_{\mathbf{h}, \mathbf{i}, \mathbf{j}}=D_{\mathbf{h}, \mathbf{l}, \mathbf{h}}B_{\mathbf{h}, \mathbf{i}, \mathbf{j}}=(B_{\mathbf{j}, \mathbf{i}, \mathbf{h}}D_{\mathbf{h}, \mathbf{l}, \mathbf{h}})^T$ by combining Equations \eqref{Eq;4}, \eqref{Eq;2}, Lemmas \ref{L;Lemma3.5}, \ref{L;Lemma3.4}. In conclusion, it suffices to check that $B_{\mathbf{h}, \mathbf{i}, \mathbf{j}}D_{\mathbf{j}, \mathbf{k}, \mathbf{j}}=O$.

By Equation \eqref{Eq;6} and Lemma \ref{L;Lemma2.8}, notice that $B_{\mathbf{h}, \mathbf{i}, \mathbf{j}}D_{\mathbf{j}, \mathbf{k}, \mathbf{j}}$ is a unique $\F$-linear combination of the matrices in $\{B_{\mathbf{h}, [\mathbf{h}, \mathbf{i}, \mathbf{j}, \mathbf{a}, \mathbf{j}], \mathbf{j}}: \mathbf{a}\in\mathbb{E}, \mathbf{k}\preceq\mathbf{a}\preceq\mathbf{j}\mathbf{j}, p\nmid k_\mathbf{a}\}$. Assume that $\mathbf{m}\in\mathbb{E}$, $\mathbf{k}\preceq\mathbf{m}\preceq\mathbf{j}\mathbf{j}$, $p\nmid k_\mathbf{m}$, $q=|\mathbbm{m}\setminus(\mathbbm{i}\cup\mathbbm{k})|$. Let $c_{[\mathbf{h}, \mathbf{i}, \mathbf{j}, \mathbf{m}, \mathbf{j}]}$ be the coefficient of $B_{\mathbf{h}, [\mathbf{h}, \mathbf{i},\mathbf{j}, \mathbf{m}, \mathbf{j}],\mathbf{j}}$ in this $\F$-linear combination that represents $B_{\mathbf{h}, \mathbf{i},\mathbf{j}}D_{\mathbf{j},\mathbf{k},\mathbf{j}}$.  Lemma \ref{L;Lemma2.5} implies that $k_{\mathbf{j}\cap\mathbf{r}}\!=\!k_\mathbf{r}\!=\!k_{\mathbf{r}\setminus\mathbf{i}}k_{\mathbf{i}\cap\mathbf{r}}\!=\!k_{\mathbf{r}\setminus\mathbf{i}}k_{\mathbf{i}\cap\mathbf{j}\cap\mathbf{r}}$ for any $(\mathbf{j}, \mathbf{r}, \mathbf{j})\!\in\!\mathbb{P}$. The combination of Equation \eqref{Eq;6}, Lemmas \ref{L;Lemma2.8}, \ref{L;Lemma3.1}, \ref{L;Lemma3.2}, \ref{L;Lemma3.3}, the Binomial Theorem thus implies that
$$c_{[\mathbf{h}, \mathbf{i}, \mathbf{j}, \mathbf{m}, \mathbf{j}]}=\sum_{r=0}^{n_{\mathbf{i},\mathbf{k},\mathbf{j}}}(\overline{-1})^{q+r}\overline{k_{\mathbf{m}\setminus\mathbf{i}}}^{-1}\overline{{n_{\mathbf{i},\mathbf{k},\mathbf{j}}\choose r}}=\overline{0}.$$
As $B_{\mathbf{h}, [\mathbf{h}, \mathbf{i}, \mathbf{j}, \mathbf{m}, \mathbf{j}], \mathbf{j}}$ is arbitrarily chosen from $\{B_{\mathbf{h}, [\mathbf{h}, \mathbf{i}, \mathbf{j}, \mathbf{a}, \mathbf{j}], \mathbf{j}}: \mathbf{a}\in\mathbb{E}, \mathbf{k}\preceq\mathbf{a}\preceq\mathbf{j}\mathbf{j}, p\nmid k_\mathbf{a}\}$,
notice that $B_{\mathbf{h}, \mathbf{i}, \mathbf{j}}D_{\mathbf{j}, \mathbf{k}, \mathbf{j}}\!=\!O$ by the displayed equation. The desired lemma follows.
\end{proof}
\begin{lem}\label{L;Lemma3.7}
Assume that $g\in[1,2^{n_-}]$, $\mathbf{h}, \mathbf{i}\in\mathbb{E}$, $(\mathbf{j},\mathbf{k},\mathbf{j})\in\mathbb{C}_g$,
$n_{\mathbf{i},\mathbf{h}\cup\mathbf{k},\mathbf{j}}=0$, $n_{\mathbf{j}, \mathbf{h}\cup\mathbf{i}\cup\mathbf{k}, \mathbf{j}}=0$. Then there is a unique $(\mathbf{h}, \mathbf{l}, \mathbf{h})\in\mathbb{C}_g$ such that $\mathbbm{l}=(\mathbbm{h}\cap\mathbbm{k})\cup(\mathbbm{h}^-\setminus\mathbbm{j})$.
\end{lem}
\begin{proof}
By Lemma \ref{L;Lemma2.3}, there is a unique $\mathbf{l}\in\mathbb{E}$ such that $\mathbbm{l}=(\mathbbm{h}\cap\mathbbm{k})\cup(\mathbbm{h}^-\setminus\mathbbm{j})$. As $\mathbbm{k}\subseteq\mathbbm{j}^\circ$, notice that $(\mathbf{h}, \mathbf{l}, \mathbf{h})\in\mathbb{P}$. As $p\nmid k_\mathbf{k}$, notice that $p\nmid k_\mathbf{l}$ by Lemma \ref{L;Lemma2.5}. Moreover,
$\mathbbm{h}^-\setminus\mathbbm{l}=(\mathbbm{h}\cap\mathbbm{j})^-\setminus\mathbbm{k}=\mathbbm{j}^-\setminus\mathbbm{k}$ as $\mathbbm{k}\subseteq\mathbbm{j}^\circ$ and $\mathbbm{j}^-\setminus(\mathbbm{h}\cup\mathbbm{k})=(\mathbbm{i}\cap\mathbbm{j})^-\setminus(\mathbbm{h}\cup\mathbbm{k})=\varnothing$. So $(\mathbf{h}, \mathbf{l}, \mathbf{h})\in\mathbb{C}_g$ by the above proved statements.
The desired lemma follows.
\end{proof}
\begin{lem}\label{L;Lemma3.8}
Assume that $g\in[1,2^{n_-}]$ and $(\mathbf{h}, \mathbf{i}, \mathbf{j})\in\mathbb{P}$. Assume that there is no $n$-tuple $\mathbf{k}$ in $\mathbb{E}$ satisfying the condition $(\mathbf{j},\mathbf{k}, \mathbf{j})\in\mathbb{C}_g$. Then $B_{\mathbf{h}, \mathbf{i},\mathbf{j}}D_g=D_gB_{\mathbf{h}, \mathbf{i},\mathbf{j}}=O$.
\end{lem}
\begin{proof}
Notice that $B_{\mathbf{h}, \mathbf{i}, \mathbf{j}}D_g=B_{\mathbf{h}, \mathbf{i}, \mathbf{j}}E_\mathbf{j}^*D_g=O$ by combining Equations \eqref{Eq;4}, \eqref{Eq;2}, \eqref{Eq;7}, \eqref{Eq;6}. If $(\mathbf{h}, \mathbf{k},\mathbf{h})\notin\mathbb{C}_g$ for any $\mathbf{k}\in\mathbb{E}$, then $D_gB_{\mathbf{h}, \mathbf{i},\mathbf{j}}=D_gE_\mathbf{h}^*B_{\mathbf{h}, \mathbf{i},\mathbf{j}}=O$ by combining Equations \eqref{Eq;4}, \eqref{Eq;2}, \eqref{Eq;7}, \eqref{Eq;6}. Assume that there is $(\mathbf{h},\mathbf{k}, \mathbf{h})\in\mathbb{C}_g$. If $n_{\mathbf{h}, \mathbf{k}, \mathbf{i}}>0$, then $D_gB_{\mathbf{h}, \mathbf{i},\mathbf{j}}=(B_{\mathbf{j}, \mathbf{i}, \mathbf{h}}D_g)^T=O$ by combining Equation \eqref{Eq;7}, Lemmas \ref{L;Lemma3.4}, \ref{L;Lemma3.6}. So it suffices to check that $n_{\mathbf{h},\mathbf{k},\mathbf{i}}\!\neq\! 0$. As $n_{\mathbf{h}, \mathbf{i}\cup\mathbf{j}\cup\mathbf{k}, \mathbf{h}}=0$ and $(\mathbf{h}, \mathbf{k}, \mathbf{h})\!\in\!\mathbb{C}_g$, the equation $n_{\mathbf{h}, \mathbf{k}, \mathbf{i}}=0$ and Lemma \ref{L;Lemma3.7} imply that $n_{\mathbf{h},\mathbf{j}\cup\mathbf{k}, \mathbf{i}}=0$ and there is $(\mathbf{j},\mathbf{l},\mathbf{j})\in\mathbb{C}_g$. This is absurd. So $n_{\mathbf{h}, \mathbf{k}, \mathbf{i}}>0$ by this contradiction. The desired lemma follows.
\end{proof}
\begin{lem}\label{L;Lemma3.9}
Assume that $g\in[1,2^{n_-}]$ and $(\mathbf{h}, \mathbf{i}, \mathbf{j})\in\mathbb{P}$. Assume that $(\mathbf{j},\mathbf{k},\mathbf{j})\in\mathbb{C}_g$ and $n_{\mathbf{i}, \mathbf{k}, \mathbf{j}}=0$. Then
there is a unique $(\mathbf{h}, \mathbf{l}, \mathbf{h})\in\mathbb{C}_g$ such that $\mathbbm{l}=(\mathbbm{h}\cap\mathbbm{k})\cup(\mathbbm{h}^-\setminus\mathbbm{j})$. In particular,
$B_{\mathbf{h}, \mathbf{i}, \mathbf{j}}D_g=B_{\mathbf{h}, \mathbf{i}, \mathbf{j}}D_{\mathbf{j}, \mathbf{k}, \mathbf{j}}$ and $D_gB_{\mathbf{h}, \mathbf{i}, \mathbf{j}}=D_{\mathbf{h}, \mathbf{l}, \mathbf{h}}B_{\mathbf{h}, \mathbf{i}, \mathbf{j}}$ for this unique $(\mathbf{h}, \mathbf{l}, \mathbf{h})$.
\end{lem}
\begin{proof}
As $n_{\mathbf{i},\mathbf{h}\cup\mathbf{k},\mathbf{j}}=n_{\mathbf{j},\mathbf{h}\cup\mathbf{i}\cup\mathbf{k} ,\mathbf{j}}=0$, the first statement follows from Lemma \ref{L;Lemma3.7}. So $B_{\mathbf{h}, \mathbf{i}, \mathbf{j}}D_g\!=\!B_{\mathbf{h}, \mathbf{i}, \mathbf{j}}E_\mathbf{j}^*D_g=B_{\mathbf{h}, \mathbf{i}, \mathbf{j}}D_{\mathbf{j}, \mathbf{k}, \mathbf{j}}$ and $D_gB_{\mathbf{h}, \mathbf{i}, \mathbf{j}}\!=\!D_gE_\mathbf{h}^*B_{\mathbf{h}, \mathbf{i}, \mathbf{j}}\!=\!D_{\mathbf{h}, \mathbf{l}, \mathbf{h}}B_{\mathbf{h}, \mathbf{i}, \mathbf{j}}$ by the combination of Equations \eqref{Eq;4}, \eqref{Eq;2}, Lemma \ref{L;Lemma3.5}. The desired lemma follows.
\end{proof}
\begin{lem}\label{L;Lemma3.10}
Assume that $(\mathbf{g}, \mathbf{h}, \mathbf{i})\in\mathbb{P}$ and $\mathbf{j}, \mathbf{k}, \mathbf{l}\in\mathbb{E}$. Assume that $p\nmid k_\mathbf{k}k_\mathbf{l}$, $\mathbf{j}\preceq\mathbf{k}\preceq\mathbf{i}\mathbf{i},
\mathbf{j}\preceq\mathbf{l}\preceq\mathbf{i}\mathbf{i}$, $n_{\mathbf{h}, \mathbf{j}, \mathbf{k}}\!=\!n_{\mathbf{h}, \mathbf{j}, \mathbf{l}}$. Then $\mathbf{k}\!=\!\mathbf{l}$ if and only if $[\mathbf{g}, \mathbf{h},\mathbf{i}, \mathbf{k}, \mathbf{i}]\!=\![\mathbf{g}, \mathbf{h},\mathbf{i},\mathbf{l}, \mathbf{i}]$.
\end{lem}
\begin{proof}
As $\mathbbm{i}\setminus\mathbbm{g}\subseteq\mathbbm{h}$ and $\mathbbm{k}\cup\mathbbm{l}\subseteq\mathbbm{i}^\circ$, notice that $((\mathbbm{g}\cap\mathbbm{i})^\circ\cap\mathbbm{k})\setminus(\mathbbm{h}\cup\mathbbm{j})=\mathbbm{k}\setminus(\mathbbm{h}\cup\mathbbm{j})$ and
$((\mathbbm{g}\cap\mathbbm{i})^\circ\cap\mathbbm{l})\setminus(\mathbbm{h}\cup\mathbbm{j})=\mathbbm{l}\setminus(\mathbbm{h}\cup\mathbbm{j})$. As $\mathbbm{i}\setminus\mathbbm{g}\subseteq\mathbbm{h}$ and $\mathbbm{j}\subseteq\mathbbm{k}\cap\mathbbm{l}\subseteq\mathbbm{i}^\circ$, notice that
$((\mathbbm{g}\cap\mathbbm{i})^\circ\cap\mathbbm{j}\cap\mathbbm{k})\setminus\mathbbm{h}=\mathbbm{j}\setminus\mathbbm{h}=((\mathbbm{g}\cap\mathbbm{i})^\circ\cap\mathbbm{j}\cap\mathbbm{l})\setminus\mathbbm{h}$.
As $\mathbbm{g}\triangle\mathbbm{i}\subseteq\mathbbm{h}\subseteq(\mathbbm{g}\triangle\mathbbm{i})\cup(\mathbbm{g}\cap\mathbbm{i})^\circ$, Lemma \ref{L;Lemma2.3} implies that $\mathbbm{h}\cup(\mathbbm{j}\setminus\mathbbm{h})\cup(\mathbbm{k}\setminus(\mathbbm{h}\cup\mathbbm{j}))=\mathbbm{h}\cup(\mathbbm{j}\setminus\mathbbm{h})\cup(\mathbbm{l}\setminus(\mathbbm{h}\cup\mathbbm{j}))$ if and only if $[\mathbf{g}, \mathbf{h},\mathbf{i}, \mathbf{k}, \mathbf{i}]=[\mathbf{g}, \mathbf{h},\mathbf{i},\mathbf{l}, \mathbf{i}]$. This thus implies that $[\mathbf{g}, \mathbf{h},\mathbf{i}, \mathbf{k}, \mathbf{i}]=[\mathbf{g}, \mathbf{h},\mathbf{i},\mathbf{l}, \mathbf{i}]$ if and only if $\mathbbm{k}\setminus(\mathbbm{h}\cup\mathbbm{j})=\mathbbm{l}\setminus(\mathbbm{h}\cup\mathbbm{j})$. As
$\mathbbm{k}\cup\mathbbm{l}\subseteq\mathbbm{i}^\circ$, notice that $(\mathbbm{h}\cap\mathbbm{k})\setminus\mathbbm{j}=(\mathbbm{h}\cap\mathbbm{l})\setminus\mathbbm{j}$ by Lemma \ref{L;Lemma2.5}.
Therefore $[\mathbf{g}, \mathbf{h},\mathbf{i}, \mathbf{k}, \mathbf{i}]=[\mathbf{g}, \mathbf{h},\mathbf{i},\mathbf{l}, \mathbf{i}]$ if and only if $\mathbbm{k}\setminus\mathbbm{j}=\mathbbm{l}\setminus\mathbbm{j}$. As $\mathbbm{j}\subseteq\mathbbm{k}\cap\mathbbm{l}$, the desired lemma follows from the above discussion and Lemma \ref{L;Lemma2.3}.
\end{proof}
\begin{lem}\label{L;Lemma3.11}
Assume that $g\in[1,2^{n_-}]$, $(\mathbf{h}, \mathbf{i}, \mathbf{j})\in\mathbb{P}$, $(\mathbf{j}, \mathbf{k}, \mathbf{j}), (\mathbf{h},\mathbf{l},\mathbf{h})\in\mathbb{C}_g$. Then $\{[\mathbf{h}, \mathbf{i},\mathbf{j},\mathbf{a},\mathbf{j}]\!:\! \mathbf{a}\in\mathbb{E}, \mathbf{k}\preceq\mathbf{a}\preceq\mathbf{j}\mathbf{j}, p\nmid k_\mathbf{a}\}=\{[\mathbf{h}, \mathbf{i},\mathbf{j},\mathbf{a}\cup\mathbf{k},\mathbf{j}]\!:\! \mathbf{a} \!\in\!\mathbb{E}, \mathbbm{a}\subseteq\mathbbm{j}^-\setminus\mathbbm{k}\}$ and $\{[\mathbf{h}, \mathbf{a},\mathbf{h},\mathbf{i},\mathbf{j}]\!:\! \mathbf{a}\!\in\!\mathbb{E}, \mathbf{l}\preceq\mathbf{a}\preceq\mathbf{h}\mathbf{h}, p\nmid k_\mathbf{a}\}\!=\!\{[\mathbf{h}, \mathbf{a}\cup\mathbf{l},\mathbf{h},\mathbf{i},\mathbf{j}]\!:\! \mathbf{a} \!\in\!\mathbb{E}, \mathbbm{a}\subseteq\mathbbm{j}^-\setminus\mathbbm{k}\}$. Moreover, $[\mathbf{h}, \mathbf{i},\mathbf{j},\mathbf{k}\cup\mathbf{m},\mathbf{j}]=[\mathbf{h}, \mathbf{l}\cup\mathbf{m},\mathbf{h},\mathbf{i},\mathbf{j}]$
for any $\mathbf{m}\in\{\mathbf{a}:\mathbf{a}\in\mathbb{E},\mathbbm{a}\subseteq\mathbbm{j}^-\setminus\mathbbm{k}\}$. Furthermore, assume that $n_{\mathbf{i}, \mathbf{k},\mathbf{j}}\!=\!0$. Then $|\{[\mathbf{h}, \mathbf{i},\mathbf{j},\mathbf{a}\cup\mathbf{k},\mathbf{j}]\!:\! \mathbf{a}\!\in\!\mathbb{E}, \mathbbm{a}\subseteq\mathbbm{j}^-\setminus\mathbbm{k}\}|\!=\!2^{n_{\mathbf{j},\mathbf{k},\mathbf{j}}}$.
\end{lem}
\begin{proof}
As $\mathbbm{k}\subseteq\mathbbm{j}^\circ$, $\mathbbm{l}\subseteq\mathbbm{h}^\circ$, $p\nmid k_\mathbf{k}k_\mathbf{l}$, and  $\mathbbm{j}^-\setminus\mathbbm{k}=\mathbbm{h}^-\setminus\mathbbm{l}$, the first statement follows from Lemmas \ref{L;Lemma2.3} and \ref{L;Lemma2.5}. Notice that
$(\mathbbm{h}\cap\mathbbm{j})^-\setminus(\mathbbm{i}\cup\mathbbm{k})=(\mathbbm{h}\cap\mathbbm{j})^-\setminus(\mathbbm{i}\cup\mathbbm{k}\cup\mathbbm{l})=(\mathbbm{h}\cap\mathbbm{j})^-\setminus(\mathbbm{i}\cup\mathbbm{l})$ as $\mathbbm{j}^-\setminus\mathbbm{k}=\mathbbm{h}^-\setminus\mathbbm{l}$. As $p\nmid k_\mathbf{k}k_\mathbf{l}$, notice that $((\mathbbm{h}\cap\mathbbm{j})^\circ\cap\mathbbm{k})\setminus
\mathbbm{i}=((\mathbbm{h}\cap\mathbbm{j})^\circ\cap\mathbbm{l})\setminus \mathbbm{i}$ by Lemma \ref{L;Lemma2.5}. Assume that $\mathbf{m}\in\mathbb{E}$ and $\mathbbm{m}\subseteq\mathbbm{j}^-\setminus\mathbbm{k}=\mathbbm{h}^-\setminus\mathbbm{l}$. As $\mathbbm{j}\setminus\mathbbm{h}\subseteq\mathbbm{i}$ and $\mathbbm{m}\subseteq\mathbbm{j}^\circ$, notice that
$((\mathbbm{h}\cap\mathbbm{j})^\circ\cap\mathbbm{m})\setminus\mathbbm{i}=\mathbbm{m}\setminus\mathbbm{i}$. As
$\mathbbm{h}\triangle\mathbbm{j}\subseteq\mathbbm{i}\subseteq(\mathbbm{h}\triangle\mathbbm{j})\cup(\mathbbm{h}\cap\mathbbm{j})^\circ$, Lemma \ref{L;Lemma2.3} implies that
$\mathbbm{i}\cup(((\mathbbm{h}\cap\mathbbm{j})^\circ\cap\mathbbm{k})\setminus \mathbbm{i})\cup(\mathbbm{m}\setminus\mathbbm{i})=\mathbbm{i}\cup(((\mathbbm{h}\cap\mathbbm{j})^\circ\cap\mathbbm{l})\setminus \mathbbm{i})\cup(\mathbbm{m}\setminus\mathbbm{i})$ if and only if $[\mathbf{h}, \mathbf{i},\mathbf{j},\mathbf{k}\cup\mathbf{m},\mathbf{j}]=[\mathbf{h}, \mathbf{l}\cup\mathbf{m},\mathbf{h},\mathbf{i},\mathbf{j}]$.
The second statement thus follows as $\mathbf{m}$ is arbitrarily chosen from $\{\mathbf{a}:\mathbf{a}\in\mathbb{E},\mathbbm{a}\subseteq\mathbbm{j}^-\setminus\mathbbm{k}\}$. Lemma \ref{L;Lemma2.5} implies that
$p\nmid k_{\mathbf{k}\cup\mathbf{q}}$ for any $\mathbf{q}\in\{\mathbf{a}:\mathbf{a}\in\mathbb{E},\mathbbm{a}\subseteq\mathbbm{j}^-\setminus\mathbbm{k}\}$. For any $\mathbf{q}\in\{\mathbf{a}:\mathbf{a}\in\mathbb{E},\mathbbm{a}\subseteq\mathbbm{j}^-\setminus\mathbbm{k}\}$, $\mathbf{k}\preceq\mathbf{k}\cup\mathbf{q}\preceq \mathbf{j}\mathbf{j}$ and $n_{\mathbf{i},\mathbf{k},\mathbf{k}\cup\mathbf{q}}=n_{\mathbf{i},\mathbf{k},\mathbf{j}}=0$. The desired lemma follows from Lemmas \ref{L;Lemma3.10} and \ref{L;Lemma2.3}.
\end{proof}
\begin{lem}\label{L;Lemma3.12}
Assume that $g\in[1,2^{n_-}]$, $(\mathbf{h}, \mathbf{i}, \mathbf{j})\in\mathbb{P}$, $(\mathbf{j}, \mathbf{k}, \mathbf{j}), (\mathbf{h},\mathbf{l},\mathbf{h})\in\mathbb{C}_g$. Then $\{B_{\mathbf{h}, [\mathbf{h}, \mathbf{i},\mathbf{j},\mathbf{a}\cup\mathbf{k},\mathbf{j}], \mathbf{j}}: \mathbf{a}\in\mathbb{E}, \mathbbm{a}\subseteq\mathbbm{j}^-\setminus\mathbbm{k}\}=\{B_{\mathbf{h}, [\mathbf{h}, \mathbf{a}\cup\mathbf{l},\mathbf{h},\mathbf{i},\mathbf{j}], \mathbf{j}}: \mathbf{a}\in\mathbb{E}, \mathbbm{a}\subseteq\mathbbm{j}^-\setminus\mathbbm{k}\}$. Moreover, both  $B_{\mathbf{h},\mathbf{i},\mathbf{j}}D_{\mathbf{j},\mathbf{k},\mathbf{j}}$ and  $D_{\mathbf{h},\mathbf{l},\mathbf{h}}B_{\mathbf{h},\mathbf{i},\mathbf{j}}$ are unique $\F$-linear combinations of the matrices in
$$\{B_{\mathbf{h}, [\mathbf{h}, \mathbf{i},\mathbf{j},\mathbf{a}\cup\mathbf{k},\mathbf{j}], \mathbf{j}}: \mathbf{a} \in\mathbb{E}, \mathbbm{a}\subseteq\mathbbm{j}^-\setminus\mathbbm{k}\}.$$
\end{lem}
\begin{proof}
As $\{[\mathbf{h}, \mathbf{i},\mathbf{j},\mathbf{a}\cup\mathbf{k},\mathbf{j}]: \mathbf{a}\in\mathbb{E}, \mathbbm{a}\subseteq\mathbbm{j}^-\setminus\mathbbm{k}\}=\{[\mathbf{h}, \mathbf{a}\cup\mathbf{l},\mathbf{h},\mathbf{i},\mathbf{j}]: \mathbf{a}\in\mathbb{E}, \mathbbm{a}\subseteq\mathbbm{j}^-\setminus\mathbbm{k}\}$ by Lemma \ref{L;Lemma3.11}, the first statement follows. The desired lemma follows from the combination of Equation \eqref{Eq;6}, Lemmas \ref{L;Lemma2.8}, \ref{L;Lemma3.11}, and the first statement.
\end{proof}
\begin{lem}\label{L;Lemma3.13}
Assume that $g\in[1,2^{n_-}]$, $(\mathbf{h},\mathbf{i},\mathbf{h}), (\mathbf{j}, \mathbf{k}, \mathbf{j})\in\mathbb{C}_g$.
Assume that $\mathbf{l}\in\mathbb{E}$ and $\mathbf{k}\setminus(\mathbf{h}\cup\mathbf{l})\!=\!\mathbf{i}\setminus(\mathbf{j}\cup\mathbf{l})\!=\!\mathbf{0}$.
Then $k_{(\mathbf{i}\cup\mathbf{m})\setminus\mathbf{l}}\!=\!k_{(\mathbf{k}\cup\mathbf{m})\setminus\mathbf{l}}$ for any $\mathbf{m}\!\in\!\{\mathbf{a}: \mathbf{a}\!\in\!\mathbb{E},\mathbbm{a}\!\subseteq\!\mathbbm{j}^-\setminus\mathbbm{k}\}$.
\end{lem}
\begin{proof}
The equation $\mathbbm{h}^-\setminus\mathbbm{i}=\mathbbm{j}^-\setminus\mathbbm{k}$ and Lemma \ref{L;Lemma2.5} imply that $k_{(\mathbf{i}\cup\mathbf{m})\setminus\mathbf{l}}=k_{\mathbf{i}\setminus\mathbf{l}}k_{\mathbf{m}\setminus\mathbf{l}}$ and
$k_{(\mathbf{k}\cup\mathbf{m})\setminus\mathbf{l}}=k_{\mathbf{k}\setminus\mathbf{l}}k_{\mathbf{m}\setminus\mathbf{l}}$ for any $\mathbf{m}\in\{\mathbf{a}: \mathbf{a}\in\mathbb{E},\mathbbm{a}\subseteq \mathbbm{j}^-\setminus\mathbbm{k}\}$. So it suffices to check that $k_{\mathbf{i}\setminus\mathbf{l}}=k_{\mathbf{k}\setminus\mathbf{l}}$. Notice that $(\mathbbm{h}\cap\mathbbm{j})^-\setminus(\mathbbm{k}\cup\mathbbm{l})=
(\mathbbm{h}\cap\mathbbm{j})^-\setminus(\mathbbm{i}\cup\mathbbm{k}\cup\mathbbm{l})=(\mathbbm{h}\cap\mathbbm{j})^-\setminus(\mathbbm{i}\cup\mathbbm{l})$
as $\mathbbm{h}^-\setminus\mathbbm{i}=\mathbbm{j}^-\setminus\mathbbm{k}$. As $\mathbbm{i}\subseteq\mathbbm{h}^\circ$, $\mathbbm{k}\subseteq\mathbbm{j}^\circ$, $p\nmid k_\mathbf{i}k_\mathbf{k}$, Lemmas \ref{L;Lemma2.5} and \ref{L;Lemma2.3} imply that $\mathbbm{i}\setminus\mathbbm{l}=(\mathbbm{i}\cap\mathbbm{j})\setminus\mathbbm{l}=((\mathbbm{h}\cap\mathbbm{j})^\circ\cap\mathbbm{i})\setminus\mathbbm{l}=
((\mathbbm{h}\cap\mathbbm{j})^\circ\cap\mathbbm{k})\setminus\mathbbm{l}=(\mathbbm{h}\cap\mathbbm{k})\setminus\mathbbm{l}=\mathbbm{k}\setminus\mathbbm{l}$.
So $k_{\mathbf{i}\setminus\mathbf{l}}=k_{\mathbf{k}\setminus\mathbf{l}}$ by Lemma \ref{L;Lemma2.5}. The desired lemma follows from the above discussion.
\end{proof}
\begin{lem}\label{L;Lemma3.14}
Assume that $g\in[1,2^{n_-}]$ and $(\mathbf{h}, \mathbf{i}, \mathbf{j})\in\mathbb{P}$. Assume that $(\mathbf{j}, \mathbf{k}, \mathbf{j})\in\mathbb{C}_g$ and $n_{\mathbf{i},\mathbf{k},\mathbf{j}}=0$. Then there is a unique $(\mathbf{h}, \mathbf{l}, \mathbf{h})\in\mathbb{C}_g$ such that $\mathbbm{l}\!=\!(\mathbbm{h}\cap\mathbbm{k})\cup(\mathbbm{h}^-\setminus\mathbbm{j})$. Moreover, $B_{\mathbf{h},\mathbf{i},\mathbf{j}}D_g\!=\!B_{\mathbf{h},\mathbf{i},\mathbf{j}}D_{\mathbf{j}, \mathbf{k},\mathbf{j}}=D_{\mathbf{h}, \mathbf{l},\mathbf{h}}B_{\mathbf{h},\mathbf{i},\mathbf{j}}=D_gB_{\mathbf{h},\mathbf{i},\mathbf{j}}\!\neq\! O$ for this unique $(\mathbf{h}, \mathbf{l}, \mathbf{h})$.
\end{lem}
\begin{proof}
By Lemma \ref{L;Lemma3.9}, there is a unique $(\mathbf{h}, \mathbf{l}, \mathbf{h})\in\mathbb{C}_g$ such that $\mathbbm{l}\!=\!(\mathbbm{h}\cap\mathbbm{k})\cup(\mathbbm{h}^-\setminus\mathbbm{j})$. Lemma \ref{L;Lemma3.9} implies that
$B_{\mathbf{h},\mathbf{i},\mathbf{j}}D_g=B_{\mathbf{h},\mathbf{i},\mathbf{j}}D_{\mathbf{j}, \mathbf{k},\mathbf{j}}$ and
$D_gB_{\mathbf{h},\mathbf{i},\mathbf{j}}=D_{\mathbf{h}, \mathbf{l},\mathbf{h}}B_{\mathbf{h},\mathbf{i},\mathbf{j}}$.
So it suffices to check that $B_{\mathbf{h},\mathbf{i},\mathbf{j}}D_{\mathbf{j}, \mathbf{k},\mathbf{j}}=D_{\mathbf{h}, \mathbf{l},\mathbf{h}}B_{\mathbf{h},\mathbf{i},\mathbf{j}}\neq O$. According to Lemma \ref{L;Lemma3.12}, notice that both $B_{\mathbf{h},\mathbf{i},\mathbf{j}}D_{\mathbf{j}, \mathbf{k},\mathbf{j}}$ and $D_{\mathbf{h}, \mathbf{l},\mathbf{h}}B_{\mathbf{h},\mathbf{i},\mathbf{j}}$ are unique $\F$-linear combinations of the matrices in $\{B_{\mathbf{h}, [\mathbf{h}, \mathbf{i},\mathbf{j},\mathbf{a}\cup\mathbf{k},\mathbf{j}], \mathbf{j}}: \mathbf{a}\in\mathbb{E}, \mathbbm{a}\subseteq\mathbbm{j}^-\setminus\mathbbm{k}\}$. Assume that $\mathbf{m}\in\mathbb{E}$ and $\mathbbm{m}\subseteq\mathbbm{j}^-\setminus\mathbbm{k}$. Let $c_{[\mathbf{h},\mathbf{i},\mathbf{j}, \mathbf{k}\cup\mathbf{m}, \mathbf{j}],1}, c_{[\mathbf{h},\mathbf{i},\mathbf{j}, \mathbf{k}\cup\mathbf{m}, \mathbf{j}],2}$ be the coefficients of $B_{\mathbf{h},[\mathbf{h}, \mathbf{i}, \mathbf{j}, \mathbf{k}\cup\mathbf{m}, \mathbf{j}], \mathbf{j}}$ in these $\F$-linear combinations that represent $B_{\mathbf{h},\mathbf{i},\mathbf{j}}D_{\mathbf{j},\mathbf{k},\mathbf{j}}, D_{\mathbf{h},\mathbf{l},\mathbf{h}}B_{\mathbf{h},\mathbf{i},\mathbf{j}}$, respectively. Lemma \ref{L;Lemma2.5} implies that $k_{\mathbf{j}\cap(\mathbf{k}\cup\mathbf{m})}=k_{\mathbf{k}\cup\mathbf{m}}=
k_{(\mathbf{k}\cup\mathbf{m})\setminus\mathbf{i}}k_{\mathbf{i}\cap(\mathbf{k}\cup\mathbf{m})}=
k_{(\mathbf{k}\cup\mathbf{m})\setminus\mathbf{i}}k_{\mathbf{i}\cap\mathbf{j}\cap(\mathbf{k}\cup\mathbf{m})}$. As $\mathbbm{j}^-\setminus\mathbbm{k}=\mathbbm{h}^-\setminus\mathbbm{l}$, Lemma \ref{L;Lemma2.5} implies that $k_{\mathbf{h}\cap(\mathbf{l}\cup\mathbf{m})}=k_{\mathbf{l}\cup\mathbf{m}}=
k_{(\mathbf{l}\cup\mathbf{m})\setminus\mathbf{i}}k_{\mathbf{i}\cap(\mathbf{l}\cup\mathbf{m})}=
k_{(\mathbf{l}\cup\mathbf{m})\setminus\mathbf{i}}k_{\mathbf{h}\cap\mathbf{i}\cap(\mathbf{l}\cup\mathbf{m})}$. As $p\nmid k_\mathbf{k}k_\mathbf{l}$, notice that $p\nmid
k_{(\mathbf{k}\cup\mathbf{m})\setminus\mathbf{i}}k_{(\mathbf{l}\cup\mathbf{m})\setminus\mathbf{i}}$ by Lemma \ref{L;Lemma2.5}.
As $\mathbf{h}\triangle\mathbf{j}\!\preceq\!\mathbf{i}$, $\mathbbm{k}\!\subseteq\!\mathbbm{j}^\circ, \mathbbm{l}\!\subseteq\!\mathbbm{h}^\circ$, $\mathbf{k}\setminus(\mathbf{h}\cup\mathbf{i})=\mathbf{l}\setminus(\mathbf{i}\cup\mathbf{j})=\mathbf{0}$.
By combining Equation \eqref{Eq;6}, Lemmas \ref{L;Lemma2.8}, \ref{L;Lemma3.11}, \ref{L;Lemma3.13},
$$c_{[\mathbf{h},\mathbf{i}, \mathbf{j},\mathbf{k}\cup\mathbf{m},\mathbf{j}],1}=
(\overline{-1})^{|\mathbbm{m}|}\overline{k_{(\mathbf{k}\cup\mathbf{m})\setminus\mathbf{i}}}^{-1}=(\overline{-1})^{|\mathbbm{m}|}\overline{k_{(\mathbf{l}\cup\mathbf{m})\setminus\mathbf{i}}}^{-1}
=c_{[\mathbf{h},\mathbf{i}, \mathbf{j},\mathbf{k}\cup\mathbf{m},\mathbf{j}],2}\neq\overline{0}.$$
As $B_{\mathbf{h}, [\mathbf{h}, \mathbf{i}, \mathbf{j}, \mathbf{k}\cup\mathbf{m}, \mathbf{j}], \mathbf{j}}$ is arbitrarily chosen from $\{B_{\mathbf{h}, [\mathbf{h}, \mathbf{i}, \mathbf{j}, \mathbf{a}\cup\mathbf{k}, \mathbf{j}], \mathbf{j}}\!:\! \mathbf{a}\in\mathbb{E},\mathbbm{a}\subseteq\mathbbm{j}^-\setminus\mathbbm{k}\}$, notice that $B_{\mathbf{h},\mathbf{i},\mathbf{j}}D_{\mathbf{j}, \mathbf{k},\mathbf{j}}=D_{\mathbf{h}, \mathbf{l},\mathbf{h}}B_{\mathbf{h},\mathbf{i},\mathbf{j}}\neq O$ by the combination of the displayed inequality,
Lemmas \ref{L;Lemma3.11}, and \ref{L;Lemma2.8}. The desired lemma follows from the above discussion.
\end{proof}
\begin{lem}\label{L;Lemma3.15}
Assume that $g\in[1, 2^{n_-}]$ and $(\mathbf{h}, \mathbf{i}, \mathbf{j})\in\mathbb{P}$. Then $B_{\mathbf{h}, \mathbf{i},\mathbf{j}}D_g=D_gB_{\mathbf{h}, \mathbf{i},\mathbf{j}}$. Moreover, $MD_g=D_gM$ for any $M\in\mathbb{T}$. In particular, $\{D_1, D_2,\ldots, D_{2^{n_-}}\}\subseteq\mathrm{Z}(\mathbb{T})$.
\end{lem}
\begin{proof}
As $g$ and $(\mathbf{h},\mathbf{i},\mathbf{j})$ are arbitrary elements in $[1, 2^{n_-}]$ and $\mathbb{P}$, respectively, the desired lemma follows from the combination of Lemmas \ref{L;Lemma3.6}, \ref{L;Lemma3.8}, \ref{L;Lemma3.14}, and \ref{L;Lemma2.8}.
\end{proof}
We next display four lemmas to check the equation $\mathrm{Bl}(\mathbb{T})=\{D_1, D_2,\ldots, D_{2^{n_-}}\}$.
\begin{lem}\label{L;Lemma3.16}
Assume that $g\in[1, 2^{n_-}]$ and $\mathbb{U}$ is a nonempty proper subset of $\mathbb{C}_g$. Assume that $M$ is the sum of all matrices in $\{D_{\mathbf{a},\mathbf{b}, \mathbf{a}}\!:\! (\mathbf{a},\mathbf{b}, \mathbf{a})\!\in\!\mathbb{U}\}$. Then $M\!\notin\!\mathrm{Z}(\mathbb{T})$.
\end{lem}
\begin{proof}
Assume that $(\mathbf{h}, \mathbf{i},\mathbf{h})\in\mathbb{C}_g\setminus\mathbb{U}$ and $(\mathbf{j}, \mathbf{k}, \mathbf{j})\in\mathbb{U}$. As $(\mathbf{h}, \mathbf{h}\triangle\mathbf{j},\mathbf{j})\in\mathbb{P}$, notice that $B_{\mathbf{h},\mathbf{h}\triangle\mathbf{j}, \mathbf{j}}$ is defined. As $\mathbbm{h}^-\setminus\mathbbm{i}=\mathbbm{j}^-\setminus\mathbbm{k}$, notice that $((\mathbbm{h}\triangle\mathbbm{j})\cap\mathbbm{j})^-\setminus\mathbbm{k}=\mathbbm{j}^-\setminus(\mathbbm{h}\cup\mathbbm{k})=\varnothing$
and $n_{\mathbf{h}\triangle\mathbf{j},\mathbf{k}, \mathbf{j}}=0$. So $B_{\mathbf{h},\mathbf{h}\triangle\mathbf{j}, \mathbf{j}}M=B_{\mathbf{h},\mathbf{h}\triangle\mathbf{j}, \mathbf{j}}E_\mathbf{j}^*M=B_{\mathbf{h},\mathbf{h}\triangle\mathbf{j}, \mathbf{j}}D_{\mathbf{j}, \mathbf{k},\mathbf{j}}=B_{\mathbf{h},\mathbf{h}\triangle\mathbf{j}, \mathbf{j}}D_g\neq O$ by combining Equations \eqref{Eq;4}, \eqref{Eq;2}, \eqref{Eq;6}, Lemmas \ref{L;Lemma3.5}, \ref{L;Lemma3.14}. By combining Equations \eqref{Eq;4}, \eqref{Eq;2}, \eqref{Eq;6}, Lemma \ref{L;Lemma3.5}, notice that $MB_{\mathbf{h},\mathbf{h}\triangle\mathbf{j}, \mathbf{j}}=ME_\mathbf{h}^*B_{\mathbf{h},\mathbf{h}\triangle\mathbf{j}, \mathbf{j}}=O$. Therefore $B_{\mathbf{h},\mathbf{h}\triangle\mathbf{j}, \mathbf{j}}M\neq MB_{\mathbf{h},\mathbf{h}\triangle\mathbf{j}, \mathbf{j}}$. The desired lemma follows from the above discussion.
\end{proof}
\begin{lem}\label{L;Lemma3.17}
Assume that $M\in\mathrm{Z}(\mathbb{T})$ and $M$ is an idempotent of $\mathbb{T}$. Then $M$ is an $\F$-linear combination of the matrices in  $\{D_{\mathbf{a},\mathbf{b},\mathbf{a}}:(\mathbf{a},\mathbf{b},\mathbf{a})\!\in\!\mathbb{P}, p\nmid k_\mathbf{b}\}$, where  all matrices in $\{D_{\mathbf{a},\mathbf{b},\mathbf{a}}: (\mathbf{a},\mathbf{b},\mathbf{a})\in\mathbb{P}, p\nmid k_\mathbf{b}\}$ are pairwise orthogonal idempotents of $\mathbb{T}$.
\end{lem}
\begin{proof}
By Lemmas \ref{L;Lemma2.14} and \ref{L;Lemma2.3}, all matrices in $\{D_{\mathbf{a},\mathbf{b},\mathbf{a}}: (\mathbf{a},\mathbf{b},\mathbf{a})\in\mathbb{P}, p\nmid k_\mathbf{b}\}$ are pairwise orthogonal idempotents of $\mathbb{T}$. For any $\mathbf{g}\in\mathbb{E}$ and $(\mathbf{h}, \mathbf{i}, \mathbf{h})\in\mathbb{P}$, Equations \eqref{Eq;4} and \eqref{Eq;2} imply that
$E_\mathbf{g}^*B_{\mathbf{h}, \mathbf{i}, \mathbf{h}}=\delta_{\mathbf{g}, \mathbf{h}}B_{\mathbf{h}, \mathbf{i}, \mathbf{h}}=B_{\mathbf{h}, \mathbf{i}, \mathbf{h}}E_\mathbf{g}^*$. So the combination of Lemmas \ref{L;Lemma2.10}, \ref{L;Lemma2.14}, Equations \eqref{Eq;5}, \eqref{Eq;6} implies that $M$ is an $\F$-linear combination of the matrices in $\{B_{\mathbf{a},\mathbf{b},\mathbf{a}}: (\mathbf{a},\mathbf{b},\mathbf{a})\in\mathbb{P}, p\mid k_\mathbf{b}\}\cup\{D_{\mathbf{a},\mathbf{b},\mathbf{a}}: (\mathbf{a},\mathbf{b},\mathbf{a})\in\mathbb{P}, p\nmid k_\mathbf{b}\}$.
The desired lemma follows from combining the above discussion, Lemmas \ref{L;Lemma2.9}, \ref{L;Lemma2.7}.
\end{proof}
\begin{lem}\label{L;Lemma3.18}
Assume that $g, h\in[1, 2^{n_-}]$. Then $D_gD_h=D_hD_g=\delta_{g,h}D_g$ and $D_g$ is a nonzero idempotent of $\mathbb{T}$. Moreover, $I$ is the sum of all matrices $D_1, D_2, \ldots, D_{2^{n_-}}$.
\end{lem}
\begin{proof}
The first statement follows from combining Equation \eqref{Eq;7}, Lemmas \ref{L;Lemma3.17}, \ref{L;Lemma2.14}. By Lemmas \ref{L;Lemma3.17} and \ref{L;Lemma2.14}, it is obvious to see that $I$ is the sum of all matrices in $\{D_{\mathbf{a},\mathbf{b},\mathbf{a}}: (\mathbf{a},\mathbf{b},\mathbf{a})\!\in\!\mathbb{P}, p\nmid k_\mathbf{b}\}$. The desired lemma follows from Equation \eqref{Eq;7}.
\end{proof}
\begin{lem}\label{L;Lemma3.19}
Assume that $g\in[1, 2^{n_-}]$. Then $D_g\in\mathrm{Bl}(\mathbb{T})$. Moreover, assume that $M\!\in\!\mathrm{Bl}(\mathbb{T})$. Then $M\!\!\in\!\!\{D_1, D_2, \ldots, D_{2^{n_-}}\}$. In
particular, $\mathrm{Bl}(\mathbb{T})\!\!=\!\!\{D_1, D_2, \ldots, D_{2^{n_-}}\}$.
\end{lem}
\begin{proof}
Assume that $D_g\notin\mathrm{Bl}(\mathbb{T})$. Lemmas \ref{L;Lemma3.15} and \ref{L;Lemma3.18} imply that $D_g$ is a sum of at least two pairwise distinct matrices in $\mathrm{Bl}(\mathbb{T})$. Assume that $h\in[2,\infty)$ and $D_g$ is the sum of some pairwise distinct matrices $N_1, N_2, \ldots, N_h$ in $\mathrm{Bl}(\mathbb{T})$. As $N_i\in\mathrm{Z}(\mathbb{T})\setminus\{O\}$ and $N_iN_j=\delta_{i, j}N_i$ for any $i, j\!\in\![1, h]$, the combination of Equation \eqref{Eq;7}, Lemmas \ref{L;Lemma3.17}, \ref{L;Lemma2.14} implies that each matrix in $\{N_1, N_2,\dots, N_h\}$ is the sum of all matrices in a proper subset of $\{D_{\mathbf{a},\mathbf{b},\mathbf{a}}: (\mathbf{a},\mathbf{b},\mathbf{a})\in\mathbb{C}_g\}$. As $N_i\in\mathrm{Z}(\mathbb{T})$ for any $i\in[1, h]$, this contradicts Lemma \ref{L;Lemma3.16}. This contradiction thus implies that $D_g\in\mathrm{Bl}(\mathbb{T})$. The first statement follows. As $g$ is arbitrarily chosen from $[1, 2^{n_-}]$, the first statement implies that $\{D_1, D_2, \ldots, D_{2^{n_-}}\}\subseteq\mathrm{Bl}(\mathbb{T})$. Lemma \ref{L;Lemma3.18} implies that $M$ is the sum of all matrices $MD_1, MD_2,\ldots, MD_{2^{n_-}}$. This implies that $M\in\{D_1, D_2, \ldots, D_{2^{n_-}}\}$ as $M\neq O$. The second statement follows. As $M$ is arbitrarily chosen from $\mathrm{Bl}(\mathbb{T})$, it is obvious that $\mathrm{Bl}(\mathbb{T})\subseteq\{D_1, D_2, \ldots, D_{2^{n_-}}\}\subseteq\mathrm{Bl}(\mathbb{T})$. The desired lemma follows.
\end{proof}
We are now ready to display the main result of this section and another corollary.
\begin{thm}\label{T;Idempotent}
All pairwise distinct block idempotents of $\mathbb{T}$ are exactly the matrices $D_1, D_2, \ldots, D_{2^{n_-}}$ and $|\mathrm{Bl}(\mathbb{T})|=2^{n_-}$. In particular, the decomposition of $\mathbb{T}$ into the direct sum of all pairwise distinct block algebras of $\mathbb{T}$ can be displayed by the equation
$$\mathbb{T}=\bigoplus_{g=1}^{2^{n_-}}\mathbb{T}D_g.$$
\end{thm}
\begin{proof}
The first statement follows from an application of Lemmas \ref{L;Lemma3.18} and \ref{L;Lemma3.19}. The first statement implies that $\mathbb{T}D_1, \mathbb{T}D_2, \ldots, \mathbb{T}D_{2^{n_-}}$ are exactly all pairwise distinct block algebras of $\mathbb{T}$. The desired theorem follows from the above discussion.
\end{proof}
\begin{cor}\label{C;Corollary3.21}
The number of all pairwise distinct block idempotents of $\mathbb{T}$ and the number of all pairwise distinct block algebras of $\mathbb{T}$ are independent of the choice of $\mathbf{x}$. Moreover, the following are equivalent: $|\mathrm{Bl}(\mathbb{T})|=1$; $n_-=0$; and $[1,n]^\circ\!=\![1,n]^+$.
\end{cor}
\begin{proof}
The first statement follows from Theorem \ref{T;Idempotent}. It is obvious that $n_-=0$ if and only if $[1,n]^\circ=[1,n]^+$.
The desired corollary follows from Theorem \ref{T;Idempotent}.
\end{proof}
We are now ready to end this section by displaying an example of Theorem \ref{T;Idempotent}.
\begin{eg}\label{E;Example3.22}
Assume that $n=|\mathbb{U}_1|=2$, $|\mathbb{U}_2|=3$, $\mathbf{g}=(0,1)$, $\mathbf{h}=(1,0)$. Then $\mathbb{E}=\{\mathbf{0}, \mathbf{g}, \mathbf{h}, \mathbf{1}\}$. If $p\neq2$, Theorem \ref{T;Idempotent} implies that $|\mathrm{Bl}(\mathbb{T})|=2$ and all pairwise distinct block idempotents of $\mathbb{T}$ are $D_{\mathbf{0}, \mathbf{0}, \mathbf{0}}+D_{\mathbf{g},\mathbf{g}, \mathbf{g}}+D_{\mathbf{h},\mathbf{0}, \mathbf{h}}+D_{\mathbf{1},\mathbf{g}, \mathbf{1}}$ and $D_{\mathbf{g}, \mathbf{0}, \mathbf{g}}+D_{\mathbf{1},\mathbf{0},\mathbf{1}}$. For the remaining case $p\!=\!2$, Theorem \ref{T;Idempotent} implies that $|\mathrm{Bl}(\mathbb{T})|\!=\!1$ and $\mathrm{Bl}(\mathbb{T})\!=\!\{I\}$.
\end{eg}
\section{$\F$-Dimension of block algebra of $\mathbb{T}$}
In this section, for any block algebra of $\mathbb{T}$, we get the $\F$-dimension of this block algebra of $\mathbb{T}$. Our strategy is to present an $\F$-basis for any block algebra of $\mathbb{T}$. We first display a sequence of lemmas to offer an $\F$-basis for any block algebra of $\mathbb{T}$. By Theorem \ref{T;Idempotent}, recall that $\mathbb{T}D_1, \mathbb{T}D_2,\ldots,\mathbb{T}D_{2^{n_-}}$ are exactly all block algebras of $\mathbb{T}$.
\begin{lem}\label{L;Lemma4.1}
Assume that $\mathbf{g}, \mathbf{h}, \mathbf{i}, \mathbf{j}, \mathbf{k}\in\mathbb{E}$. Assume that $n_{\mathbf{i},\mathbf{g}\cup\mathbf{j},\mathbf{i}}=n_{\mathbf{g}\cap\mathbf{i},\mathbf{j}, \mathbf{k}}=0$. Then $n_{(\mathbf{g}\triangle\mathbf{i})\cup\mathbf{k}, \mathbf{j}, \mathbf{i}}=0$. In particular, $n_{(\mathbf{g}\triangle\mathbf{i})\cup\mathbf{l}^+, \mathbf{j}, \mathbf{i}}=0$ for any $\mathbf{l}\in\mathbb{E}$.
\end{lem}
\begin{proof}
As $(((\mathbbm{g}\triangle\mathbbm{i})\cup\mathbbm{k})\cap\mathbbm{i})^-\setminus\mathbbm{j}=((\mathbbm{g}\cap\mathbbm{i}\cap\mathbbm{k})\cup(\mathbbm{i}\setminus\mathbbm{g}))^-\setminus\mathbbm{j}$ and $(\mathbbm{g}\cap\mathbbm{i}\cap\mathbbm{k})^-\setminus\mathbbm{j}=\varnothing$, notice that
$(((\mathbbm{g}\triangle\mathbbm{i})\cup\mathbbm{k})\cap\mathbbm{i})^-\setminus\mathbbm{j}=\mathbbm{i}^-\setminus(\mathbbm{g}\cup\mathbbm{j})$. The equation $n_{\mathbf{i},\mathbf{g}\cup\mathbf{j},\mathbf{i}}=0$ thus implies that $n_{(\mathbf{g}\triangle\mathbf{i})\cup\mathbf{k}, \mathbf{j}, \mathbf{i}}=0$. The first statement thus follows. Notice that $n_{\mathbf{g}\cap\mathbf{i}, \mathbf{j},\mathbf{l}^+}\!=\!0$ for any $\mathbf{l}\in\mathbb{E}$. The desired lemma follows from an application of the first statement.
\end{proof}
\begin{lem}\label{L;Lemma4.2}
Assume that $\mathbf{g},\mathbf{h},\mathbf{i}, \mathbf{j}, \mathbf{k},\mathbf{l}\in\mathbb{E}$ and $p\nmid k_{(\mathbf{g}\cap\mathbf{i}\cap\mathbf{k})\setminus\mathbf{h}}k_{(\mathbf{g}\cap\mathbf{i}\cap\mathbf{l})\setminus\mathbf{j}}$. Assume that $[\mathbf{g},\mathbf{h},\mathbf{i}, \mathbf{k}, \mathbf{i}]=[\mathbf{g},\mathbf{j},\mathbf{i}, \mathbf{l}, \mathbf{i}]$. Then $(\mathbf{g}\cap\mathbf{h}\cap\mathbf{i})^+=(\mathbf{g}\cap\mathbf{i}\cap\mathbf{j})^+$.
\end{lem}
\begin{proof}
According to Lemma \ref{L;Lemma2.5}, notice that $((\mathbbm{g}\cap\mathbbm{i})^\circ\cap\mathbbm{k})\setminus\mathbbm{h}=(\mathbbm{g}\cap\mathbbm{i}\cap\mathbbm{k})^-\setminus\mathbbm{h}$ and $((\mathbbm{g}\cap\mathbbm{i})^\circ\cap\mathbbm{l})\setminus\mathbbm{j}=(\mathbbm{g}\cap\mathbbm{i}\cap\mathbbm{l})^-\setminus\mathbbm{j}$. By Lemma \ref{L;Lemma2.3}, $[\mathbf{g},\mathbf{h},\mathbf{i}, \mathbf{k}, \mathbf{i}]=[\mathbf{g},\mathbf{j},\mathbf{i}, \mathbf{l}, \mathbf{i}]$ if and only if $(\mathbbm{g}\triangle\mathbbm{i})\cup((\mathbbm{g}\cap\mathbbm{i})^\circ\cap\mathbbm{h})
\cup(((\mathbbm{g}\cap\mathbbm{i})^\circ\cap\mathbbm{k})\setminus\mathbbm{h})=(\mathbbm{g}\triangle\mathbbm{i})\cup((\mathbbm{g}\cap\mathbbm{i})^\circ\cap\mathbbm{j})
\cup(((\mathbbm{g}\cap\mathbbm{i})^\circ\cap\mathbbm{l})\setminus\mathbbm{j})$. Therefore
$(\mathbbm{g}\cap\mathbbm{h}\cap\mathbbm{i})^+=(\mathbbm{g}\cap\mathbbm{i}\cap\mathbbm{j})^+$. The desired lemma follows from Lemma \ref{L;Lemma2.3}.
\end{proof}
\begin{lem}\label{L;Lemma4.3}
Assume that $g\!\in\![1, 2^{n_-}]$ and $\mathbf{h}, \mathbf{i}, \mathbf{j}\!\in\!\mathbb{E}$. Assume that $\mathbf{i}\preceq(\mathbf{h}\cap\mathbf{j})^+$, $(\mathbf{j},\mathbf{k},\mathbf{j})\!\in\!\mathbb{C}_g$, $n_{\mathbf{j}, \mathbf{h}\cup\mathbf{k},\mathbf{j}}\!=\!0$. Then $(\mathbf{h}, (\mathbf{h}\triangle\mathbf{j})\cup\mathbf{i}, \mathbf{j})\!\in\!\mathbb{P}$ and $B_{\mathbf{h}, (\mathbf{h}\triangle\mathbf{j})\cup\mathbf{i}, \mathbf{j}}D_g\neq O$. Moreover, assume that $\mathbf{l},\mathbf{m}, \mathbf{q}\in\mathbb{E}$, $\mathbf{m}\preceq(\mathbf{l}\cap\mathbf{q})^+$, $(\mathbf{q},\mathbf{r},\mathbf{q})\in\mathbb{C}_g$, $n_{\mathbf{q},\mathbf{l}\cup\mathbf{r},\mathbf{q}}=0$. Then the equations $\mathbf{h}\!=\!\mathbf{l}$, $\mathbf{i}\!=\!\mathbf{m}$, $\mathbf{j}\!=\!\mathbf{q}$ hold together if and only if $B_{\mathbf{h}, (\mathbf{h}\triangle\mathbf{j})\cup\mathbf{i}, \mathbf{j}}D_g\!\!=\!\!B_{\mathbf{l}, (\mathbf{l}\triangle\mathbf{q})\cup\mathbf{m}, \mathbf{q}}D_g$.
\end{lem}
\begin{proof}
As $\mathbf{h}\triangle\mathbf{j}\preceq(\mathbf{h}\triangle\mathbf{j})\cup\mathbf{i}\preceq\mathbf{h}\mathbf{j}$, notice that $(\mathbf{h}, (\mathbf{h}\triangle\mathbf{j})\cup\mathbf{i}, \mathbf{j})\in\mathbb{P}$ and $B_{\mathbf{h}, (\mathbf{h}\triangle\mathbf{j})\cup\mathbf{i}, \mathbf{j}}$ is defined. Notice that $n_{(\mathbf{h}\triangle\mathbf{j})\cup\mathbf{i}, \mathbf{k}, \mathbf{j}}=0$ by Lemma \ref{L;Lemma4.1}. Lemma \ref{L;Lemma3.14} thus implies that $B_{\mathbf{h}, (\mathbf{h}\triangle\mathbf{j})\cup\mathbf{i}, \mathbf{j}}D_g\neq O$. The first statement follows. The combination of Theorem \ref{T;Idempotent}, Equations \eqref{Eq;4}, \eqref{Eq;2} implies that $E_\mathbf{l}^*B_{\mathbf{h}, (\mathbf{h}\triangle\mathbf{j})\cup\mathbf{i},\mathbf{j}}D_gE_\mathbf{q}^*=\delta_{\mathbf{h},\mathbf{l}}\delta_{\mathbf{j},\mathbf{q}}B_{\mathbf{h}, (\mathbf{h}\triangle\mathbf{j})\cup\mathbf{i}, \mathbf{j}}D_g$. The first statement thus implies that $B_{\mathbf{h}, (\mathbf{h}\triangle\mathbf{j})\cup\mathbf{i}, \mathbf{j}}D_g=B_{\mathbf{l}, (\mathbf{l}\triangle\mathbf{q})\cup\mathbf{m}, \mathbf{q}}D_g$ only if $\mathbf{h}=\mathbf{l}$ and $\mathbf{j}=\mathbf{q}$. Assume that $\mathbf{h}=\mathbf{l}$ and $\mathbf{j}=\mathbf{q}$. Notice that $(\mathbf{h}\cap\mathbf{j}\cap((\mathbf{h}\triangle\mathbf{j})\cup\mathbf{i}))^+=\mathbf{i}$ and $(\mathbf{h}\cap\mathbf{j}\cap((\mathbf{h}\triangle\mathbf{j})\cup\mathbf{m}))^+=\mathbf{m}$ by Lemma \ref{L;Lemma2.3}.
For any $\mathbf{r}, \mathbf{s}\in\mathbb{E}$, Lemma \ref{L;Lemma2.5} implies that $p\nmid k_\mathbf{r}$ only if $p\nmid k_{\mathbf{r}\cap\mathbf{s}}$. The combination of the first statement, Equations \eqref{Eq;7}, \eqref{Eq;6}, Lemmas \ref{L;Lemma2.8}, \ref{L;Lemma4.2} thus implies that $B_{\mathbf{h}, (\mathbf{h}\triangle\mathbf{j})\cup\mathbf{i}, \mathbf{j}}D_g=B_{\mathbf{h}, (\mathbf{h}\triangle\mathbf{j})\cup\mathbf{m}, \mathbf{j}}D_g$ only if $\mathbf{i}=\mathbf{m}$. The desired lemma follows as the proof of the remaining direction is trivial.
\end{proof}
\begin{lem}\label{L;Lemma4.4}
Assume that $\mathbf{g},\mathbf{h},\mathbf{i}, \mathbf{j}, \mathbf{k}\in\mathbb{E}$, $p\nmid k_\mathbf{k}$,
$n_{\mathbf{i},\mathbf{g}\cup\mathbf{h}\cup\mathbf{k},\mathbf{i}}=n_{\mathbf{i},\mathbf{g}\cup\mathbf{j}\cup\mathbf{k},\mathbf{i}}=0$. Assume that
$n_{\mathbf{h},\mathbf{k},\mathbf{i}}\!=\!n_{\mathbf{i},\mathbf{k},\mathbf{j}}$. Then $[\mathbf{g},\mathbf{h},\mathbf{i}, \mathbf{k}, \mathbf{i}]\!=\![\mathbf{g},\mathbf{j},\mathbf{i}, \mathbf{k}, \mathbf{i}]$ if and only if $(\mathbf{g}\cap\mathbf{h}\cap\mathbf{i})^+\!=\!(\mathbf{g}\cap\mathbf{i}\cap\mathbf{j})^+$.
\end{lem}
\begin{proof}
As $\mathbbm{i}^-\setminus(\mathbbm{g}\cup\mathbbm{k})\subseteq\mathbbm{h}\cap\mathbbm{j}$, notice that $(\mathbbm{h}\cap\mathbbm{i})^-\setminus\mathbbm{k}=(\mathbbm{i}^-\setminus(\mathbbm{g}\cup\mathbbm{k}))\cup
((\mathbbm{g}\cap\mathbbm{h}\cap\mathbbm{i})^-\setminus\mathbbm{k})$ and $(\mathbbm{i}\cap\mathbbm{j})^-\setminus\mathbbm{k}=(\mathbbm{i}^-\setminus(\mathbbm{g}\cup\mathbbm{k}))\cup
((\mathbbm{g}\cap\mathbbm{i}\cap\mathbbm{j})^-\setminus\mathbbm{k})$. As $(\mathbbm{h}\cap\mathbbm{i})^-\setminus\mathbbm{k}=(\mathbbm{i}\cap\mathbbm{j})^-\setminus\mathbbm{k}$, notice that $(\mathbbm{g}\cap\mathbbm{h}\cap\mathbbm{i})^-\setminus\mathbbm{k}=(\mathbbm{g}\cap\mathbbm{i}\cap\mathbbm{j})^-\setminus\mathbbm{k}$. As $(\mathbbm{g}\cap\mathbbm{h}\cap\mathbbm{i})^+\setminus\mathbbm{k}=(\mathbbm{g}\cap\mathbbm{h}\cap\mathbbm{i})^+$ and $(\mathbbm{g}\cap\mathbbm{i}\cap\mathbbm{j})^+\setminus\mathbbm{k}=(\mathbbm{g}\cap\mathbbm{i}\cap\mathbbm{j})^+$ by Lemma \ref{L;Lemma2.5}, notice that $(\mathbbm{g}\cap\mathbbm{h}\cap\mathbbm{i})^+=(\mathbbm{g}\cap\mathbbm{i}\cap\mathbbm{j})^+$ if and only if $((\mathbbm{g}\cap\mathbbm{i})^\circ\cap\mathbbm{h})\setminus\mathbbm{k}=((\mathbbm{g}\cap\mathbbm{i})^\circ\cap\mathbbm{j})\setminus\mathbbm{k}$.
By Lemma \ref{L;Lemma2.3}, notice that
$(\mathbbm{g}\triangle\mathbbm{i})\cup(((\mathbbm{g}\cap\mathbbm{i})^\circ\cap\mathbbm{h})\setminus\mathbbm{k})
\cup((\mathbbm{g}\cap\mathbbm{i})^\circ\cap\mathbbm{k})=(\mathbbm{g}\triangle\mathbbm{i})\cup(((\mathbbm{g}\cap\mathbbm{i})^\circ\cap\mathbbm{j})\setminus\mathbbm{k})
\cup((\mathbbm{g}\cap\mathbbm{i})^\circ\cap\mathbbm{k})$ if and only if $[\mathbf{g},\mathbf{h},\mathbf{i}, \mathbf{k}, \mathbf{i}]=[\mathbf{g},\mathbf{j},\mathbf{i}, \mathbf{k}, \mathbf{i}]$. The desired lemma follows as Lemma \ref{L;Lemma2.3} implies that $[\mathbf{g},\mathbf{h},\mathbf{i}, \mathbf{k}, \mathbf{i}]=[\mathbf{g},\mathbf{j},\mathbf{i}, \mathbf{k}, \mathbf{i}]$ if and only if $(\mathbf{g}\cap\mathbf{h}\cap\mathbf{i})^+=(\mathbf{g}\cap\mathbf{i}\cap\mathbf{j})^+$.
\end{proof}
\begin{lem}\label{L;Lemma4.5}
Assume that $(\mathbf{g},\mathbf{h},\mathbf{i})\in\mathbb{P}$ and $\mathbf{j}, \mathbf{k}\in\mathbb{E}$. Assume that $p\nmid k_{(\mathbf{g}\cap\mathbf{h})\cap\mathbf{i}\cap\mathbf{j}}$ and
$(\mathbbm{g}\cap\mathbbm{h}\cap\mathbbm{i}\cap\mathbbm{k})^-\subseteq(\mathbbm{g}\cap\mathbbm{h}\cap\mathbbm{i}\cap\mathbbm{j})^-$. Then $k_{(\mathbf{j}\cup\mathbf{k})\setminus((\mathbf{g}\triangle\mathbf{i})\cup(\mathbf{g}\cap\mathbf{h}\cap\mathbf{i})^+)}=k_{(\mathbf{j}\cup\mathbf{k})\setminus\mathbf{h}}k_{(\mathbf{g}\cap\mathbf{h})\cap\mathbf{i}\cap\mathbf{j}}$.
\end{lem}
\begin{proof}
As $((\mathbbm{j}\cup\mathbbm{k})\cap(\mathbbm{g}\cap\mathbbm{h}\cap\mathbbm{i})^-)\setminus((\mathbbm{g}\triangle\mathbbm{i})\cup(\mathbbm{g}\cap\mathbbm{h}\cap\mathbbm{i})^+)
=(\mathbbm{g}\cap\mathbbm{h}\cap\mathbbm{i}\cap\mathbbm{j})^-$ and $\mathbbm{g}\triangle\mathbbm{i}\subseteq\mathbbm{h}\subseteq(\mathbbm{g}\triangle\mathbbm{i})\cup(\mathbbm{g}\cap\mathbbm{i})^\circ$, notice that $(\mathbbm{j}\cup\mathbbm{k})\setminus\mathbbm{h}=(\mathbbm{j}\cup\mathbbm{k})\setminus((\mathbbm{g}\triangle\mathbbm{i})\cup(\mathbbm{g}\cap\mathbbm{h}\cap\mathbbm{i})^\circ)$ and $(\mathbbm{j}\cup\mathbbm{k})\setminus((\mathbbm{g}\triangle\mathbbm{i})\cup(\mathbbm{g}\cap\mathbbm{h}\cap\mathbbm{i})^+)=
((\mathbbm{j}\cup\mathbbm{k})\setminus\mathbbm{h})\cup(\mathbbm{g}\cap\mathbbm{h}\cap\mathbbm{i}\cap\mathbbm{j})^-$. Lemma \ref{L;Lemma2.5} thus implies that $k_{(\mathbf{j}\cup\mathbf{k})\setminus((\mathbf{g}\triangle\mathbf{i})\cup(\mathbf{g}\cap\mathbf{h}\cap\mathbf{i})^+)}\!\!=\!\!k_{(\mathbf{j}\cup\mathbf{k})\setminus\mathbf{h}}k_{(\mathbf{g}\cap\mathbf{h})\cap\mathbf{i}\cap\mathbf{j}}$.
The desired lemma follows.
\end{proof}
\begin{lem}\label{L;Lemma4.6}
Assume that $g\!\in\![1, 2^{n_-}]$ and $(\mathbf{h}, \mathbf{i}, \mathbf{j})\!\in\!\mathbb{P}$. Assume that $(\mathbf{j}, \mathbf{k}, \mathbf{j})\in\mathbb{C}_g$ and $n_{\mathbf{i},\mathbf{k},\mathbf{j}}\!\!=\!\!0$. Then $(\mathbf{h}\!, \! (\mathbf{h}\triangle\mathbf{j})\!\cup\!(\mathbf{h}\cap\mathbf{i}\cap\mathbf{j})^+\!,\!\mathbf{j})\!\in\!\mathbb{P}$ and $B_{\mathbf{h},\mathbf{i},\mathbf{j}}D_g\!\!=\!\!\overline{k_{\mathbf{h}\cap\mathbf{i}\cap\mathbf{k}}}B_{\mathbf{h},(\mathbf{h}\triangle\mathbf{j})\!\cup\!(\mathbf{h}\cap\mathbf{i}\cap\mathbf{j})^+,\mathbf{j}}D_g\!\neq\! O $.
\end{lem}
\begin{proof}
As $\mathbbm{j}\setminus\mathbbm{h}\subseteq\mathbbm{i}$, notice that $n_{\mathbf{i}, \mathbf{k},\mathbf{j}}=n_{\mathbf{j},\mathbf{h}\cup\mathbf{k},\mathbf{j}}=0$. By Lemmas \ref{L;Lemma3.14} and \ref{L;Lemma4.3},
both $B_{\mathbf{h},\mathbf{i},\mathbf{j}}D_g$ and $B_{\mathbf{h},(\mathbf{h}\triangle\mathbf{j})\cup(\mathbf{h}\cap\mathbf{i}\cap\mathbf{j})^+,\mathbf{j}}D_g$ are defined nonzero matrices. By Lemma \ref{L;Lemma4.1}, notice that $n_{\mathbf{i},\mathbf{l},\mathbf{j}}=n_{\mathbf{j},\mathbf{h}\cup\mathbf{l},\mathbf{j}}=n_{(\mathbf{h}\triangle\mathbf{j})\cup(\mathbf{h}\cap\mathbf{i}\cap\mathbf{j})^+, \mathbf{l}, \mathbf{j}}=0$ and $(\mathbbm{h}\cap\mathbbm{i}\cap\mathbbm{j}\cap\mathbbm{m})^-\subseteq(\mathbbm{h}\cap\mathbbm{i}\cap\mathbbm{j}\cap\mathbbm{k})^-$
for any $\mathbf{l},\mathbf{m}\in\mathbb{E}$ and $\mathbf{k}\preceq\mathbf{l}$. Lemma \ref{L;Lemma3.9} thus implies that $B_{\mathbf{h},\mathbf{i},\mathbf{j}}D_g=B_{\mathbf{h},\mathbf{i},\mathbf{j}}D_{\mathbf{j}, \mathbf{k}, \mathbf{j}}$ and
$B_{\mathbf{h},(\mathbf{h}\triangle\mathbf{j})\cup(\mathbf{h}\cap\mathbf{i}\cap\mathbf{j})^+,\mathbf{j}}D_g\!\!=\!\!
B_{\mathbf{h},(\mathbf{h}\triangle\mathbf{j})\cup(\mathbf{h}\cap\mathbf{i}\cap\mathbf{j})^+,\mathbf{j}}D_{\mathbf{j},\mathbf{k},\mathbf{j}}$.
By combining Lemmas \ref{L;Lemma3.12}, \ref{L;Lemma2.5}, \ref{L;Lemma4.4}, both $B_{\mathbf{h},\mathbf{i},\mathbf{j}}D_{\mathbf{j}, \mathbf{k}, \mathbf{j}}$ and $B_{\mathbf{h},(\mathbf{h}\triangle\mathbf{j})\cup(\mathbf{h}\cap\mathbf{i}\cap\mathbf{j})^+,\mathbf{j}}D_{\mathbf{j},\mathbf{k},\mathbf{j}}$
are unique $\F$-linear combinations of the matrices in $\{B_{\mathbf{h}, [\mathbf{h}, \mathbf{i},\mathbf{j},\mathbf{a}\cup\mathbf{k},\mathbf{j}], \mathbf{i}}\!:\! \mathbf{a}\!\in\!\mathbb{E}, \mathbbm{a}\subseteq\mathbbm{j}^-\setminus\mathbbm{k}\}$. Assume that $\mathbf{l}\!\in\!\mathbb{E}$ and $\mathbbm{l}\!\subseteq\!\mathbbm{j}^-\!\setminus\!\mathbbm{k}$. Let $c_{[\mathbf{h}, \mathbf{i}, \mathbf{j}, \mathbf{k}\cup\mathbf{l},\mathbf{j}],1}$, $c_{[\mathbf{h}, \mathbf{i}, \mathbf{j}, \mathbf{k}\cup\mathbf{l},\mathbf{j}],2}$ be the coefficients of $B_{\mathbf{h},[\mathbf{h}, \mathbf{i}, \mathbf{j}, \mathbf{k}\cup\mathbf{l},\mathbf{j}],\mathbf{j}}$ in the two
$\F$-linear combinations that represent $B_{\mathbf{h},\mathbf{i},\mathbf{j}}D_{\mathbf{j}, \mathbf{k}, \mathbf{j}}$, $B_{\mathbf{h},(\mathbf{h}\triangle\mathbf{j})\cup(\mathbf{h}\cap\mathbf{i}\cap\mathbf{j})^+,\mathbf{j}}D_{\mathbf{j},\mathbf{k},\mathbf{j}}$, respectively. Assume that $\mathbf{m}=\mathbf{i}$ or $\mathbf{m}=(\mathbf{h}\triangle\mathbf{j})\cup(\mathbf{h}\cap\mathbf{i}\cap\mathbf{j})^+$.
As $p\nmid k_\mathbf{k}$, notice that $p\nmid k_{(\mathbf{k}\cup\mathbf{l})\setminus\mathbf{m}}k_{\mathbf{h}\cap\mathbf{i}\cap\mathbf{k}}$,
$k_{\mathbf{j}\cap(\mathbf{k}\cup\mathbf{l})}=k_{\mathbf{k}\cup\mathbf{l}}=k_{(\mathbf{k}\cup\mathbf{l})\setminus\mathbf{m}}k_{(\mathbf{k}\cup\mathbf{l})\cap\mathbf{m}}=
k_{(\mathbf{k}\cup\mathbf{l})\setminus\mathbf{m}}k_{\mathbf{j}\cap(\mathbf{k}\cup\mathbf{l})\cap\mathbf{m}}$, $k_{(\mathbf{h}\cap\mathbf{i})\cap\mathbf{j}\cap\mathbf{k}}\!=\!k_{\mathbf{h}\cap\mathbf{i}\cap\mathbf{k}}$ by Lemma \ref{L;Lemma2.5}.
The combination of Equation \eqref{Eq;6}, Lemmas \ref{L;Lemma2.8}, \ref{L;Lemma3.11}, \ref{L;Lemma4.4}, \ref{L;Lemma4.5} thus implies that
\begin{align*}
c_{[\mathbf{h},\mathbf{i}, \mathbf{j},\mathbf{k}\cup\mathbf{l},\mathbf{j}],1}=
(\overline{-1})^{|\mathbbm{l}|}\overline{k_{(\mathbf{k}\cup\mathbf{l})\setminus\mathbf{i}}}^{-1}&=(\overline{-1})^{|\mathbbm{l}|}\overline{k_{(\mathbf{k}\cup\mathbf{l})\setminus
(\mathbf{(\mathbf{h}\triangle\mathbf{j})\cup(\mathbf{h}\cap\mathbf{i}\cap\mathbf{j})^+})}}^{-1}
\overline{k_{\mathbf{h}\cap\mathbf{i}\cap\mathbf{k}}}\\
&=\overline{k_{\mathbf{h}\cap\mathbf{i}\cap\mathbf{k}}}c_{[\mathbf{h},\mathbf{i}, \mathbf{j},\mathbf{k}\cup\mathbf{l},\mathbf{j}],2}.
\end{align*}
Therefore $B_{\mathbf{h},\mathbf{i},\mathbf{j}}D_{\mathbf{j}, \mathbf{k},\mathbf{j}}=\overline{k_{\mathbf{h}\cap\mathbf{i}\cap\mathbf{k}}}B_{\mathbf{h},(\mathbf{h}\triangle\mathbf{j})\cup(\mathbf{h}\cap\mathbf{i}\cap\mathbf{j})^+,\mathbf{j}}D_{\mathbf{j}, \mathbf{k},\mathbf{j}}$ as $B_{\mathbf{h}, [\mathbf{h}, \mathbf{i}, \mathbf{j}, \mathbf{k}\cup\mathbf{l}, \mathbf{j}], \mathbf{j}}$ is arbitrarily chosen from $\{B_{\mathbf{h}, [\mathbf{h}, \mathbf{i}, \mathbf{j}, \mathbf{a}\cup\mathbf{k}, \mathbf{j}], \mathbf{j}}: \mathbf{a}\in\mathbb{E},\ \mathbbm{a}\subseteq\mathbbm{j}^-\setminus\mathbbm{k}\}$. The desired lemma follows.
\end{proof}
\begin{lem}\label{L;Lemma4.7}
Assume that $g\in[1, 2^{n_-}]$ and $\mathbf{h}, \mathbf{i}, \mathbf{j}\in\mathbb{E}$. Assume that $\mathbf{i}\preceq(\mathbf{h}\cap\mathbf{j})^+$. Then $(\mathbf{h}, (\mathbf{h}\triangle\mathbf{j})\cup\mathbf{i}, \mathbf{j})\in\mathbb{P}$ and the block algebra $\mathbb{T}D_g$ of $\mathbb{T}$ has a nonempty $\F$-linearly independent subset
$\{\!B_{\mathbf{a}, (\mathbf{a}\triangle\mathbf{c})\cup\mathbf{b}, \mathbf{c}}\!D_g\!\!:\! \mathbf{a}, \mathbf{b},\mathbf{c}\!\in\!\mathbb{E}, \mathbf{b}\!\preceq\!(\mathbf{a}\!\cap\!\mathbf{c})^+, \exists\ (\mathbf{c}, \mathbf{d}, \mathbf{c})\!\in\!\mathbb{C}_g, n_{\mathbf{c}, \mathbf{a}\cup\mathbf{d}, \mathbf{c}}\!\!=\!\!0\}$.
\end{lem}
\begin{proof}
Let $\mathbb{U}\!=\!\{B_{\mathbf{a}, (\mathbf{a}\triangle\mathbf{c})\cup\mathbf{b}, \mathbf{c}}D_g: \mathbf{a}, \mathbf{b},\mathbf{c}\!\in\!\mathbb{E}, \mathbf{b}\preceq(\mathbf{a}\cap\mathbf{c})^+, \exists\ (\mathbf{c}, \mathbf{d}, \mathbf{c})\!\in\!\mathbb{C}_g, n_{\mathbf{c}, \mathbf{a}\cup\mathbf{d}, \mathbf{c}}\!=\!0\}$. There is no loss to assume that $(\mathbf{j}, \mathbf{k}, \mathbf{j})\in\mathbb{C}_g$. So $\mathbb{U}\neq\varnothing$ as $E_\mathbf{j}^*D_g=B_{\mathbf{j},\mathbf{0},\mathbf{j}}D_g\in\mathbb{U}$. By Lemma \ref{L;Lemma4.3}, each matrix in $\mathbb{U}$ is a defined nonzero matrix.
Assume that $n_{\mathbf{j}, \mathbf{h}\cup\mathbf{k},\mathbf{j}}=0$. Let $\mathbb{V}=\mathbb{U}\setminus\{B_{\mathbf{h},(\mathbf{h}\triangle\mathbf{j})\cup\mathbf{i}, \mathbf{j}}D_g\}$ and $\mathbb{W}=\{B_{\mathbf{h},(\mathbf{h}\triangle\mathbf{j})\cup\mathbf{a}, \mathbf{j}}D_g:\mathbf{a}\in\mathbb{E}\setminus\{\mathbf{i}\}, \mathbf{a}\preceq(\mathbf{h}\cap\mathbf{j})^+\}$.
As $B_{\mathbf{h},(\mathbf{h}\triangle\mathbf{j})\cup\mathbf{i}, \mathbf{j}}D_g$ is arbitrarily chosen from $\mathbb{U}$, it suffices to check that $B_{\mathbf{h},(\mathbf{h}\triangle\mathbf{j})\cup\mathbf{i}, \mathbf{j}}D_g$ is not an $\F$-linear combination of the matrices in $\mathbb{V}$. Assume that $B_{\mathbf{h},(\mathbf{h}\triangle\mathbf{j})\cup\mathbf{i}, \mathbf{j}}D_g$ is an $\F$-linear combination of the matrices in $\mathbb{V}$. The combination of Theorem \ref{T;Idempotent}, Equations \eqref{Eq;4}, \eqref{Eq;2} implies that $B_{\mathbf{h},(\mathbf{h}\triangle\mathbf{j})\cup\mathbf{i}, \mathbf{j}}D_g$ is an $\F$-linear combination of the matrices in $\mathbb{W}$. Lemma \ref{L;Lemma2.3} implies that $(\mathbf{h}\cap\mathbf{j}\cap((\mathbf{h}\triangle\mathbf{j})\cup\mathbf{l}))^+=\mathbf{l}$ for any $\mathbf{l}\in\mathbb{E}$ and $\mathbf{l}\preceq\!(\mathbf{h}\cap\mathbf{j})^+$. For any $\mathbf{l},\mathbf{m}\in\mathbb{E}$, Lemma \ref{L;Lemma2.5} implies that $p\nmid k_\mathbf{l}$ only if $p\nmid k_{\mathbf{l}\cap\mathbf{m}}$. Therefore the combination of Equations \eqref{Eq;7}, \eqref{Eq;6}, Lemmas \ref{L;Lemma2.8}, \ref{L;Lemma4.2} implies that $B_{\mathbf{h},(\mathbf{h}\triangle\mathbf{j})\cup\mathbf{i}, \mathbf{j}}D_g$ is not an $\F$-linear combination of the matrices in $\mathbb{W}$. This thus yields an obvious contradiction. The desired lemma follows from this contradiction.
\end{proof}
\begin{lem}\label{L;Lemma4.8}
Assume that $g\in[1, 2^{n_-}]$ and $\mathbf{h}, \mathbf{i}, \mathbf{j}\in\mathbb{E}$. Assume that $\mathbf{i}\preceq(\mathbf{h}\cap\mathbf{j})^+$. Then
$(\mathbf{h}, (\mathbf{h}\triangle\mathbf{j})\cup\mathbf{i},\mathbf{j})\in\mathbb{P}$ and the block algebra $\mathbb{T}D_g$ of $\mathbb{T}$ has the following $\F$-basis
$$\{B_{\mathbf{a}, (\mathbf{a}\triangle\mathbf{c})\cup\mathbf{b}, \mathbf{c}}D_g: \mathbf{a}, \mathbf{b},\mathbf{c}\in\mathbb{E}, \mathbf{b}\preceq(\mathbf{a}\cap\mathbf{c})^+, \exists\ (\mathbf{c}, \mathbf{d}, \mathbf{c})\in\mathbb{C}_g, n_{\mathbf{c}, \mathbf{a}\cup\mathbf{d}, \mathbf{c}}=0\}.$$
\end{lem}
\begin{proof}
By combining Lemmas \ref{L;Lemma2.8}, \ref{L;Lemma3.6}, \ref{L;Lemma3.8}, each matrix in $\mathbb{T}D_g$ is also an $\F$-linear combination of the matrices in $\{B_{\mathbf{a},\mathbf{b},\mathbf{c}}D_g: (\mathbf{a},\mathbf{b},\mathbf{c})\in\mathbb{P},\exists\  (\mathbf{c},\mathbf{d},\mathbf{c})\in\mathbb{C}_g, n_{\mathbf{b},\mathbf{d},\mathbf{c}}=0\}$. For any $\mathbf{k}, \mathbf{l}, \mathbf{m}, \mathbf{q}\in\mathbb{E}$ and $\mathbbm{m}\setminus\mathbbm{k}\subseteq\mathbbm{l}$, notice that $n_{\mathbf{l}, \mathbf{q}, \mathbf{m}}=0$ only if $n_{\mathbf{m}, \mathbf{k}\cup\mathbf{q}, \mathbf{m}}=0$. So Lemma \ref{L;Lemma4.6} implies that each matrix in $\mathbb{T}D_g$ is an $\F$-linear combination of the matrices in $\{B_{\mathbf{a}, (\mathbf{a}\triangle\mathbf{c})\cup\mathbf{b}, \mathbf{c}}D_g: \mathbf{a}, \mathbf{b},\mathbf{c}\in\mathbb{E}, \mathbf{b}\preceq(\mathbf{a}\cap\mathbf{c})^+, \exists\ (\mathbf{c}, \mathbf{d}, \mathbf{c})\in\mathbb{C}_g, n_{\mathbf{c}, \mathbf{a}\cup\mathbf{d}, \mathbf{c}}=0\}$. By Lemma \ref{L;Lemma4.7}, $\{B_{\mathbf{a}, (\mathbf{a}\triangle\mathbf{c})\cup\mathbf{b}, \mathbf{c}}D_g\!: \mathbf{a}, \mathbf{b},\mathbf{c}\!\in\!\mathbb{E}, \mathbf{b}\preceq(\mathbf{a}\cap\mathbf{c})^+, \exists\ (\mathbf{c}, \mathbf{d}, \mathbf{c})\!\in\!\mathbb{C}_g, n_{\mathbf{c}, \mathbf{a}\cup\mathbf{d}, \mathbf{c}}\!=\!0\}$ is indeed an $\F$-basis of the block algebra $\mathbb{T}D_g$ of $\mathbb{T}$. The desired lemma follows.
\end{proof}
We are now ready to display the main result of this section and another corollary.
\begin{thm}\label{T;Dimension}
Assume that $g\!\in\![1, 2^{n_-}]$ and $(\mathbf{h},\mathbf{i},\mathbf{h})\!\in\!\mathbb{C}_g$. Then $n_{\mathbf{h}, \mathbf{i}, \mathbf{h}}$ is independent of the choice of $(\mathbf{h},\mathbf{i},\mathbf{h})$ and the block algebra $\mathbb{T}D_g$ of $\mathbb{T}$ has the following $\F$-dimension
$$4^{n-n_+-n_{\mathbf{h},\mathbf{i}, \mathbf{h}}}5^{n_+}.$$
\end{thm}
\begin{proof}
As Lemmas \ref{L;Lemma4.8} and \ref{L;Lemma4.3} imply that the $\F$-dimension of $\mathbb{T}D_g$ is equal to
$|\{(\mathbf{a}, \mathbf{b},\mathbf{c}):\mathbf{a}, \mathbf{b},\mathbf{c}\in\mathbb{E}, \mathbf{b}\preceq(\mathbf{a}\cap\mathbf{c})^+, \exists\ (\mathbf{c}, \mathbf{d}, \mathbf{c})\in\mathbb{C}_g, n_{\mathbf{c}, \mathbf{a}\cup\mathbf{d}, \mathbf{c}}=0\}|$, it suffices to check that $|\{(\mathbf{a}, \mathbf{b},\mathbf{c}):\mathbf{a}, \mathbf{b},\mathbf{c}\in\mathbb{E}, \mathbf{b}\preceq(\mathbf{a}\cap\mathbf{c})^+, \exists\ (\mathbf{c}, \mathbf{d}, \mathbf{c})\in\mathbb{C}_g, n_{\mathbf{c}, \mathbf{a}\cup\mathbf{d}, \mathbf{c}}=0\}|$ can be computed by the displayed formula. Notice that $n_{\mathbf{h},\mathbf{i}, \mathbf{h}}$ is independent of the choice of $(\mathbf{h},\mathbf{i},\mathbf{h})$ and $n_{\mathbf{h},\mathbf{i},\mathbf{h}}\leq n_-$. Let $j=|[1,n]\setminus[1,n]^\circ|$. So $n=j+n_++n_-$. Assume that $(\mathbf{k},\mathbf{l}, \mathbf{m})\in\{(\mathbf{a}, \mathbf{b},\mathbf{c}):\mathbf{a}, \mathbf{b},\mathbf{c}\in\mathbb{E}, \mathbf{b}\preceq(\mathbf{a}\cap\mathbf{c})^+, \exists\ (\mathbf{c}, \mathbf{d}, \mathbf{c})\in\mathbb{C}_g, n_{\mathbf{c}, \mathbf{a}\cup\mathbf{d}, \mathbf{c}}=0\}$ and $(\mathbf{m}, \mathbf{q}, \mathbf{m})\in\mathbb{C}_g$. As $\mathbbm{q}\subseteq\mathbbm{m}^\circ$ and $p\nmid k_\mathbf{q}$, notice that $\mathbbm{q}\subseteq[1,n]^-$ by Lemma \ref{L;Lemma2.5}.
Moreover, $\mathbbm{h}^-\setminus\mathbbm{i}=\mathbbm{m}^-\setminus\mathbbm{q}\subseteq\mathbbm{k}$ and $\mathbbm{l}\subseteq(\mathbbm{k}\cap\mathbbm{m})^+$ by Lemma \ref{L;Lemma3.5}. Lemma \ref{L;Lemma2.3} shows that  $|\{(\mathbf{a}, \mathbf{b},\mathbf{c}):\mathbf{a}, \mathbf{b},\mathbf{c}\in\mathbb{E},\ \mathbf{b}\preceq(\mathbf{a}\cap\mathbf{c})^+, \exists\ (\mathbf{c}, \mathbf{d}, \mathbf{c})\in\mathbb{C}_g, n_{\mathbf{c}, \mathbf{a}\cup\mathbf{d}, \mathbf{c}}=0\}|$ is equal to $$\sum_{r=0}^{n_+}\sum_{s=0}^{n_--n_{\mathbf{h}, \mathbf{i},\mathbf{h}}}\sum_{t=0}^j\sum_{u=0}^r\sum_{v=0}^s\sum_{w=0}^t{n_+\choose r}{n_--n_{\mathbf{h}, \mathbf{i},\mathbf{h}}\choose s}{j\choose t}{r\choose u}{s\choose v}{t\choose w}2^{n-n_{\mathbf{h}, \mathbf{i},\mathbf{h}}-r-s-t+u}.$$ So $|\{(\mathbf{a}, \mathbf{b},\mathbf{c})\!:\!\mathbf{a}, \mathbf{b},\mathbf{c}\in\mathbb{E},\mathbf{b}\preceq(\mathbf{a}\cap\mathbf{c})^+, \exists\ (\mathbf{c}, \mathbf{d}, \mathbf{c})\in\mathbb{C}_g, n_{\mathbf{c}, \mathbf{a}\cup\mathbf{d}, \mathbf{c}}\!=\!0\}|\!=\!4^{n_--n_{\mathbf{h}, \mathbf{i},\mathbf{h}}+j}5^{n_+}$ by the Binomial Theorem. The desired theorem follows since $n=j+n_++n_-$.
\end{proof}
\begin{cor}\label{C;Corokkary4.10}
Assume that $g\!\in\![1,2^{n_-}]$. Then the $\F$-dimension of the block algebra $\mathbb{T}D_g$ of $\mathbb{T}$ is independent of the choice of $\mathbf{x}$ and is also divisible by the $5$-power $5^{n_+}$.
\end{cor}
\begin{proof}
The desired corollary follows from the displayed formula in Theorem \ref{T;Dimension}.
\end{proof}
We are now ready to conclude this section by listing an example of Theorem \ref{T;Dimension}.
\begin{eg}\label{E;Example4.11}
Assume that $n=|\mathbb{U}_1|=2$ and $|\mathbb{U}_2|=3$. If $p\neq2$, Example \ref{E;Example3.22} implies that $\mathbb{T}$ has exactly two block algebras. If $p\neq2$, Theorem \ref{T;Dimension} implies that the $\F$-dimensions of the only two distinct block algebras of $\mathbb{T}$ are equal to sixteen and four, respectively. For the remaining case $p=2$, Example \ref{E;Example3.22} and Theorem \ref{T;Dimension} also imply that the $\F$-dimension of the unique block algebra of $\mathbb{T}$ is equal to twenty.
\end{eg}
\section{Center of block algebra of $\mathbb{T}$}
In this section, for any block algebra of $\mathbb{T}$, we get an $\F$-basis and the $\F$-dimension of the center of this block algebra of $\mathbb{T}$. For any block algebra of $\mathbb{T}$, we also get an $\F$-basis, the nilpotent index, the $\F$-dimension of the Jacobson radical of the center of this block algebra of $\mathbb{T}$. As a preparation, we first display a sequence of lemmas to offer an $\F$-basis for the center of an arbitrary chosen block algebra of $\mathbb{T}$. By Theorem \ref{T;Idempotent}, recall that $\mathbb{T}D_1, \mathbb{T}D_2, \ldots, \mathbb{T}D_{2^{n_-}}$ are exactly all block algebras of $\mathbb{T}$.
\begin{lem}\label{L;Lemma5.1}
Assume that $g\in[1,2^{n_-}]$, $\mathbf{h}\in\mathbb{E}$, $(\mathbf{i}, \mathbf{j},\mathbf{i})\in\mathbb{C}_g$, $n_{\mathbf{h},\mathbf{j},\mathbf{i}}>0$. Then $n_{\mathbf{h},\mathbf{l},\mathbf{k}}\!>\!0$ for any
$(\mathbf{k},\mathbf{l},\mathbf{k})\!\in\!\mathbb{C}_g$. Moreover, assume that $(\mathbf{1}, \mathbf{h}, \mathbf{1})\!\in\!\mathbb{P}$. Then $C_\mathbf{h}D_g\!=\!O$.
\end{lem}
\begin{proof}
As $\mathbbm{i}^-\setminus\mathbbm{j}=\mathbbm{k}^-\setminus\mathbbm{l}$ for any $(\mathbf{k},\mathbf{l},\mathbf{k})\in\mathbb{C}_g$, notice that $n_{\mathbf{h},\mathbf{l},\mathbf{k}}=n_{\mathbf{h},\mathbf{j},\mathbf{i}}>0$ for any $(\mathbf{k},\mathbf{l},\mathbf{k})\in\mathbb{C}_g$. The first statement follows. By Equation \eqref{Eq;5} and Lemma \ref{L;Lemma3.8}, $C_\mathbf{h}D_g$ is an $\F$-linear combination of the matrices in $\{B_{\mathbf{a},\mathbf{a}\cap\mathbf{h},\mathbf{a}}D_g:\exists\ (\mathbf{a},\mathbf{b},\mathbf{a})\in\mathbb{C}_g\}$. As $n_{\mathbf{h}, \mathbf{l}, \mathbf{k}}>0$ for any $(\mathbf{k}, \mathbf{l}, \mathbf{k})\in\mathbb{C}_g$, the desired lemma follows from Lemma \ref{L;Lemma3.6}.
\end{proof}
\begin{lem}\label{L;Lemma5.2}
Assume that $g\in[1,2^{n_-}]$, $\mathbf{h}\in\mathbb{E}$, $(\mathbf{i}, \mathbf{j}, \mathbf{i})\in\mathbb{C}_g$, $n_{\mathbf{h}, \mathbf{j}, \mathbf{i}}=0$. Then there is $\mathbf{k}\in\mathbb{E}$ such that $\mathbf{h}^+=\mathbf{k}^+$, $(\mathbf{k}, \mathbf{j}, \mathbf{k})\!\in\!\mathbb{C}_g$, and $n_{\mathbf{h}, \mathbf{j}, \mathbf{k}}=0$. In particular, if $\mathbf{l}\in\mathbb{E}$ and $\mathbf{l}\preceq\mathbf{1}^+$, then there is $\mathbf{m}\in\mathbb{E}$ such that $\mathbf{l}=\mathbf{l}^+=\mathbf{m}^+$, $(\mathbf{m}, \mathbf{j}, \mathbf{m})\in\mathbb{C}_g$, and $n_{\mathbf{l}, \mathbf{j},\mathbf{m}}=0$.
\end{lem}
\begin{proof}
By Lemma \ref{L;Lemma2.3}, there is a unique $\mathbf{k}\in\mathbb{E}$ such that $\mathbbm{k}=\mathbbm{h}^+\cup\mathbbm{i}^-$. Notice that $\mathbf{h}^+=\mathbf{k}^+$, $\mathbbm{j}\subseteq\mathbbm{i}^-=\mathbbm{k}^-$, $\mathbbm{i}^-\setminus\mathbbm{j}=\mathbbm{k}^-\setminus\mathbbm{j}$, and $n_{\mathbf{h}, \mathbf{j}, \mathbf{k}}=n_{\mathbf{h}, \mathbf{j}, \mathbf{i}}=0$ by Lemmas \ref{L;Lemma2.3} and \ref{L;Lemma2.5}. The first statement thus follows from Lemma \ref{L;Lemma2.5}. As $n_{\mathbf{i}, \mathbf{j}, \mathbf{l}}=0$ if $\mathbf{l}\in\mathbb{E}$ and $\mathbbm{l}\subseteq[1,n]^+$, the desired lemma follows from the first statement and Lemma \ref{L;Lemma2.3}.
\end{proof}
\begin{lem}\label{L;Lemma5.3}
Assume that $g\in[1,2^{n_-}]$, $\mathbf{h}\in\mathbb{E}$, $\mathbf{h}\preceq\mathbf{1}^+$. Then $(\mathbf{1}, \mathbf{h}, \mathbf{1})\in\mathbb{P}$ and $C_\mathbf{h}D_g\neq O$. Moreover, assume that $\mathbf{i}\in\mathbb{E}$ and $\mathbf{i}\preceq\mathbf{1}^+$. Then $\mathbf{h}=\mathbf{i}$ if and only if $C_\mathbf{h}D_g\!=\!C_\mathbf{i}D_g$. In particular, $|\{C_\mathbf{a}D_g:\mathbf{a}\in\mathbb{E}, \mathbf{a}\preceq\mathbf{1}^+\}|\!=\!|\{\mathbf{a}: \mathbf{a}\in\mathbb{E}, \mathbf{a}\preceq\mathbf{1}^+\}|\!=\!2^{n_+}$.
\end{lem}
\begin{proof}
By Lemma \ref{L;Lemma2.3}, $(\mathbf{1}, \mathbf{h}, \mathbf{1})\in\mathbb{P}$ and $C_\mathbf{h}D_g$ is a defined matrix. Assume that $(\mathbf{j},\mathbf{k},\mathbf{j})\in\mathbb{C}_g$. By Lemma \ref{L;Lemma5.2}, there is $\mathbf{l}\in\mathbb{E}$ such that $\mathbf{h}=\mathbf{l}^+$, $(\mathbf{l}, \mathbf{k}, \mathbf{l})\in\mathbb{C}_g$, and $n_{\mathbf{h},\mathbf{k},\mathbf{l}}=0$. As $k_{\mathbf{h}\setminus\mathbf{l}}=1$ by Lemma \ref{L;Lemma2.5}, notice that
$E_\mathbf{l}^*C_\mathbf{h}D_g=B_{\mathbf{l},\mathbf{h},\mathbf{l}}D_g\neq O$ by combining Equations \eqref{Eq;5}, \eqref{Eq;4}, \eqref{Eq;2}, Lemmas \ref{L;Lemma2.3}, \ref{L;Lemma3.14}. The first statement follows. In particular, notice that $C_\mathbf{i}D_g\neq O$ by the first statement. By Lemma \ref{L;Lemma5.2}, there is $\mathbf{m}\in\mathbb{E}$ such that $\mathbf{i}=\mathbf{m}^+$, $(\mathbf{m}, \mathbf{k}, \mathbf{m})\in\mathbb{C}_g$, and $n_{\mathbf{i}, \mathbf{k}, \mathbf{m}}=0$. As $k_{\mathbf{i}\setminus\mathbf{m}}=1$ by Lemma \ref{L;Lemma2.5}, notice that
$E_\mathbf{m}^*C_\mathbf{i}D_g=B_{\mathbf{m},\mathbf{i},\mathbf{m}}D_g\neq O$ by combining Equations \eqref{Eq;5}, \eqref{Eq;4}, \eqref{Eq;2}, Lemmas \ref{L;Lemma2.3}, \ref{L;Lemma3.14}. Assume that
$C_\mathbf{h}D_g=C_\mathbf{i}D_g$. This thus implies that $E_\mathbf{l}^*C_\mathbf{h}D_g\!=\!E_\mathbf{l}^*C_\mathbf{i}D_g\!=\!\overline{k_{\mathbf{i}\setminus\mathbf{l}}}B_{\mathbf{l}, \mathbf{i}\cap\mathbf{l}, \mathbf{l}}D_g$ and $E_\mathbf{m}^*C_\mathbf{i}D_g=E_\mathbf{m}^*C_\mathbf{h}D_g=\overline{k_{\mathbf{h}\setminus\mathbf{m}}}B_{\mathbf{m}, \mathbf{h}\cap\mathbf{m}, \mathbf{m}}D_g$ by combining Equations \eqref{Eq;5}, \eqref{Eq;4}, \eqref{Eq;2}. This thus implies that $p\nmid k_{\mathbf{h}\setminus\mathbf{m}}k_{\mathbf{i}\setminus\mathbf{l}}$. As $\mathbf{h}=\mathbf{l}^+$ and $\mathbf{i}=\mathbf{m}^+$, Lemmas \ref{L;Lemma2.5} and \ref{L;Lemma2.3} imply that $\mathbf{h}=\mathbf{i}$. The second statement thus follows as the proof of the remaining direction is trivial. The desired lemma follows as the final equation follows from the second statement and Lemma \ref{L;Lemma2.3}.
\end{proof}
\begin{lem}\label{L;Lemma5.4}
Assume that $\mathbf{g}, \mathbf{h}, \mathbf{i}\!\in\!\mathbb{E}$, $p\!\nmid\! k_{\mathbf{g}\!\setminus\!\mathbf{h}}k_{\mathbf{g}\cap\mathbf{h}\cap\mathbf{i}}$, $n_{\mathbf{g}, \mathbf{i}, \mathbf{h}}\!\!=\!\!0$. Then
$k_{\mathbf{g}\!\setminus\!\mathbf{h}}k_{\mathbf{g}\cap\mathbf{h}\cap\mathbf{i}}\!\!=\!\!k_{\mathbf{g}\setminus\mathbf{g}^+}$.
\end{lem}
\begin{proof}
As $(\mathbbm{g}\cap\mathbbm{h})^-\setminus\mathbbm{i}=\varnothing$, notice that $\mathbbm{g}^\circ\setminus\mathbbm{g}^+=\mathbbm{g}^-=(\mathbbm{g}^-\setminus\mathbbm{h})\cup((\mathbbm{g}\cap\mathbbm{h})^-\cap\mathbbm{i})$. Lemmas \ref{L;Lemma2.3} and \ref{L;Lemma2.5} imply that $k_{\mathbf{g}\setminus\mathbf{h}}k_{\mathbf{g}\cap\mathbf{h}\cap\mathbf{i}}=k_{\mathbf{g}\setminus\mathbf{g}^+}$. The desired lemma follows.
\end{proof}
\begin{lem}\label{L;Lemma5.5}
Assume that $g\in[1,2^{n_-}]$, $(\mathbf{1}, \mathbf{h}, \mathbf{1})\in\mathbb{P}$, and there is $(\mathbf{i}, \mathbf{j}, \mathbf{i})\in\mathbb{C}_g$ satisfying the equation $n_{\mathbf{h}, \mathbf{j}, \mathbf{i}}=0$. Then $(\mathbf{1},\mathbf{h}^+,  \mathbf{1})\in\mathbb{P}$ and $C_\mathbf{h}D_g\!=\!\overline{k_{\mathbf{h}\setminus\mathbf{h}^+}}C_{\mathbf{h}^+}D_g\neq O$.
\end{lem}
\begin{proof}
According to Lemmas \ref{L;Lemma2.5} and \ref{L;Lemma5.3}, notice that $p\nmid k_{\mathbf{h}\setminus\mathbf{h}^+}$ and $C_{\mathbf{h}^+}D_g$ is a defined nonzero matrix. Lemma \ref{L;Lemma2.5} implies that $p\mid k_{\mathbf{h}^+\setminus\mathbf{k}}$ or $k_{\mathbf{h}^+\setminus\mathbf{k}}=1$ for any $\mathbf{k}\in\mathbb{E}$. For any $\mathbf{k}\in\mathbb{E}$, Lemma \ref{L;Lemma2.5} implies that $k_{\mathbf{h}^+\setminus\mathbf{k}}=1$ if and only if $p\nmid k_{\mathbf{h}\setminus\mathbf{k}}$. Let $\mathbb{U}=\{(\mathbf{a}, \mathbf{b}, \mathbf{a}):(\mathbf{a}, \mathbf{b}, \mathbf{a})\in\mathbb{C}_g, p\nmid k_{\mathbf{h}\setminus\mathbf{a}}\}$. As $n_{\mathbf{h},\mathbf{l},\mathbf{k}}=n_{\mathbf{h},\mathbf{j},\mathbf{i}}=0$ for any $(\mathbf{k}, \mathbf{l}, \mathbf{k})\in\mathbb{C}_g$, the combination of Equation \eqref{Eq;5}, Lemmas \ref{L;Lemma3.8}, \ref{L;Lemma4.6}, \ref{L;Lemma2.3}, \ref{L;Lemma2.5}, \ref{L;Lemma5.4} gives the equation
\begin{align*}
C_\mathbf{h}D_g&=\sum_{\mathbf{k}\in\mathbb{E}}\overline{k_{\mathbf{h}\setminus\mathbf{k}}}B_{\mathbf{k}, \mathbf{h}\cap\mathbf{k},\mathbf{k}}D_g\\
&=\sum_{(\mathbf{k}, \mathbf{l},\mathbf{k})\in\mathbb{U}}\overline{k_{\mathbf{h}\setminus\mathbf{k}}}B_{\mathbf{k}, \mathbf{h}\cap\mathbf{k},\mathbf{k}}D_g\\
&=\sum_{(\mathbf{k}, \mathbf{l},\mathbf{k})\in\mathbb{U}}\overline{k_{\mathbf{h}\setminus\mathbf{k}}}\overline{k_{\mathbf{h}\cap\mathbf{k}\cap\mathbf{l}}}
B_{\mathbf{k}, (\mathbf{h}\cap\mathbf{k})^+,\mathbf{k}}D_g\\
&=\sum_{(\mathbf{k}, \mathbf{l},\mathbf{k})\in\mathbb{U}}\overline{k_{\mathbf{h}\setminus\mathbf{h}^+}}B_{\mathbf{k}, \mathbf{h}^+\cap\mathbf{k},\mathbf{k}}D_g=\overline{k_{\mathbf{h}\setminus\mathbf{h}^+}}C_{\mathbf{h}^+}D_g.
\end{align*}
The desired lemma follows from the displayed equation and the above discussion.
\end{proof}
\begin{lem}\label{L;Lemma5.6}
Assume that $g\in[1,2^{n_-}]$, $\mathbf{h}\in\mathbb{E}$, $\mathbf{h}\preceq\mathbf{1}^+$. Then $(\mathbf{1}, \mathbf{h}, \mathbf{1})\in\mathbb{P}$ and the center $\mathrm{Z}(\mathbb{T}D_g)$ of the block algebra $\mathbb{T}D_g$ of $\mathbb{T}$ has a nonempty $\F$-linearly independent subset $\{C_\mathbf{a}D_g:\mathbf{a}\in\mathbb{E}, \mathbf{a}\preceq\mathbf{1}^+\}$.
\end{lem}
\begin{proof}
Let $\mathbb{U}\!=\!\{C_\mathbf{a}D_g: \mathbf{a}\in\mathbb{E}, \mathbf{a}\preceq\mathbf{1}^+\}$. So $\varnothing\neq\mathbb{U}\subseteq\mathrm{Z}(\mathbb{T}D_g)\setminus\{O\}$ by combining Theorem \ref{T;Idempotent}, Lemmas \ref{L;Lemma2.10}, \ref{L;Lemma5.3}. Assume that $O$ is a nonzero $\F$-linear combination of the matrices in $\mathbb{U}$. For any $\mathbf{i}\in\mathbb{E}$ and $\mathbf{i}\preceq\mathbf{1}^+$, let $c_\mathbf{i}$ be the coefficient of $C_\mathbf{i}D_g$ in this $\F$-linear combination that represents $O$. Let $\mathbb{V}=\{\mathbf{a}:\mathbf{a}\in\mathbb{E}, \mathbf{a}\preceq\mathbf{1}^+, c_\mathbf{a}\neq\overline{0}\}$. By Lemma \ref{L;Lemma2.3}, let $\mathbf{i}$ be a minimal element in $\mathbb{V}$ with respect to the partial order $\preceq$. Let $\mathbb{W}\!=\!\{C_\mathbf{a}D_g:\mathbf{a}\!\in\!\mathbb{V}\}\setminus\{C_\mathbf{i}D_g\}$. So $C_\mathbf{i}D_g$ is a nonzero $\F$-linear combination of the matrices in $\mathbb{W}$. Assume that $(\mathbf{j},\mathbf{k},\mathbf{j})\in\mathbb{C}_g$. By Lemma \ref{L;Lemma5.2}, there is $\mathbf{l}\in\mathbb{E}$ such that $\mathbf{i}=\mathbf{l}^+, (\mathbf{l}, \mathbf{k}, \mathbf{l})\in\mathbb{C}_g, \text{and}\ n_{\mathbf{i}, \mathbf{k}, \mathbf{l}}=0$. Notice that  $E_\mathbf{l}^*C_\mathbf{i}D_g$ is a nonzero $\F$-linear combination of the matrices in $\{E_\mathbf{l}^*C_\mathbf{a}D_g: C_\mathbf{a}D_g\in\mathbb{W}\}$. Notice that $k_{\mathbf{i}\setminus\mathbf{l}}=1$ by Lemma \ref{L;Lemma2.5}. The combination of Equations \eqref{Eq;5}, \eqref{Eq;4}, \eqref{Eq;2}, Lemmas \ref{L;Lemma2.3}, \ref{L;Lemma3.14} thus implies that the nonzero matrix $B_{\mathbf{l}, \mathbf{i}, \mathbf{l}}D_g$ is a nonzero $\F$-linear combination of the matrices in
$\{\overline{k_{\mathbf{a}\setminus\mathbf{l}}}B_{\mathbf{l},\mathbf{a}\cap\mathbf{l},\mathbf{l}}D_g:C_\mathbf{a}D_g\in\mathbb{W}\}$. So there is $\mathbf{m}\in\mathbb{V}\setminus\{\mathbf{i}\}$ such that $p\nmid k_{\mathbf{m}\setminus\mathbf{l}}$. Lemma \ref{L;Lemma2.5} thus implies that $\mathbf{m}\preceq\mathbf{l}^+$. This thus implies that $\mathbf{m}\preceq\mathbf{i}$. This contradicts the choice of $\mathbf{i}$. The desired lemma follows from this contradiction.
\end{proof}
\begin{lem}\label{L;Lemma5.7}
Assume that $g\in[1,2^{n_-}], \mathbf{h}\in\mathbb{E}, \mathbf{h}\preceq\mathbf{1}^+$. Then $(\mathbf{1}, \mathbf{h}, \mathbf{1})\!\in\!\mathbb{P}$ and the center $\mathrm{Z}(\mathbb{T}D_g)$ of the block algebra $\mathbb{T}D_g$ of $\mathbb{T}$ has an $\F$-basis $\{C_\mathbf{a}D_g:\mathbf{a}\in\mathbb{E}, \mathbf{a}\preceq\!\mathbf{1}^+\}$.
\end{lem}
\begin{proof}
Notice that $\mathrm{Z}(\mathbb{T}D_g)=\mathrm{Z}(\mathbb{T})D_g$ by Theorem \ref{T;Idempotent}. So the combination of Lemmas \ref{L;Lemma2.10}, \ref{L;Lemma5.1}, \ref{L;Lemma5.5} implies that each matrix in $\mathrm{Z}(\mathbb{T}D_g)$ is indeed an $\F$-linear combination of the matrices in $\{C_\mathbf{a}D_g:\mathbf{a}\in\mathbb{E},\mathbf{a}\preceq\mathbf{1}^+\}$. Lemma \ref{L;Lemma5.6} thus implies that $\mathrm{Z}(\mathbb{T}D_g)$ has an $\F$-basis $\{C_\mathbf{a}D_g:\mathbf{a}\in\mathbb{E},\mathbf{a}\preceq\mathbf{1}^+\}$. The desired lemma follows.
\end{proof}
We next introduce a necessary lemma to complete the preparation of this section.
\begin{lem}\label{L;Lemma5.8}
Assume that $g\in[1,2^{n_-}]$, $\mathbf{h}\in\mathbb{E}$, $\mathbf{h}\preceq\mathbf{1}^+$. Then $(\mathbf{1}, \mathbf{h}, \mathbf{1})\in\mathbb{P}$ and the Jacobson radical $\mathrm{Rad}(\mathrm{Z}(\mathbb{T}D_g))$ of the center of the block algebra $\mathbb{T}D_g$ of $\mathbb{T}$ has an \ \ $\F$-basis $\{C_\mathbf{a}D_g: \mathbf{a}\in\mathbb{E}\setminus\{\mathbf{0}\}, \mathbf{a}\preceq\mathbf{1}^+\}$. Moreover, the nilpotent index of $\mathrm{Rad}(\mathrm{Z}(\mathbb{T}D_g))$ is equal to $n_++1$. Furthermore, $\mathrm{Z}(\mathbb{T}D_g)/\mathrm{Rad}(\mathrm{Z}(\mathbb{T}D_g))\cong\mathrm{M}_1(\F)$ as $\F$-algebras.
\end{lem}
\begin{proof}
Notice that $\mathrm{Rad}(\mathrm{Z}(\mathbb{T}D_g))=\mathrm{Rad}(\mathrm{Z}(\mathbb{T}))D_g$ by Theorem \ref{T;Idempotent} and Lemma \ref{L;Lemma2.6}. So $\{C_\mathbf{a}D_g: \mathbf{a}\in\mathbb{E}\setminus\{\mathbf{0}\}, \mathbf{a}\preceq\mathbf{1}^+\}\subseteq\mathrm{Rad}(\mathrm{Z}(\mathbb{T}D_g))\setminus\{O\}$ by combining Lemmas \ref{L;Lemma2.5}, \ref{L;Lemma2.11}, \ref{L;Lemma5.3}. By combining Lemmas \ref{L;Lemma2.11}, \ref{L;Lemma5.1}, \ref{L;Lemma5.5}, each matrix in $\mathrm{Rad}(\mathrm{Z}(\mathbb{T}D_g))$ is an $\F$-linear combination of the matrices in $\{C_\mathbf{a}D_g: \mathbf{a}\in\mathbb{E}\setminus\{\mathbf{0}\}, \mathbf{a}\preceq\mathbf{1}^+\}$. The first statement follows from Lemma \ref{L;Lemma5.7}. The first statement and Lemma \ref{L;Lemma2.3} thus imply that $\mathrm{Rad}(\mathrm{Z}(\mathbb{T}D_g))=\{O\}$ if and only if $n_+=0$. Assume that $n_+>0$ and $[1,n]^+=\{i_1, i_2, \ldots, i_{n_+}\}$. For any $j\in[1, n_+]$, let $\mathbf{k}(j)$ be the $n$-tuple in $\mathbb{E}$ whose unique nonzero entry is the $i_j$th-entry. For any $j\in[1, n_+]$, the first statement thus implies that $C_{\mathbf{k}(j)}D_g\in\mathrm{Rad}(\mathrm{Z}(\mathbb{T}D_g))\setminus\{O\}$. By combining Theorem \ref{T;Idempotent}, Lemmas \ref{L;Lemma2.10}, \ref{L;Lemma2.5}, \ref{L;Lemma5.3}, $C_{\mathbf{k}(1)}D_gC_{\mathbf{k}(2)}D_g\cdots C_{\mathbf{k}(n_+)}D_g\in\mathrm{Rad}(\mathrm{Z}(\mathbb{T}D_g))\setminus\{O\}$. As Theorem \ref{T;Idempotent} implies that
$\mathrm{Rad}(\mathrm{Z}(\mathbb{T}D_g))\subseteq\mathrm{Rad}(\mathrm{Z}(\mathbb{T}))$, the second statement follows from Lemma \ref{L;Lemma2.11}. The desired lemma follows from Lemma \ref{L;Lemma5.7} and the first statement.
\end{proof}
We are now ready to display the main result of this section and another corollary.
\begin{thm}\label{T;Center}
Assume that $g\in[1,2^{n_-}], \mathbf{h}\in\mathbb{E}, \mathbf{h}\preceq\mathbf{1}^+$. Then $(\mathbf{1}, \mathbf{h}, \mathbf{1})\!\in\!\mathbb{P}$ and the center $\mathrm{Z}(\mathbb{T}D_g)$ of the block algebra $\mathbb{T}D_g$ of $\mathbb{T}$ has an $\F$-basis $\{C_\mathbf{a}D_g:\mathbf{a}\in\mathbb{E}, \mathbf{a}\preceq\mathbf{1}^+\}$. In particular, the $\F$-dimension of $\mathrm{Z}(\mathbb{T}D_g)$ is equal to $2^{n_+}$. Moreover, $\mathrm{Rad}(\mathrm{Z}(\mathbb{T}D_g))$ has an $\F$-basis $\{C_\mathbf{a}D_g\!:\!\mathbf{a}\in\mathbb{E}\setminus\{\mathbf{0}\}, \mathbf{a}\preceq\mathbf{1}^+\}$ and the nilpotent index of $\mathrm{Rad}(\mathrm{Z}(\mathbb{T}D_g))$ is equal to $n_++1$. In particular, the $\F$-dimension of $\mathrm{Rad}(\mathrm{Z}(\mathbb{T}D_g))$ is equal to $2^{n_+}-1$.
\end{thm}
\begin{proof}
By Lemmas \ref{L;Lemma5.7} and \ref{L;Lemma5.8}, it suffices to determine the $\F$-dimensions of  $\mathrm{Z}(\mathbb{T}D_g)$ and $\mathrm{Rad}(\mathrm{Z}(\mathbb{T}D_g))$. The formulae of the $\F$-dimensions of $\mathrm{Z}(\mathbb{T}D_g)$ and $\mathrm{Rad}(\mathrm{Z}(\mathbb{T}D_g))$ can be deduced by combining Lemmas \ref{L;Lemma5.7}, \ref{L;Lemma5.3}, \ref{L;Lemma5.8}.
The desired theorem follows.
\end{proof}
\begin{cor}\label{C;Corollary5.10}
Assume that $g\in[1,2^{n_-}]$. Then the $\F$-dimension of the center $\mathrm{Z}(\mathbb{T}D_g)$ of the block algebra $\mathbb{T}D_g$ of $\mathbb{T}$ is independent of the choice of $\mathbf{x}$. Moreover, the nilpotent index and the $\F$-dimension of $\mathrm{Rad}(\mathrm{Z}(\mathbb{T}D_g))$ are independent of the choice of $\mathbf{x}$. Furthermore, assume that $h\in[1,2^{n_-}]$. Then $\mathrm{Z}(\mathbb{T}D_g)\cong\mathrm{Z}(\mathbb{T}D_h)\ \text{as $\F$-algebras.}$
\end{cor}
\begin{proof}
By combining Theorems \ref{T;Center}, \ref{T;Idempotent}, Lemma \ref{L;Lemma2.10}, notice that the $\F$-linear map from $\mathrm{Z}(\mathbb{T}D_g)$ to $\mathrm{Z}(\mathbb{T}D_h)$ sending $C_\mathbf{i}D_g$ to $C_\mathbf{i}D_h$ for any $\mathbf{i}\in\{\mathbf{a}: \mathbf{a}\in\mathbb{E}, \mathbf{a}\preceq\mathbf{1}^+\}$ is also an $\F$-algebra isomorphism. The desired corollary follows from Theorem \ref{T;Center}.
\end{proof}
We are now ready to close this section by presenting an example of Theorem \ref{T;Center}.
\begin{eg}\label{E;Example5.11}
Assume that $n=|\mathbb{U}_1|=2$ and $|\mathbb{U}_2|=3$. If $p\neq2$, Example \ref{E;Example3.22} implies that $\mathbb{T}$ has exactly
two block algebras. If $p\neq2$, Theorem \ref{T;Center} implies that the $\F$-dimensions of the centers of the only two distinct block algebras of $\mathbb{T}$ are equal to one. For the remaining case $p=2$, Example \ref{E;Example3.22} and Theorem \ref{T;Center} also imply that the $\F$-dimension of the center of the unique block algebra of $\mathbb{T}$ is equal to two. If $p=2$, Example \ref{E;Example3.22} and Theorem \ref{T;Center} imply that the nilpotent index of the Jacobson radical of the center of the unique block algebra of $\mathbb{T}$ is equal to two.
\end{eg}
\section{Jacobson radical of block algebra of $\mathbb{T}$}
In this section, for any block algebra of $\mathbb{T}$, we get an $\F$-basis and the nilpotent index of the Jacobson radical of this block algebra of $\mathbb{T}$. For any block algebra of $\mathbb{T}$, we also get the algebraic structure of the quotient $\F$-algebra of this block algebra of $\mathbb{T}$ with respect to its Jacobson radical. In particular, we get the $\F$-dimension of the Jacobson radical of an arbitrary chosen block algebra of $\mathbb{T}$. For our purpose, we display a sequence of lemmas to successively study these themes of this section. By Theorem \ref{T;Idempotent}, recall that $\mathbb{T}D_1, \mathbb{T}D_2, \ldots, \mathbb{T}D_{2^{n_-}}$ are exactly all block algebras of $\mathbb{T}$.
\begin{lem}\label{L;Lemma6.1}
Assume that $g\in[1, 2^{n_-}]$. Then the Jacobson radical $\mathrm{Rad}(\mathbb{T}D_g)$ of the block algebra $\mathbb{T}D_g$ of $\mathbb{T}$ satisfies the equation $\mathrm{Rad}(\mathbb{T}D_g)=\mathrm{Rad}(\mathbb{T})D_g$. Moreover, assume that $\mathbf{h}, \mathbf{i}, \mathbf{j}\!\in\!\mathbb{E}$ and $\mathbf{i}\!\!\preceq\!(\mathbf{h}\cap\mathbf{j})^+$. Then $(\mathbf{h}, (\mathbf{h}\triangle\mathbf{j})\cup\mathbf{i}, \mathbf{j})\!\in\!\mathbb{P}$ and $\mathrm{Rad}(\mathbb{T}D_g)$ has an $\F$-basis $\{\!B_{\mathbf{a}, (\mathbf{a}\triangle\mathbf{c})\cup\mathbf{b}, \mathbf{c}}\!D_g\!\!:\! \mathbf{a}, \mathbf{b},\mathbf{c}\!\in\!\mathbb{E}, \mathbf{b}\!\preceq\!(\mathbf{a}\cap\mathbf{c})^+, \exists\ (\mathbf{c}, \mathbf{d}, \mathbf{c})\!\!\in\!\!\mathbb{C}_g, n_{\mathbf{c}, \mathbf{a}\cup\mathbf{d}, \mathbf{c}}\!\!=\!\!0, p\!\mid\! k_{(\mathbf{a}\triangle\mathbf{c})\cup\mathbf{b}}\}$.
\end{lem}
\begin{proof}
The first statement follows from Theorem \ref{T;Idempotent} and Lemma \ref{L;Lemma2.6}. By combining the first statement, Lemmas \ref{L;Lemma4.8}, \ref{L;Lemma2.9}, notice that $\mathrm{Rad}(\mathbb{T}D_g)$ contains all matrices in $\{\!B_{\mathbf{a}, (\mathbf{a}\triangle\mathbf{c})\cup\mathbf{b}, \mathbf{c}}D_g\!:\! \mathbf{a}, \mathbf{b},\mathbf{c}\!\in\!\mathbb{E}, \mathbf{b}\!\preceq\!(\mathbf{a}\cap\mathbf{c})^+, \exists\ (\mathbf{c}, \mathbf{d}, \mathbf{c})\!\in\!\mathbb{C}_g, n_{\mathbf{c}, \mathbf{a}\cup\mathbf{d}, \mathbf{c}}\!\!=\!\!0, p\!\mid\! k_{(\mathbf{a}\triangle\mathbf{c})\cup\mathbf{b}}\}$. For any $\mathbf{k}, \mathbf{l}, \mathbf{m}, \mathbf{q}\in\mathbb{E}$ and $\mathbbm{m}\setminus\mathbbm{k}\subseteq\mathbbm{l}$, notice that $n_{\mathbf{l}, \mathbf{q}, \mathbf{m}}=0$ only if $n_{\mathbf{m}, \mathbf{k}\cup\mathbf{q}, \mathbf{m}}=0$. Therefore the combination of the first statement, Lemmas \ref{L;Lemma2.9}, \ref{L;Lemma3.6}, \ref{L;Lemma3.8}, \ref{L;Lemma4.6}, \ref{L;Lemma2.5} implies that each matrix in $\mathrm{Rad}(\mathbb{T}D_g)$ is an $\F$-linear combination of the matrices in $\{B_{\mathbf{a}, (\mathbf{a}\triangle\mathbf{c})\cup\mathbf{b}, \mathbf{c}}D_g\!:\! \mathbf{a}, \mathbf{b},\mathbf{c}\!\in\!\mathbb{E}, \mathbf{b}\!\preceq\!(\mathbf{a}\cap\mathbf{c})^+, \exists\ (\mathbf{c}, \mathbf{d}, \mathbf{c})\!\in\!\mathbb{C}_g, n_{\mathbf{c}, \mathbf{a}\cup\mathbf{d}, \mathbf{c}}\!=\!0, p\!\mid\! k_{(\mathbf{a}\triangle\mathbf{c})\cup\mathbf{b}}\}$. The desired lemma follows from the above proved statements and Lemma \ref{L;Lemma4.8}.
\end{proof}
\begin{lem}\label{L;Lemma6.2}
Assume that $g\in[1,2^{n_-}]$, $\mathbf{h}, \mathbf{i}, \mathbf{j}\in\mathbb{E}$, $\mathbf{i}\preceq(\mathbf{h}\cap\mathbf{j})^+$, $[1,n]^-\subseteq\mathbbm{j}$. Then $(\mathbf{h},(\mathbf{h}\triangle\mathbf{j})\cup\mathbf{i},\mathbf{j})\in\mathbb{P}$ and there is a unique $\mathbf{k}\in\mathbb{E}$ such that $(\mathbf{j},\mathbf{k},\mathbf{j})\in\mathbb{C}_g$. Moreover, assume that $n_{\mathbf{j},\mathbf{h}\cup\mathbf{k},\mathbf{j}}=0$ and $p\mid k_{(\mathbf{h}\triangle\mathbf{j})\cup\mathbf{i}}$. Then $B_{\mathbf{h},(\mathbf{h}\triangle\mathbf{j})\cup\mathbf{i}, \mathbf{j}}D_g$ is a nonzero matrix in the Jacobson radical $\mathrm{Rad}(\mathbb{T}D_g)$ of the block algebra $\mathbb{T}D_g$ of $\mathbb{T}$. Furthermore, the following are equivalent: $\mathrm{Rad}(\mathbb{T}D_g)=\{O\}$; $n_+=0$; and $\mathfrak{S}$ is a $p'$-valenced scheme.
\end{lem}
\begin{proof}
The first statement follows from combining Lemmas \ref{L;Lemma2.3}, \ref{L;Lemma2.5}, \ref{L;Lemma3.5}. The second statement follows from the first statement and Lemma \ref{L;Lemma6.1}. By the first statement, there is a unique $\mathbf{l}\in\mathbb{E}$ such that $(\mathbf{1}, \mathbf{l},\mathbf{1})\in\mathbb{C}_g$. If $n_+>0$, Lemma \ref{L;Lemma2.5} and the second statement imply that  $n_{\mathbf{1}, \mathbf{1}, \mathbf{1}}=0$, $p\mid k_{\mathbf{1}^+}$, and $B_{\mathbf{1},\mathbf{1}^+,\mathbf{1}}D_g\in\mathrm{Rad}(\mathbb{T}D_g)\setminus\{O\}$. So $\mathrm{Rad}(\mathbb{T}D_g)=\{O\}$ only if $n_+=0$. As Lemma \ref{L;Lemma2.9} implies that $n_+=0$ if and only if $\mathfrak{S}$ is a $p'$-valenced scheme, the desired lemma follows from Lemma \ref{L;Lemma6.1}.
\end{proof}
\begin{lem}\label{L;Lemma6.3}
Assume that $g\in[1,2^{n_-}]$. Then the nilpotent index of the Jacobson radical $\mathrm{Rad}(\mathbb{T}D_g)$ of the block algebra $\mathbb{T}D_g$ of $\mathbb{T}$ is equal to $2n_++1$.
\end{lem}
\begin{proof}
By Lemma \ref{L;Lemma6.2}, there is no loss to assume further that $n_+>0$. Assume that $[1,n]^+=\{h_1, h_2, \ldots, h_{n_+}\}$. For any $i\in[1, n_+]$, let $\mathbf{j}(i)$ be the $n$-tuple in $\mathbb{E}$ whose unique nonzero entry is the $h_i$th-entry. For any $i\in[1, n_+]$, Lemmas \ref{L;Lemma6.2} and \ref{L;Lemma2.5} imply that $(\mathbf{1}, \mathbf{k}, \mathbf{1}), (\mathbf{1}\setminus\mathbf{j}(i), \mathbf{l}(i),\mathbf{1}\setminus\mathbf{j}(i))\in\mathbb{C}_g$, $n_{\mathbf{1}\setminus\mathbf{j}(i),\mathbf{1},\mathbf{1}\setminus\mathbf{j}(i)}=n_{\mathbf{1}, \mathbf{1}\setminus\mathbf{j}(i),\mathbf{1}}=0$, $p\mid k_{\mathbf{j}(i)}$, and $B_{\mathbf{1}, \mathbf{j}(i), \mathbf{1}\setminus\mathbf{j}(i)}D_g, B_{\mathbf{1}\setminus\mathbf{j}(i), \mathbf{j}(i), \mathbf{1}}D_g\in\mathrm{Rad}(\mathbb{T}D_g)\setminus\{O\}$. The combination of Theorem \ref{T;Idempotent}, Lemmas \ref{L;Lemma2.8}, \ref{L;Lemma2.5}, \ref{L;Lemma2.3}, \ref{L;Lemma4.3} implies that $B_{\mathbf{1}, \mathbf{j}(i), \mathbf{1}\setminus\mathbf{j}(i)}D_gB_{\mathbf{1}\setminus\mathbf{j}(i), \mathbf{j}(i), \mathbf{1}}D_g=B_{\mathbf{1},\mathbf{j}(i),\mathbf{1}}D_g$ and
$B_{\mathbf{1},\mathbf{j}(1),\mathbf{1}}D_gB_{\mathbf{1},\mathbf{j}(2),\mathbf{1}}D_g\cdots B_{\mathbf{1},\mathbf{j}(n_+),\mathbf{1}}D_g\in\mathrm{Rad}(\mathbb{T}D_g)\setminus\{O\}$ for any $i\in[1,n_+]$. As $\mathrm{Rad}(\mathbb{T}D_g)\!\subseteq\!\mathrm{Rad}(\mathbb{T})$ by Lemma \ref{L;Lemma6.1}, the desired lemma follows from Lemma \ref{L;Lemma2.9}.
\end{proof}
\begin{lem}\label{L;Lemma6.4}
Assume that $g\in[1, 2^{n_-}]$, $h\in[1, 2^{n_+}]$, $(\mathbf{i}, \mathbf{j}, \mathbf{k}), (\mathbf{i}, \mathbf{l}, \mathbf{k})\in\mathbb{C}_{g, h}$. Then $\mathbf{j}=\mathbf{l}$ and there is a unique $\mathbf{m}\in\mathbb{E}$ such that $(\mathbf{k}, \mathbf{m}, \mathbf{k})\in\mathbb{C}_{g, h}$. Moreover, assume that $(\mathbf{q}, \mathbf{r}, \mathbf{s})\in\mathbb{P}$ and $p\nmid k_\mathbf{r}$. Then $(\mathbbm{i}\cap\mathbbm{k})^-\setminus\mathbbm{j}=(\mathbbm{q}\cap\mathbbm{s})^-\setminus\mathbbm{r}$ if and only if there is $t\in[1, 2^{n_+}]$ such that $(\mathbf{q}, \mathbf{r}, \mathbf{s})\in\mathbb{C}_{g,t}$. In particular, $D_{\mathbf{q},\mathbf{r},\mathbf{s}}D_g=\delta_{(\mathbbm{i}\cap\mathbbm{k})^-\setminus\mathbbm{j}, (\mathbbm{q}\cap\mathbbm{s})^-\setminus\mathbbm{r}}D_{\mathbf{q},\mathbf{r},\mathbf{s}}$.
\end{lem}
\begin{proof}
The first statement follows from Lemma \ref{L;Lemma2.16}. It is obvious to notice that
$(\mathbbm{q}\cap\mathbbm{s})^-\setminus\mathbbm{r}=(\mathbbm{q}\cap\mathbbm{s})^-\setminus(\mathbbm{q}\cap\mathbbm{r}\cap\mathbbm{s})$ and $(\mathbbm{q}\cap\mathbbm{s})^\circ\setminus\mathbbm{r}=(\mathbbm{q}\cap\mathbbm{s})^\circ\setminus(\mathbbm{q}\cap\mathbbm{r}\cap\mathbbm{s})$. Notice that  $\mathbbm{q}\cap\mathbbm{r}\cap\mathbbm{s}\subseteq(\mathbbm{q}\cap\mathbbm{s})^\circ$ and $p\nmid k_{\mathbf{q}\cap\mathbf{r}\cap\mathbf{s}}$ by Lemma \ref{L;Lemma2.5}. The equation  $(\mathbbm{i}\cap\mathbbm{k})^-\setminus\mathbbm{j}=(\mathbbm{q}\cap\mathbbm{s})^-\setminus\mathbbm{r}$ thus implies that $(\mathbf{q}\cap\mathbf{s}, \mathbf{q}\cap\mathbf{r}\cap\mathbf{s}, \mathbf{q}\cap\mathbf{s})\in\mathbb{C}_g$ and $(\mathbf{q}, \mathbf{r}, \mathbf{s})\in\mathbb{C}_{g,t}$ for some $t\in[1, 2^{n_+}]$. The second statement follows as the proof of the remaining direction is trivial. If $(\mathbbm{i}\cap\mathbbm{k})^-\setminus\mathbbm{j}\neq(\mathbbm{q}\cap\mathbbm{s})^-\setminus\mathbbm{r}$, then $D_{\mathbf{q},\mathbf{r},\mathbf{s}}D_g=O$ by combining the second statement, Equation \eqref{Eq;7}, Lemma \ref{L;Lemma2.14}. Assume that $(\mathbbm{i}\cap\mathbbm{k})^-\setminus\mathbbm{j}=(\mathbbm{q}\cap\mathbbm{s})^-\setminus\mathbbm{r}$. The second statement and the first statement imply that $(\mathbf{s}, \mathbf{u}, \mathbf{s})\in\mathbb{C}_g$ and $(\mathbbm{q}\cap\mathbbm{s})^\circ\setminus\mathbbm{r}=\mathbbm{s}^\circ\setminus\mathbbm{u}$. So $D_{\mathbf{q}, \mathbf{r}, \mathbf{s}}D_g=D_{\mathbf{q}, \mathbf{r}, \mathbf{s}}E_\mathbf{s}^*D_g=D_{\mathbf{q}, \mathbf{r}, \mathbf{s}}D_{\mathbf{s}, \mathbf{u}, \mathbf{s}}=D_{\mathbf{q}, \mathbf{r}, \mathbf{s}}$ by combining Equations \eqref{Eq;6}, \eqref{Eq;4}, \eqref{Eq;2}, Lemmas \ref{L;Lemma3.5}, \ref{L;Lemma2.14}, \ref{L;Lemma2.15}, the first statement. The desired lemma follows.
\end{proof}
\begin{lem}\label{L;Lemma6.5}
Assume that $g\in[1, 2^{n_-}]$. Then the quotient $\F$-algebra of the block algebra $\mathbb{T}D_g$ of $\mathbb{T}$ with respect to its Jacobson radical $\mathrm{Rad}(\mathbb{T}D_g)$ has an $\F$-basis $\{D_{\mathbf{a},\mathbf{b},\mathbf{c}}+\mathrm{Rad}(\mathbb{T}D_g): \exists\ d\in[1, 2^{n_+}], (\mathbf{a}, \mathbf{b}, \mathbf{c})\in\mathbb{C}_{g,d}\}$. In particular, $\mathbb{T}D_g/\mathrm{Rad}(\mathbb{T}D_g)$ is also equal to the following direct sum of the $\F$-linear subspaces of $ \mathbb{T}D_g/\mathrm{Rad}(\mathbb{T}D_g)$
$$\bigoplus_{h=1}^{2^{n_+}}\langle\{D_{\mathbf{a},\mathbf{b},\mathbf{c}}+\mathrm{Rad}(\mathbb{T}D_g): (\mathbf{a}, \mathbf{b}, \mathbf{c})\in\mathbb{C}_{g,h}\}\rangle_{\mathbb{T}D_g/\mathrm{Rad}(\mathbb{T}D_g)}.$$
\end{lem}
\begin{proof}
Let $\mathbb{U}=\{D_{\mathbf{a},\mathbf{b},\mathbf{c}}+\mathrm{Rad}(\mathbb{T}D_g): \exists\ d\in[1, 2^{n_+}], (\mathbf{a}, \mathbf{b}, \mathbf{c})\in\mathbb{C}_{g,d}\}$ and notice that $\mathbb{U}\!\subseteq\!\mathbb{T}D_g/\mathrm{Rad}(\mathbb{T}D_g)$ by Lemma \ref{L;Lemma6.4}. Let $\mathbb{V}\!=\!\{D_{\mathbf{a},\mathbf{b},\mathbf{c}}:\exists\ d\in[1, 2^{n_+}], (\mathbf{a}, \mathbf{b}, \mathbf{c})\!\in\!\mathbb{C}_{g,d}\}$. By combining Lemmas \ref{L;Lemma6.1}, \ref{L;Lemma2.9}, \ref{L;Lemma2.14}, notice that each nonzero $\F$-linear combination of the matrices in $\mathbb{V}$ is not contained in $\mathrm{Rad}(\mathbb{T}D_g)$. Therefore $\mathbb{U}$ is an $\F$-linearly independent subset of $\mathbb{T}D_g/\mathrm{Rad}(\mathbb{T}D_g)$. For the remaining part of the first statement, let $\mathbb{W}\!\!=\!\!\{\!B_{\mathbf{a}, (\mathbf{a}\triangle\mathbf{c})\cup\mathbf{b}, \mathbf{c}}\!D_g\!\!:\! \mathbf{a}, \mathbf{b},\mathbf{c}\!\in\!\mathbb{E}, \mathbf{b}\!\preceq\!(\mathbf{a}\cap\mathbf{c})^+, \exists\ (\mathbf{c}, \mathbf{d}, \mathbf{c})\!\in\!\mathbb{C}_g, n_{\mathbf{c}, \mathbf{a}\cup\mathbf{d}, \mathbf{c}}\!\!=\!\!0, p\!\mid\! k_{(\mathbf{a}\triangle\mathbf{c})\cup\mathbf{b}}\}$. By combining Lemmas \ref{L;Lemma4.8}, \ref{L;Lemma2.14}, \ref{L;Lemma6.4}, notice
that each matrix in $\mathbb{T}D_g$ is an $\F$-linear combination of the matrices in $\mathbb{V}\cup\mathbb{W}$. Lemma \ref{L;Lemma6.1} thus implies that each element in $\mathbb{T}D_g/\mathrm{Rad}(\mathbb{T}D_g)$ is an $\F$-linear combination of the elements in $\mathbb{U}$. The above discussion thus implies that $\mathbb{U}$ is an $\F$-basis of $\mathbb{T}D_g/\mathrm{Rad}(\mathbb{T}D_g)$. The first statement follows. The desired lemma follows from an application of the first statement.
\end{proof}
\begin{lem}\label{L;Lemma6.6}
Assume that $g\!\in\![1, 2^{n_-}]$ and $h\!\in\![1, 2^{n_+}]$. Then the quotient $\F$-algebra of the block algebra $\mathbb{T}D_g$ of $\mathbb{T}$ with respect to its Jacobson radical $\mathrm{Rad}(\mathbb{T}D_g)$ has a two-sided ideal $\langle\{D_{\mathbf{a}, \mathbf{b}, \mathbf{c}}+\mathrm{Rad}(\mathbb{T}D_g):(\mathbf{a},\mathbf{b},\mathbf{c})\in\mathbb{C}_{g, h}\}\rangle_{\mathbb{T}D_g/\mathrm{Rad}(\mathbb{T}D_g)}$. Moreover, assume that $(\mathbf{i}, \mathbf{j}, \mathbf{i})\in\mathbb{C}_g$. Then $n_{\mathbf{i},\mathbf{j}, \mathbf{i}}$ is independent of the choice of $(\mathbf{i},\mathbf{j}, \mathbf{i})$ and the $\F$-dimension of $\langle\{\!D_{\mathbf{a}, \mathbf{b}, \mathbf{c}}\!+\!\mathrm{Rad}(\mathbb{T}D_g)\!:\! (\mathbf{a},\mathbf{b},\mathbf{c})\!\!\in\!\!\mathbb{C}_{g, h}\!\}\rangle_{\mathbb{T}D_g/\mathrm{Rad}(\mathbb{T}D_g)}$ is equal to $4^{n\!-\!n_+\!-\!n_{\mathbf{i},\mathbf{j},\mathbf{i}}}$.
\end{lem}
\begin{proof}
The first statement follows from combining Lemmas \ref{L;Lemma6.5}, \ref{L;Lemma2.14}, \ref{L;Lemma2.15}. It is obvious to see that
$n_{\mathbf{i},\mathbf{j}, \mathbf{i}}$ is independent of the choice of $(\mathbf{i},\mathbf{j}, \mathbf{i})$. For the other part of the remaining statement, recall that $|\mathbb{C}_g(h)|=2^{n-n_+-n_{\mathbf{i}, \mathbf{j}, \mathbf{i}}}$. Therefore Lemma \ref{L;Lemma2.16} implies that $|\mathbb{C}_{g, h}|=4^{n-n_+-n_{\mathbf{i},\mathbf{j},\mathbf{i}}}$. The desired lemma follows from Lemma \ref{L;Lemma6.5}.
\end{proof}
\begin{lem}\label{L;Lemma6.7}
Assume that $g\!\in\![1, 2^{n_-}]$ and $h\!\in\![1, 2^{n_+}]$. Then the quotient $\F$-algebra of the block algebra $\mathbb{T}D_g$ of $\mathbb{T}$ with respect to its Jacobson radical $\mathrm{Rad}(\mathbb{T}D_g)$ has a two-sided ideal $\langle\{D_{\mathbf{a}, \mathbf{b}, \mathbf{c}}+\mathrm{Rad}(\mathbb{T}D_g):(\mathbf{a},\mathbf{b},\mathbf{c})\in\mathbb{C}_{g, h}\}\rangle_{\mathbb{T}D_g/\mathrm{Rad}(\mathbb{T}D_g)}$. Moreover, assume that $(\mathbf{i}, \mathbf{j}, \mathbf{i})\in\mathbb{C}_g$. Then $n_{\mathbf{i},\mathbf{j}, \mathbf{i}}$ is independent of the choice of $(\mathbf{i},\mathbf{j}, \mathbf{i})$ and
$\langle\{D_{\mathbf{a}, \mathbf{b}, \mathbf{c}}+\mathrm{Rad}(\mathbb{T}D_g):(\mathbf{a},\mathbf{b},\mathbf{c})\in\mathbb{C}_{g, h}\}\rangle_{\mathbb{T}D_g/\mathrm{Rad}(\mathbb{T}D_g)}\cong\mathrm{M}_{2^{n-n_+-n_{\mathbf{i}, \mathbf{j}, \mathbf{i}}}}(\F)$ as $\F$-algebras.
\end{lem}
\begin{proof}
By Lemma \ref{L;Lemma6.6}, it suffices to check the final formula. For any $\mathbf{k}, \mathbf{l}\in\mathbb{C}_g(h)$, recall that $M_{\mathbf{k},\mathbf{l}}(\mathbb{C}_g(h))$ is the $\{\overline{0}, \overline{1}\}$-matrix in $\mathrm{M}_{\mathbb{C}_g(h)}(\F)$ whose unique nonzero entry is the $(\mathbf{k}, \mathbf{l})$-entry. In particular, $M_{\mathbf{k}, \mathbf{l}}(\mathbb{C}_g(h))M_{\mathbf{m}, \mathbf{q}}(\mathbb{C}_g(h))=\delta_{\mathbf{l}, \mathbf{m}}M_{\mathbf{k},\mathbf{q}}(\mathbb{C}_g(h))$ for any $\mathbf{k}, \mathbf{l}, \mathbf{m}, \mathbf{q}\in\mathbb{C}_g(h)$. For any $(\mathbf{k}, \mathbf{l}, \mathbf{m}), (\mathbf{q}, \mathbf{r}, \mathbf{s})\in\mathbb{C}_{g, h}$, the combination of Lemmas \ref{L;Lemma2.14}, \ref{L;Lemma2.15}, \ref{L;Lemma6.4} implies that there is a unique $(\mathbf{k}, \mathbf{t}, \mathbf{s})\in\mathbb{C}_{g, h}$ satisfying the equation
$(D_{\mathbf{k},\mathbf{l},\mathbf{m}}+\mathrm{Rad}(\mathbb{T}D_g))(D_{\mathbf{q},\mathbf{r},\mathbf{s}}+\mathrm{Rad}(\mathbb{T}D_g))\!=\!
\delta_{\mathbf{m},\mathbf{q}}(D_{\mathbf{k},\mathbf{t},\mathbf{s}}+\mathrm{Rad}(\mathbb{T}D_g))$. By Lemmas \ref{L;Lemma2.16} and \ref{L;Lemma6.5}, the $\F$-linear map from $\langle\{D_{\mathbf{a}, \mathbf{b}, \mathbf{c}}+\mathrm{Rad}(\mathbb{T}D_g):(\mathbf{a},\mathbf{b},\mathbf{c})\in\mathbb{C}_{g, h}\}\rangle_{\mathbb{T}D_g/\mathrm{Rad}(\mathbb{T}D_g)}$ to $\mathrm{M}_{\mathbb{C}_g(h)}(\F)$ sending
$D_{\mathbf{k},\mathbf{l},\mathbf{m}}+\mathrm{Rad}(\mathbb{T}D_g)$ to $M_{\mathbf{k},\mathbf{m}}(\mathbb{C}_g(h))$ for any $(\mathbf{k},\mathbf{l},\mathbf{m})\in\mathbb{C}_{g, h}$ is thus an $\F$-algebra isomorphism. Recall that $|\mathbb{C}_g(h)|=2^{n-n_+-n_{\mathbf{i},\mathbf{j},\mathbf{i}}}$. So this equation implies that $\mathrm{M}_{\mathbb{C}_g(h)}\!(\F)\!\cong\!\mathrm{M}_{2^{n\!-\!n_+\!-\!n_{\mathbf{i}, \mathbf{j}, \mathbf{i}}}}\!(\F)$ as $\F$-algebras. The desired lemma follows.
\end{proof}
We are now ready to display the main result of this section and another corollary.
\begin{thm}\label{T;Jacobson}
Assume that $g\!\in\![1, 2^{n_-}]$, $\mathbf{h}, \mathbf{i}, \mathbf{j}\!\in\!\mathbb{E}$, $\mathbf{i}\!\!\preceq\!\!(\mathbf{h}\cap\mathbf{j})^+$. Then $(\!\mathbf{h},(\mathbf{h}\triangle\mathbf{j})\!\cup\!\mathbf{i}, \mathbf{j})\!\!\in\!\!\mathbb{P}$ and the Jacobson radical $\mathrm{Rad}(\mathbb{T}D_g)$ of the block algebra $\mathbb{T}D_g$ of $\mathbb{T}$ has an $\F$-basis $\{B_{\mathbf{a}, (\mathbf{a}\triangle\mathbf{c})\cup\mathbf{b}, \mathbf{c}}D_g\!:\! \mathbf{a}, \mathbf{b},\mathbf{c}\!\in\!\mathbb{E}, \mathbf{b}\!\preceq\!(\mathbf{a}\cap\mathbf{c})^+, \exists\ (\mathbf{c}, \mathbf{d}, \mathbf{c})\!\in\!\mathbb{C}_g, n_{\mathbf{c}, \mathbf{a}\cup\mathbf{d}, \mathbf{c}}\!=\!0, p\!\mid\! k_{(\mathbf{a}\triangle\mathbf{c})\cup\mathbf{b}}\}$. Moreover, the nilpotent index of $\mathrm{Rad}(\mathbb{T}D_g)$ is equal to $2n_++1$. On the other hand, assume that $(\mathbf{k}, \mathbf{l}, \mathbf{k})\in\mathbb{C}_g$. Then $n_{\mathbf{k}, \mathbf{l},\mathbf{k}}$ is independent of the choice of $(\mathbf{k}, \mathbf{l}, \mathbf{k})$ and there is an $\F$-algebra isomorphism from $\mathbb{T}D_g/\mathrm{Rad}(\mathbb{T}D_g)$ to $2^{n_+}\mathrm{M}_{2^{n-n_+-n_{\mathbf{k}, \mathbf{l}, \mathbf{k}}}}(\F)$.
In particular, the $\F$-dimension of $\mathrm{Rad}(\mathbb{T}D_g)$ is equal to $4^{n-n_+-n_{\mathbf{k},\mathbf{l},\mathbf{k}}}5^{n_+}-2^{2n-n_+-2n_{\mathbf{k},\mathbf{l},\mathbf{k}}}$.
\end{thm}
\begin{proof}
By Lemmas \ref{L;Lemma6.1} and \ref{L;Lemma6.3}, it suffices to check the third statement and the fourth statement. The third statement thus follows from Lemmas \ref{L;Lemma6.5} and \ref{L;Lemma6.7}. The desired theorem follows from an application of Theorem \ref{T;Dimension} and the third statement.
\end{proof}
\begin{cor}\label{C;Corollary6.9}
Assume that $g\in[1, 2^{n_-}]$. Then the nilpotent index of the Jacobson radical $\mathrm{Rad}(\mathbb{T}D_g)$ of the block algebra $\mathbb{T}D_g$ of $\mathbb{T}$ is independent of the choice of $\mathbf{x}$. Moreover, the $\F$-dimension of $\mathrm{Rad}(\mathbb{T}D_g)$ is also independent of the choice of $\mathbf{x}$. Furthermore, the following are equivalent: $\mathbb{T}D_g$ is a semisimple $\F$-subalgebra of $\mathbb{T}$; $\mathbb{T}\!D_1\!, \mathbb{T}\!D_2,\ldots, \!\mathbb{T}\!D_{2^{n_-}}$ are semisimple $\F$\!-\!subalgebras of $\mathbb{T}$; and $\mathfrak{S}$ is a $p'$\!-\!valenced scheme.
\end{cor}
\begin{proof}
The first statement and the second statement follow from Theorem \ref{T;Jacobson}. The desired corollary follows from the combination of Theorem \ref{T;Jacobson}, Lemmas \ref{L;Lemma6.2}, \ref{L;Lemma2.6}.
\end{proof}
We are now ready to finish the whole paper by giving an example of Theorem \ref{T;Jacobson}.
\begin{eg}\label{E;Example6.10}
Assume that $n=|\mathbb{U}_1|=2$ and $|\mathbb{U}_2|=3$. If $p\neq2$, Example \ref{E;Example3.22} implies that $\mathbb{T}$ has exactly two block algebras. If $p\neq2$, Corollary \ref{C;Corollary6.9} implies that the only two distinct block algebras of $\mathbb{T}$ are semisimple $\F$-subalgebras of $\mathbb{T}$. For the remaining case $p=2$, Example \ref{E;Example3.22} and Theorem \ref{T;Jacobson} imply that the nilpotent index of the Jacobson radical of the unique block algebra of $\mathbb{T}$ is equal to three. For the case $p=2$, Example \ref{E;Example3.22} and Theorem \ref{T;Jacobson} imply that the Jacobson radical of the unique block algebra of $\mathbb{T}$ is a twelve-dimensional nilpotent two-sided ideal of $\mathbb{T}$.
\end{eg}
\subsection*{Disclosure statement} No relevant financial or nonfinancial interests are reported.
\subsection*{Data availability statement}All used data are contained in this submitted paper.
\subsection*{Ethical approval statement}The ethical approval of this paper is not applicable.


\end{document}